\documentclass[a4paper,oneside,english,10pt]{elsarticle}
\usepackage[T1]{fontenc}
\usepackage[latin9]{inputenc}
\usepackage[dvipsnames]{xcolor}
\usepackage{amsthm}
\usepackage{mathtools}
\usepackage{amsmath}
\usepackage{amstext}
\usepackage{amssymb}
\usepackage{amsfonts}
\usepackage{physics}
\usepackage{caption}
\usepackage{float}
\usepackage{graphicx}
\usepackage{esint}
\usepackage{a4wide}
\usepackage{bbm}
\usepackage{mathrsfs}
\usepackage{upgreek}
\usepackage{setspace}
\usepackage{multirow}
\usepackage{placeins}
\usepackage{hyperref}
\usepackage{bm}
\usepackage{nicefrac}
\usepackage{enumitem}
\usepackage{stackengine}
\usepackage{scalerel}
\usepackage{bbding}
\usepackage{wasysym}
\usepackage{pifont}
\usepackage{booktabs}

\usepackage{fancyhdr}
\usepackage{amsbsy}
\usepackage{stmaryrd}
\usepackage{cases}
\usepackage{wasysym}
\usepackage{appendix}
\usepackage{sidecap}
\usepackage{booktabs}
\usepackage{graphicx}
\usepackage{float}
\usepackage[margin=1.2in]{geometry}
\usepackage{dsfont}
\usepackage{mathtools}
\usepackage{indentfirst}
\usepackage{frontespizio}
\usepackage{mathrsfs}
\usepackage{tikz}
\usepackage{listings}
\usepackage[dvipsnames]{xcolor}

\usepackage{enumitem}
\makeatletter
\newcommand{\mylabel}[2]{#2\def\@currentlabel{#2}\label{#1}}
\makeatother

\newtheorem{remark}{Remark}

\numberwithin{equation}{section}

\newcommand{\subscript}[2]{$#1 . #2$}

\newcommand{\A}{\mathcal{A}}

\renewcommand{\phi}{\varphi}
\newcommand{\K}{\mathcal{K}}
\newcommand{\B}{\mathcal{B}}

\newcommand{\D}{\mathcal{D}}

\newcommand{\0}{\scriptscriptstyle 0}

\DeclarePairedDelimiter{\norma}{\lVert}{\rVert}
\DeclarePairedDelimiter{\abs*}{\lvert}{\rvert}
\theoremstyle{definition}

\newtheorem{Definition}{Definition}[section]
\newtheorem{Theorem}[Definition]{Theorem}

\newtheorem{Lemma}[Definition]{Lemma}

\newcommand{\executeiffilenewer}[3]{%
	\ifnum\pdfstrcmp{\pdffilemoddate{#1}}%
	{\pdffilemoddate{#2}}>0%
	{\immediate\write18{#3}}\fi%
}

\usepackage{empheq,etoolbox}
\patchcmd{\subequations}
  {\theparentequation\alph{equation}}
  {\theparentequation.\alph{equation}}
  {}{}

\begin{document}


\title{On the coupling of the Curved Virtual Element Method with the one-equation Boundary Element Method for 2D exterior Helmholtz problems}
\author[UNIPR]{L. Desiderio}
\ead{luca.desiderio@unipr.it}
\author[POLITO]{S. Falletta\corref{cor1}}
\ead{silvia.falletta@polito.it}
\author[POLITO]{M. Ferrari}
\ead{matteo.ferrari@polito.it}
\author[POLITO]{L. Scuderi}
\ead{letizia.scuderi@polito.it}
\address[UNIPR]{Dipartimento di Scienze Matematiche, Fisiche e Informatiche, Universit\`{a} di Parma,\\ Parco Area delle Scienze, 53/A, 43124, Parma, Italia}
\address[POLITO]{Dipartimento di Scienze Matematiche ``G.L. Lagrange'', Politecnico di Torino,\\ Corso Duca degli Abruzzi, 24, 10129, Torino, Italia}
\cortext[cor1]{Corresponding author.}
\begin{abstract} 
We consider the Helmholtz equation defined in unbounded domains, external to 2D bounded ones, endowed with a Dirichlet condition on the boundary and the Sommerfeld radiation condition at infinity. To solve it, we reduce the infinite region, in which the solution is defined, to a bounded computational one, delimited by a curved smooth artificial boundary and we impose on this latter a non reflecting condition of boundary integral type. Then, we apply the curved virtual element method in the finite computational domain, combined with the one-equation boundary element method on the artificial boundary.
We present the theoretical analysis of the proposed approach and we provide an optimal convergence error estimate in the energy norm. The numerical tests confirm the theoretical results and show the effectiveness of the new proposed approach.
\end{abstract}
\maketitle
\section{Introduction}\label{sec_1_introduction}
Frequency-domain wave propagation problems defined in unbounded regions, external to bounded obstacles, turn out to be a difficult physical and numerical task due to the issue of  determining the solution in an infinite domain. One of the typical techniques to solve such problems is the Boundary Integral Equation (BIE) method, which allows to reduce by one the dimension of the problem, requiring only the discretization of the obstacle boundary. Once the boundary distribution is retrieved by means of a Boundary Element Method (BEM) \cite{SauterSchwab2011}, the solution of the original problem at each point of the exterior domain is obtained by computing a boundary integral. However, this procedure may result not efficient, especially when the solution has to be evaluated at many points of the infinite domain.

During the last decades much effort has been concentrated on developing alternative approaches. Among these we mention those based on the coupling of domain methods, such as Finite Difference Method (FDM), Finite Element Method (FEM) and the recent Virtual Element Method (VEM), with the BEM. These are obtained by reducing the unbounded domain to a bounded computational one, delimited by an artificial boundary, on which a suitable Boundary Integral-Non Reflecting Boundary Condition (BI-NRBC) is imposed. This latter guarantees that the artificial boundary is transparent and that no spurious reflections arise from the resolution of the original problem by means of the interior domain method applied in the finite computational domain.\\
\indent
The most popular approaches for such a coupling, associated to the use of the FEM in the interior domain, involve the Green representation of the solution and are often referred to as the Johnson \& N\'ed\'elec Coupling (JNC) \cite{JohnsonNedelec1980} or the Costabel \& Han Coupling (CHC) \cite{Costabel1987, Han1990}.
Since the JNC is based on a single BIE, involving both the single and the double layer integral operators associated with the fundamental solution, it is known as the \emph{one equation BEM-FEM coupling} and it gives rise to a non-symmetric final linear system. 

On the contrary, the CHC is based on a couple of BIEs, one of which involves the second order normal derivative of the fundamental solution (hence a hypersingular integral operator),
and it yields to a symmetric scheme. Despite the fact that an integration by parts strategy can be applied to weaken the hypersingularity, the approach turns out to be quite onerous from the computational point of view, especially in the case of frequency-domain wave problems for which the accuracy of the BEM is strictly connected to the frequency parameter and to the density of discretization points per wavelength. Even if the CHC has been applied in several contexts, among which we mention the recent paper \cite{GaticaMeddahi2020}, where the theoretical analysis has been derived for the solution of the Helmholtz problem by means of a VEM, from the engineering point of view the JNC is the most natural and appealing way to deal with unbounded domain problems (for very recent real-life applications see, for example, \cite{AIMI2021741} and \cite{DesiderioFallettaScuderi2021}).\\
\indent
In this paper we propose a new approach based on the JNC between the Galerkin BEM and the Curved Virtual Element Method (CVEM) in the interior of the computational domain.
This choice is based on the fact that the VEM allows to
broaden the classical family of the FEM for the discretization of partial differential
equations for what concerns both the decomposition of domain with complex geometry and the definition of local high order discrete spaces. In the standard VEM formulation the discrete spaces, built on meshes made of polygonal or polyhedral elements, are similar to the usual finite element spaces with the addition of suitable non-polynomial functions. The novelty of the VEM consists in defining discrete spaces and degrees of freedom in such a way that the elementary stiffness and mass matrices can be computed using only the degrees of freedom, without the need of explicitly knowing the non-polynomial functions (from which the ``virtual'' word descends), with a consequent easiness of implementation even for high approximation orders. Originally developed as a variational reformulation of the nodal Mimetic Finite Difference (MFD) method \cite{beirao_2011, brezzi_2009, manzini_2017}, the VEM has been applied to a wide variety of interior problems (among the most recent papers we refer the reader to \cite{antonietti_2016, benvenuti_2019, berrone_2018, certik_2020}). On the contrary, only few papers deal with VEM applied to exterior problems, among which the already mentioned \cite{GaticaMeddahi2020} and \cite{DesiderioFallettaScuderi2021}. In this latter the JNC between the collocation BEM and
the VEM has been numerically investigated for the approximation of the solution of Dirichlet boundary
value problems defined by the 2D Helmholtz equation.

\indent
The very satisfactory results we have obtained in \cite{DesiderioFallettaScuderi2021} have stimulated us to further investigate on the application of the VEM to the solution of exterior problems. For this reason, we propose here a novel approach in the CVEM-Galerkin context that we have studied both from the theoretical and the numerical point of view. In particular, the choice of the CVEM instead of the standard (polygonal) VEM relies on the fact that the use of curvilinear elements allows to avoid the sub-optimal rate of convergence for orders of accuracy higher than 2, when curvilinear obstacles are considered. Further, due to the arbitrariness of the choice of the artificial boundary, and dealing with curved virtual elements, we choose the latter of curvilinear type. It is worth mentioning that the choice of the CVEM space refers in particular to that proposed in \cite{BeiraoRussoVacca2019}. 
 For the discretization of the BI-NRBC, we consider a classical BEM associated to Lagrangian nodal basis functions. As already remarked, the main challenge in the theoretical analysis is the lack of ellipticity of the associated bilinear form. However, using the Fredholm theory for integral operators, it is possible to prove the well-posedness of the problem in case of computational domains with smooth artificial boundaries. Moreover, the analysis of the Helmholtz problem and of the proposed numerical method for its solution, is carried out by interpreting the new main operators as perturbations of the Laplace ones. We present the theoretical analysis of the method in a quite general framework and we provide an optimal error estimate in the energy norm. Since the analysis is based on the pioneering paper by Jonhson and Nédélec, the smoothness properties of the artificial boundary represent a key requirement. We remark that, for the classical Galerkin approach, the breakthrough in the theoretical analysis that validates the stability of the JNC also in case of non-smooth boundaries, was proved by Sayas in \cite{Sayas2009}. However, since we deal with a generalized Galerkin method, the same analysis can not be straightforwardly applied and needs further investigations.

\

\indent
The paper is organized as follows: in the next section we present the model problem for the Helmholtz equation and its reformulation
in a bounded region, by the introduction of the artificial boundary and the associated one equation BI-NRBC. In Section \ref{sec_3_weak_pb} we introduce the variational formulation of the problem restricted to the finite computational domain, recalling the corresponding main theoretical issues, among which existence and uniqueness of the solution.  In Section \ref{sec_4_galerkin_pb} we apply the Galerkin method providing an error estimate in the energy norm, for a quite generic class of approximation spaces. 
Then, in Section \ref{sec_5_discrete_scheme} we describe the choice of the CVEM-BEM approximation spaces and we prove the validity of the error analysis in this specific context. Additionally, we detail the algebraic formulation of the coupled global scheme. Finally, in Section \ref{sec_6_num_results} we present some numerical results highlighting the effectiveness of the proposed approach and the validation of the theoretical results. Furthermore, in the last example we show that the optimal convergence order of the scheme is guaranteed also when polygonal computational domains are considered. Finally, some conclusions are drawn in Section \ref{sec_7_conclusions}.

\section{The model problem}\label{sec_2_model_pb}
In a fixed Cartesian coordinates system $\mathbf{x}=\left(x_{1},x_{2}\right)$, we consider an open bounded domain $\Omega_{\tiny{0}}\subset\mathbf{R}^{2}$ with a Lipschitz boundary $\Gamma_{\tiny{0}}$ having positive Lebesgue measure.
We denote by $\Omega_{e}:=\mathbf{R}^{2}\setminus\overline{\Omega}_{\tiny{0}}$ the exterior unbounded domain (see Figure \ref{fig:rob_domain} (a))
and we consider the following frequency-domain wave propagation problem:
\begin{subequations}\label{dirichlet_problem}
  \begin{empheq}[left=\empheqlbrace]{align}
    \label{dirichlet_problem_1} &\Delta u_{e}(\mathbf{x})+\kappa^{2}u_{e}(\mathbf{x})= -f(\mathbf{x})&  &\mathbf{x}\in \Omega_{e},\\
    \label{dirichlet_problem_2} &u_{e}(\mathbf{x})=g(\mathbf{x})& 						 &\mathbf{x}\in \Gamma_{0},\\
    \label{dirichlet_problem_3} &\lim\limits_{\|\mathbf{x}\|\rightarrow\infty}\|\mathbf{x}\|^{\frac{1}{2}}\left(\nabla u_e(\mathbf{x})\cdot\frac{\mathbf{x}}{\|\mathbf{x}\|}-\imath\kappa u_e(\mathbf{x})\right)=0.
  \end{empheq}
\end{subequations}
In the above problem, Equation (\ref{dirichlet_problem_1}) is known as the Helmholtz equation, with source term $f\in L^{2}(\Omega_{e})$, Equation (\ref{dirichlet_problem_2}) represents a boundary condition of Dirichlet type with datum $g$, and Equation (\ref{dirichlet_problem_3}) is the Sommerfeld radiation condition, that ensures the appropriate behaviour of the complex-valued unknown function $u_{e}$ at infinity. Furthermore, $\nabla$ and $\Delta$ denote the nabla and Laplace operators, respectively, and $\imath$ stands for the imaginary unit.\\
We recall that the wave number $\kappa$ is often real and constant, and it is complex if the propagation medium is energy absorbing, or a function of the space if the medium is inhomogeneous. Here, we suppose that $\kappa$ is real, positive and constant.\\
In the sequel we assume that $g\in H^{\nicefrac{1}{2}}(\Gamma_0)$, to guarantee existence and uniqueness of the solution $u_e$ of Problem (\ref{dirichlet_problem}) in the Sobolev space  $H^{1}_{\text{loc}}(\Omega_{e})$ (see \cite{ColtonKress2019}). 

\

\noindent
As many practical situations require, we aim at determining the solution $u_{e}$ of Problem (\ref{dirichlet_problem}) in a bounded subregion of $\Omega_{e}$ surrounding $\Omega_{0}$. To this end, we introduce an artificial boundary $\Gamma$ which allows decomposing $\Omega_{e}$ into a finite computational domain $\Omega$, bounded internally by $\Gamma_{0}$ and externally by $\Gamma$, and an infinite residual one, denoted by $\Omega_{\infty}$, as depicted in Figure \ref{fig:rob_domain} (b). We choose $\Gamma$ such that $\text{supp}(f)$ is a bounded subset of $\Omega$.
We assume that $\Gamma_{0}$ and $\Gamma$ are made up of a finite number of curves of class $C^{m+1}$, with $m\geq0$, so that $\Omega$ is a domain with piece-wise smooth boundaries. Moreover, we assume that $\Gamma$ is a Lyapunov regular contour, i.e. the gradient of any local parametrization is Hölder continuous.
\begin{figure}[H]
	\centering
	\includegraphics[width=0.85\textwidth]{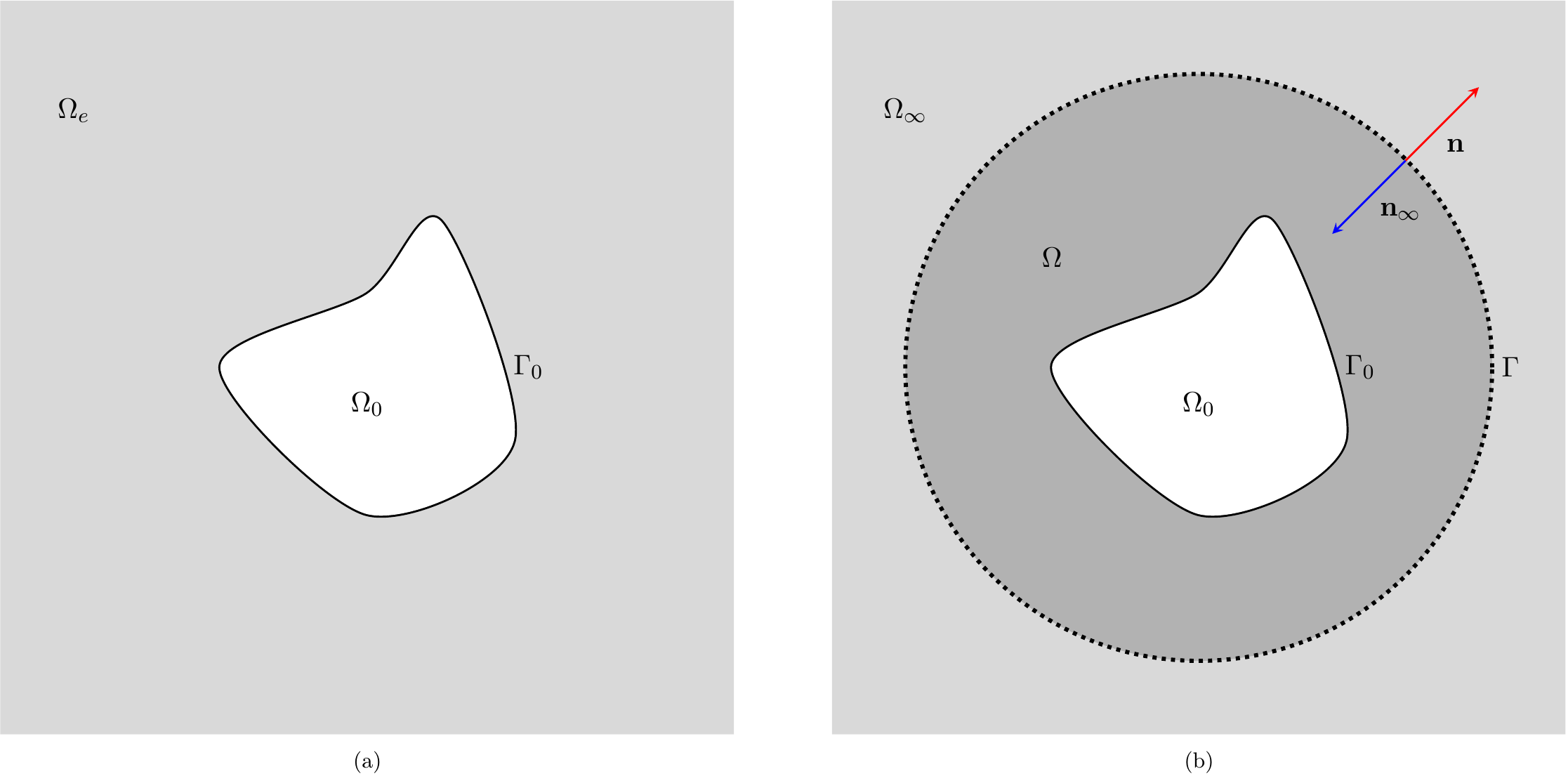}
	\caption{Model problem setting.}
	\label{fig:rob_domain}
\end{figure}
\noindent 
Denoting by $u$ and $u_{\infty}$ the restrictions of the solution $u_{e}$ to $\Omega$ and $\Omega_{\infty}$ respectively, and by $\mathbf{n}$ and $\mathbf{n}_{\infty}$ the unit normal vectors on $\Gamma$ pointing outside $\Omega$ and $\Omega_{\infty}$, we impose the following compatibility and equilibrium conditions on $\Gamma$ (recall that $\mathbf{n}_{\infty}=-\mathbf{n}$):
\begin{equation}
    \label{compatibility_condition}
    u(\mathbf{x}) =  u_{\infty}(\mathbf{x}), \qquad        \frac{\partial u}{\partial{\mathbf{n}}}(\mathbf{x})=-\frac{\partial u_{\infty}}{\partial{\mathbf{n}_{\infty}}}(\mathbf{x}),	\qquad \mathbf{x}\in \Gamma.
  \end{equation}
In the above relations and in the sequel we omit, for simplicity, the use of the trace operator to indicate the restriction of $H^1$ functions to the boundary $\Gamma$ from the exterior or interior.
\noindent
In order to obtain a well posed problem in $\Omega$, we need to impose a proper boundary condition on $\Gamma$. It is known that the solution $u_{\infty}$ in $\Omega_{\infty}$ can be represented by the following Kirchhoff formula:
\begin{equation}\label{boundary_integral_representation}
u_{\infty}(\mathbf{x})=\int_{\Gamma}G_{\kappa}(\mathbf{x},\mathbf{y})\frac{\partial u_{\infty}}{\partial\mathbf{n}_{\infty}}(\mathbf{y})\,\dd\Gamma_{\mathbf{y}}-\int_{\Gamma}\frac{\partial G_{\kappa}}{\partial\mathbf{n}_{\infty,\mathbf{y}}}(\mathbf{x},\mathbf{y})u_{\infty}(\mathbf{y})\,\dd\Gamma_{\mathbf{y}} \qquad \mathbf{x}\in\Omega_{\infty}\setminus\Gamma,
\end{equation}
in which $G_{\kappa}$ is the fundamental solution of the 2D Helmholtz problem and $\mathbf{n}_{\infty,\mathbf{y}}$ denotes the normal unit vector with initial point in $\mathbf{y} \in\Gamma$. The expression of $G_{\kappa}$ and of its normal derivative in \eqref{boundary_integral_representation} are given by
\begin{equation*}\label{kernels}
G_{\kappa}(\mathbf{x},\mathbf{y}):=\frac{\imath}{4}H_{0}^{(1)}(\kappa r) \quad \text{and} \quad \frac{\partial G_{\kappa}}{\partial\mathbf{n}_{\infty,\mathbf{y}}}(\mathbf{x},\mathbf{y})=\frac{\imath\kappa}{4}\frac{\mathbf{r}\cdot\mathbf{n}_{\infty,\mathbf{y}}}{r}H_{1}^{(1)}(\kappa r),
\end{equation*}
where $r=|\mathbf{r}|=|\mathbf{x}-\mathbf{y}|$ represents the distance between the source point $\mathbf{x}$ and the field point $\mathbf{y}$, and $H_{m}^{(1)}$ denotes the $m$-th order Hankel function of the first kind. 

\noindent
We introduce the single-layer integral operator $\text{V}_{\kappa} \colon H^{-\nicefrac{1}{2}}(\Gamma)\to H^{\nicefrac{1}{2}}(\Gamma)$
\begin{equation*}\label{single_layer_operator}
\text{V}_{\kappa}\psi(\mathbf{x}):=\int_{\Gamma}G_{\kappa}(\mathbf{x},\mathbf{y})\psi(\mathbf{y})\,\dd\Gamma_{\mathbf{y}}, \qquad \mathbf{x}\in\Gamma
\end{equation*}
and the double-layer integral operator $\text{K}_{\kappa}\colon H^{\nicefrac{1}{2}}(\Gamma)\to H^{\nicefrac{1}{2}}(\Gamma)$
\begin{equation*}\label{single_layer_operator}
\text{K}_{\kappa}\varphi(\mathbf{x}):=-\int_{\Gamma}\frac{\partial G_{\kappa}}{\partial\mathbf{n}_{\infty,\mathbf{y}}}(\mathbf{x},\mathbf{y})\varphi(\mathbf{y})\,\dd\Gamma_{\mathbf{y}}, \qquad \mathbf{x}\in\Gamma,
\end{equation*}
which are continuous for all $\kappa > 0$ (see \cite{HsiaoWendland2008}).
Then, the trace of \eqref{boundary_integral_representation} on $\Gamma$ reads (see \cite{ColtonKress2019}) 
\begin{equation}\label{boundary_integral_equation_operators}
\frac{1}{2}u_{\infty}(\mathbf{x})-\text{V}_{\kappa}\frac{\partial u_{\infty}}{\partial\mathbf{n}_{\infty}}({\mathbf{x}})-\text{K}_{\kappa}u_{\infty}(\mathbf{x})=0, \qquad \mathbf{x}\in\Gamma.
\end{equation}
\noindent
Equation \eqref{boundary_integral_equation_operators}, which expresses the natural relation that $u_{\infty}$ and its normal derivative have to satisfy at each point of the artificial boundary, is imposed on $\Gamma$ as an exact (non local) BI-NRBC to solve Problem \eqref{dirichlet_problem} in the finite computational domain.
Thus, taking into account the compatibility and equilibrium  conditions (\ref{compatibility_condition}),  and introducing the notation $\displaystyle \lambda:=\frac{\partial u}{\partial\mathbf{n}}$, the new problem defined in the domain of interest $\Omega$ takes the form:
\begin{subequations}\label{model_problem_full}
  \begin{empheq}[left=\empheqlbrace]{align}
    \label{model_problem_full_1} &\Delta u(\mathbf{x})+\kappa^{2}u(\mathbf{x})=-f(\mathbf{x})&  &\mathbf{x}\in \Omega\\
    \label{model_problem_full_2} &u(\mathbf{x})=g(\mathbf{x})& 						 &\mathbf{x}\in \Gamma_0\\
    \label{model_problem_full_3} &\frac{1}{2}u(\mathbf{x})+\text{V}_{\kappa}\lambda({\mathbf{x}})-\text{K}_{\kappa}u(\mathbf{x})=0& &\mathbf{x}\in\Gamma.
  \end{empheq}
\end{subequations}
\noindent
We point out that $\lambda$, which is defined on the boundary $\Gamma$ in general by means of a trace operator (see \cite{QuarteroniValli1994}), is an additional unknown function. 
\noindent

For the theoretical analysis we will present in the forthcoming section, we further need to introduce the fundamental solution $G_{0}$ of the Laplace equation and its normal derivative:
\begin{equation*}\label{kernels_0}
G_{0}(\mathbf{x},\mathbf{y}):=-\frac{1}{2\pi}\log r \quad \text{and} \quad \frac{\partial G_{0}}{\partial\mathbf{n}_{\infty,\mathbf{y}}}(\mathbf{x},\mathbf{y})=\frac{1}{2\pi}\frac{\mathbf{r}\cdot\mathbf{n}_{\infty,\mathbf{y}}}{r^{2}}.
\end{equation*}
Denoting by $\text{V}_0$ and $\text{K}_0$ the associated single and double layer operators, the following regularity property of the operators $\text{V}_{\kappa}-\text{V}_{0}$ and $\text{K}_{\kappa}-\text{K}_{0}$ holds.
\begin{Lemma}  \label{Lemma:Vcompa}
	The operators $\text{V}_{\kappa} - \text{V}_{0} : H^{-\nicefrac{1}{2}}(\Gamma)\rightarrow H^{\nicefrac{5}{2}}(\Gamma)$ and $\text{K}_{\kappa} - \text{K}_{0} : H^{\nicefrac{1}{2}}(\Gamma)\rightarrow H^{\nicefrac{3}{2}}(\Gamma)$ are continuous.
\end{Lemma}
\begin{proof}
	We preliminary recall that 
	the Hankel functions $H_{m}^{(1)}$, with $m=0,1$, have the following asymptotic behaviour when $r\rightarrow0$ (see formulae (2.14) and (2.15) in \cite{SaranenVainikko2002}):
	\begin{subequations}\label{asymptotic_behaviour_H}
		\begin{empheq}{align}
		\label{asymptotic_behaviour_H_0_1} 
		&H_0^{(1)}(r) = \frac{\imath2}{\pi} \log{(r)} + 1 + \frac{\imath2}{\pi}\left(\gamma - \log{(2)}\right) + O(r^{2}),\\
		\label{asymptotic_behaviour_H_1_1}
		 &H_1^{(1)}(r) = -\frac{\imath2}{\pi r} + O(1)
		\end{empheq}
	\end{subequations}
	\noindent
	where $\gamma\simeq0.577216$ is the Euler constant. Then it easily follows that, when $r\rightarrow0$
	\begin{subequations}
		\begin{empheq}{align*}
		&G_{\kappa}(\mathbf{x},\mathbf{y})-G_{0}(\mathbf{x},\mathbf{y}) = \frac{\imath}{4}-\frac{1}{2\pi} \left(\gamma - \log{\left(\frac{\kappa}{2}\right)}\right) + O(r^2),\\
		&\frac{\partial G_{\kappa}}{\partial\mathbf{n}_{\infty,\mathbf{y}}}(\mathbf{x},\mathbf{y})-\frac{\partial G_{0}}{\partial\mathbf{n}_{\infty,\mathbf{y}}}(\mathbf{x},\mathbf{y}) = O(1).
		\end{empheq}
	\end{subequations}
	\noindent
	Following \cite{HsiaoWendland2008} (see Section 7.1), we can therefore deduce that  
	$G_{\kappa}-G_{0}$ and $\displaystyle \frac{\partial G_{\kappa}}{\partial\mathbf{n}_{\infty,\mathbf{y}}}-\frac{\partial G_{0}}{\partial\mathbf{n}_{\infty,\mathbf{y}}}$ are kernel functions with pseudo-homogeneous expansions of degree 2 and 0, respectively.
	From these properties, and proceeding as in \cite{SauterSchwab2011} (see Remark 3.13), the thesis is proved. 
\end{proof}

\begin{remark}\label{rk:reg_V0}
	Similarly, since $G_0$ is a kernel function with a pseudo-homogeneous expansion of degree 0, we can deduce that  the operator $\text{V}_0 : H^{s}(\Gamma) \rightarrow H^{s+1}(\Gamma)$ is continuous for all $s \in \mathbf{R}$.
\end{remark}

\section{The weak formulation of the model problem}\label{sec_3_weak_pb}
	We start by noting that, as usual, we can reduce the non homogeneous boundary condition on $\Gamma_0$ in \eqref{dirichlet_problem} to a homogeneous
	one by splitting $u_e$ as the sum of a suitable fixed function in $H^{1}_{g,\Gamma_{0}}(\Omega):= \{ u \in H^{1}(\Omega) : u = g \, \text{ on } \Gamma_0\}$ satisfying the Sommerfeld radiation condition
	and of an unknown function belonging to the space  $H^{1}_{0,\Gamma_{0}}(\Omega)$. Therefore, from now on, we consider Problem  \eqref{model_problem_full} with $g=0$.
	
	\

In order to derive the weak form of Problem (\ref{model_problem_full}), we introduce the bilinear forms $a:H^{1}(\Omega)\times H^{1}(\Omega)\rightarrow\mathbf{C}$ and $m : L^2(\Omega) \times L^2(\Omega) \rightarrow \mathbf{C}$ given by
\begin{equation}\label{bilinear_form}
a(u,v):=\int\limits_{\Omega}\nabla u(\mathbf{x})\cdot\nabla v(\mathbf{x})\,\dd\mathbf{x} \qquad \text{and} \qquad m(u,v):=\int\limits_{\Omega}u(\mathbf{x})v(\mathbf{x})\,\dd\mathbf{x},
\end{equation}
\noindent
and the $L^{2}(\Gamma)$-inner product $(\cdot,\cdot)_{\Gamma}:L^{2}(\Gamma)\times L^{2}(\Gamma)\rightarrow\mathbf{C}$
\begin{equation*}
(\lambda,v)_{\Gamma}=\int\limits_{\Gamma}\lambda(\mathbf{x})v(\mathbf{x})\dd\Gamma_{\mathbf{x}},
\end{equation*}
extended to the duality pairing $\langle\cdot,\cdot\rangle_{\Gamma}$ on $H^{-\nicefrac{1}{2}}(\Gamma)\times H^{\nicefrac{1}{2}}(\Gamma)$.\\

The variational formulation of Problem (\ref{model_problem_full}) consists in finding $u\in H_{0,\Gamma_0}^{1}(\Omega)$ and $\lambda\in H^{-\nicefrac{1}{2}}(\Gamma)$ such that 
\small
\begin{subequations}\label{model_problem_variational}
	\begin{empheq}[left=\empheqlbrace]{align}
	\label{model_problem_variational_1} &a(u,v)-\kappa^{2}m(u,v)-\langle \lambda,v\rangle_{\Gamma}=m(f,v)&  &\forall\, v\in H^{1}_{0,\Gamma_{0}}(\Omega),\\
	\label{model_problem_variational_2} &\langle \mu,\left(\frac{1}{2}I-\text{K}_{\kappa}\right)u\rangle_{\Gamma}+\langle \mu,\text{V}_{\kappa}\lambda\rangle_{\Gamma}=0 & 						 &\forall\, \mu\in H^{-\nicefrac{1}{2}}(\Gamma),
	\end{empheq}
\end{subequations}
\normalsize
\noindent
where $I$ stands for the identity operator. In order to reformulate the above problem as an equation in operator form, following \cite{JohnsonNedelec1980}, we consider the Hilbert space $V:=H^{1}_{0,\Gamma_{0}}(\Omega)\times H^{-\nicefrac{1}{2}}(\Gamma)$, equipped with the norm
\begin{equation*}\label{norm_V}
\left\|\hat{u}\right\|_{V}^{2}:=\left\|u\right\|_{H^{1}(\Omega)}^{2} + \left\|\lambda\right\|_{H^{-\nicefrac{1}{2}}(\Gamma)}^{2}, \quad \text{for} \ \hat u = (u,\lambda),
\end{equation*}
 induced by the scalar product
\begin{equation*}\label{scalar_product_V}
\left(\hat{u}, \hat{v}\right)_{V}:=\left(u,v\right)_{H^{1}(\Omega)} + \left(\lambda,\mu\right)_{H^{-\nicefrac{1}{2}}(\Gamma)}, \quad  \text{for} \ \hat u = (u,\lambda), \, \hat v = (v,\mu).
\end{equation*}
\noindent
We introduce the bilinear form $\mathcal{A}_{\kappa}:V\times V\rightarrow\mathbf{C}$ defined, for $\hat{u}  = (u,\lambda)$ and $\hat{v} = (v,\mu)$, by 
\begin{equation}\label{def_A_kappa}
\mathcal{A}_{\kappa}(\hat{u},\hat{v}):=a(u,v)-\kappa^{2}m(u,v)-\langle \lambda,v\rangle_{\Gamma}+\langle \mu,u\rangle_{\Gamma}+2\langle \mu,\text{V}_{\kappa}\lambda\rangle_{\Gamma}-2\langle \mu,\text{K}_{\kappa}u\rangle_{\Gamma},
\end{equation}
and the linear continuous operator $\mathcal{L}_{f}:V\rightarrow\mathbf{C}$ 
\begin{equation*}\label{def_lf_kappa}
\mathcal{L}_{f}(\hat{v}):=m(f,v), \quad \hat{v} = (v,\mu). 
\end{equation*}
\noindent
Thus, Problem (\ref{model_problem_variational}) can be rewritten as follows: find $\hat{u}\in V$ such that
\begin{equation}\label{model_problem_operator}
\mathcal{A}_{\kappa}(\hat{u},\hat{v})=\mathcal{L}_{f}(\hat{v}) \qquad \forall\,\hat{v}\in V.
\end{equation}
\noindent
The well-posedness of the above problem has been proved in \cite{MarquezMeddahiSelgas2003} (see Theorem 3.2), provided that $\kappa^{2}$ is not an eigenvalue of the Dirichlet-Laplace problem in $\Omega$.\\

\noindent
For what follows, it will be useful to rewrite $\mathcal{A}_{\kappa}$ by means of
the bilinear forms
$\mathcal{B}_{\kappa},\mathcal{K}_{\kappa}:V\times V\rightarrow\mathbf{C}$, defined as:
\begin{subequations}\label{definition_A_B_K}
  \begin{empheq}{align}
  \label{definition_A_kappa}
  &\mathcal{A}_{\kappa}(\hat{u},\hat{v}):=\mathcal{B}_{\kappa}(\hat{u},\hat{v})+\mathcal{K}_{\kappa}(\hat{u},\hat{v})\\
    \label{definition_B_kappa} &\mathcal{B}_{\kappa}(\hat{u},\hat{v}) := a(u,v)-\kappa^{2}m(u,v)-\langle \lambda,v\rangle_{\Gamma}+\langle \mu,u\rangle_{\Gamma}+2\langle \mu,\text{V}_{\kappa}\lambda\rangle_{\Gamma} \\
    \label{definition_K_kappa} &\mathcal{K}_{\kappa}(\hat{u},\hat{v}) := -2\langle \mu,\text{K}_{\kappa}u\rangle_{\Gamma}  \end{empheq}
\end{subequations}
for $\hat u = (u,\lambda), \hat v = (v,\mu) \in V$.
Due to the continuity property of both the operators $\text{V}_{\kappa}$ and $\text{K}_{\kappa}$ when $\kappa$ is real and non-negative (see \cite{GaticaMeddahi2020}), by using the trace theorem and the Cauchy-Schwarz inequality, it is easy to prove that the corresponding linear mappings $\mathcal{A}_{\kappa}, \mathcal{B}_{\kappa},\mathcal{K}_{\kappa}:V\rightarrow V^{'}$, defined by
$$\left(\mathcal{A}_{\kappa}\hat{u}\right)(\hat{v}) := \mathcal{A}_{\kappa}(\hat{u},\hat{v}), \qquad
	\left(\mathcal{B}_{\kappa}\hat{u}\right)(\hat{v}) := \mathcal{B}_{\kappa}(\hat{u},\hat{v}), \qquad
\left(\mathcal{K}_{\kappa}\hat{u}\right)(\hat{v}) := \mathcal{K}_{\kappa}(\hat{u},\hat{v}),$$
%
are continuous from $V$ to its dual $V^{'}$. 
\noindent
Finally, we introduce the adjoint operators $\mathcal{A}^{*}_{\kappa}, \mathcal{B}^{*}_{\kappa}:V\rightarrow V^{'}$ defined by:
\begin{subequations}
  \begin{empheq}{align*}
    &\left(\mathcal{A}^{*}_{\kappa}\hat{v}\right)(\hat{u}) := \left(\mathcal{A}_{\kappa}\hat{u}\right)(\hat{v})= \mathcal{A}_{\kappa}(\hat{u},\hat{v})\\
    &\left(\mathcal{B}^{*}_{\kappa}\hat{v}\right)(\hat{u}) := \left(\mathcal{B}_{\kappa}\hat{u}\right)(\hat{v})= \mathcal{B}_{\kappa}(\hat{u},\hat{v}).
   \end{empheq}
\end{subequations}

\noindent
In the following remarks we recall classical results about the afore introduced maps.
\begin{remark}\label{rk:Astar_cont}
	Theorem 3.2 in 
 \cite{MarquezMeddahiSelgas2003} and the closed graph theorem ensure that, if $\kappa^{2}$ is not an eigenvalue of the Dirichlet-Laplace problem in $\Omega$,
the inverse linear mappings $\mathcal{A}_{\kappa}^{-1},\mathcal{A}_{\kappa}^{* -1}:V^{'}\rightarrow V$ are continuous.
\end{remark}

\begin{remark}\label{rk:isomorfismi_k0}
Denoting by $H_{0}^{-\nicefrac{1}{2}}(\Gamma):=\left\{\lambda\in H^{-\nicefrac{1}{2}}(\Gamma) \ : \ \langle \lambda,1\rangle_{\Gamma}=0\right\}$, we set $\widetilde{V}:=H^{1}_{0,\Gamma_{0}}(\Omega)\times H^{-\nicefrac{1}{2}}_0(\Gamma)$. It has been proved in \cite{JohnsonNedelec1980} (see Lemmas 1, 2 and 3) that the mappings $\mathcal{A}_{0},\mathcal{A}^{*}_{0},\mathcal{B}_{0},\mathcal{B}^{*}_{0}:\widetilde{V}\rightarrow\widetilde{V}^{'}$ are isomorphisms. Moreover, for $s\geq0$, the mappings $\mathcal{A}_{0}^{-1},\mathcal{A}_{0}^{*  -1}, \mathcal{B}_{0}^{-1},\mathcal{B}_{0}^{*  -1} : H^{s-1}(\Omega) \times H^{s-\nicefrac{1}{2}}(\Gamma) \times H^{s+\nicefrac{1}{2}}(\Gamma) \to H^{s+1}(\Omega) \times H^{s-\nicefrac{1}{2}}(\Gamma)$ are continuous. Finally, we recall that $\mathcal{B}_{0}$ is coercive in the $\widetilde{V}$-norm.
\end{remark}

\section{The Galerkin method}\label{sec_4_galerkin_pb}
In what follows, the notation $Q_1 \apprle Q_2$ (resp. $Q_1 \apprge Q_2$) means that the quantity $Q_1$ is bounded from above (resp. from below) by $c\,Q_2$, where $c$ is a positive constant that, unless explicitly stated, does not depend on any relevant parameter involved in the definition of $Q_1$ and $Q_2$.

\

In order to describe the Galerkin approach applied to \eqref{model_problem_operator}, we introduce a sequence of unstructured meshes $\{\mathcal{T}_{h}\}_{h>0}$, that represent coverages of the domain $\Omega$ with a finite number of elements $E$, having diameter $h_E$. The mesh width $h>0$, related to the spacing of the grid, is defined as 
$h:=\underset{E\in\mathcal{T}_{h}}{\max}h_E$. Moreover, we denote by $\mathcal{T}_{h}^{\Gamma}$ the decomposition of the artificial boundary $\Gamma$, inherited from $\mathcal{T}_{h}$, which consists of curvilinear parts joined with continuity.
\noindent
We suppose that for each $h$ and for each element $E\in\mathcal{T}_{h}$ there exists a constant $\varrho>0$ such that the following assumptions are fulfilled:
\begin{enumerate}[label=(\subscript{A}{{\arabic*}})]
    {\setlength\itemindent{50pt} \item\label{A1} $E$ is star-shaped with respect to a ball of radius greater than $\varrho h_{E}$};
    {\setlength\itemindent{50pt} \item\label{A2} the length of any (eventually curved) edge of $E$ is greater than $\varrho h_{E}$}.
\end{enumerate}
We introduce the splitting of the bilinear forms $a$ and $m$ defined in \eqref{bilinear_form} into a sum of local bilinear forms $a^{\text{\tiny{E}}},m^{\text{\tiny{E}}}:H^{1}(E)\times H^{1}(E)\rightarrow\mathbf{C}$, associated to the elements $E$ of the decomposition of $\Omega$:
\begin{align*}
a(u,v)&=\sum\limits_{E\in\mathcal{T}_{h}}a^{\text{\tiny{E}}}(u,v):=\sum\limits_{E\in\mathcal{T}_{h}}\int\limits_{E}\nabla u(\mathbf{x})\cdot\nabla v(\mathbf{x})\,\dd\mathbf{x},\\
m(u,v)&=\sum\limits_{E\in\mathcal{T}_{h}}m^{\text{\tiny{E}}}(u,v):=\sum\limits_{E\in\mathcal{T}_{h}}\int\limits_{E}u(\mathbf{x})v(\mathbf{x})\,\dd\mathbf{x}.
\end{align*}
\noindent
Then, for any $k\in\mathbf{N}$, denoting by $P_{k}(E)$ the space of polynomials of degree $k$ defined on $E$, we introduce the local polynomial $H^{1}$-projection $\Pi_{k}^{\nabla,E}:H^{1}(E)\rightarrow P_{k}(E)$, defined such that for every $v\in H^{1}(E)$:
\begin{equation}\label{proj_Pi_nabla}
\begin{cases}
\displaystyle\int_{E}\nabla\Pi_{k}^{\nabla,E}v \cdot \nabla q\,\dd E = \displaystyle\int_{E}\nabla v \cdot\nabla q\,\dd E \qquad  \forall\, q\in P_{k}(E),\\[10pt]
\displaystyle\int_{\partial E}\Pi_k^{\nabla,E} v\,\dd s = \displaystyle\int_{\partial E}v\,\dd s 
\end{cases}
\end{equation}
and the local polynomial $L^{2}$-projection operator $\Pi_{k}^{0,E}:L^{2}(E)\rightarrow P_{k}(E)$, defined for all $v\in L^{2}(E)$ such that
\begin{equation}\label{proj_Pi_0}
\int_{E}\Pi_{k}^{0,E}v \, q\,\dd E= \int_{E} v \, q\,\dd E \qquad \forall\, q\in P_{k}(E).
\end{equation}
From the definition of $\Pi_{k}^{\nabla,E}$ and of $\Pi_{k}^{0,E}$, it follows:
\begin{align*}
&a^{\text{\tiny{E}}}\left(\Pi_{k}^{\nabla,E}v,q\right)=a^{\text{\tiny{E}}}\left(v,q\right) & \forall\,  q\in P_{k}(E),\\
&m^{\text{\tiny{E}}}\left(\Pi_{k}^{0,E}v,q\right)=m^{\text{\tiny{E}}}\left(v,q\right) & \forall\,  q\in P_{k}(E).
\end{align*}
\noindent
Moreover, since $\Omega$ is the union of star-shaped domains $E\in\mathcal{T}_{h}$, the local polynomial projectors $\Pi_{k}^{\nabla,E}$ and $\Pi_{k}^{0,E}$ can be extended to the global projectors $\Pi_{k}^{\nabla}:H^{1}(\Omega)\rightarrow P_{k}(\mathcal{T}_{h})$ and $\Pi_{k}^{0}:L^{2}(\Omega)\rightarrow P_{k}(\mathcal{T}_{h})$ as follows:\\
\begin{subequations}
  \begin{empheq}{align*}
    &\left(\Pi_{k}^{\nabla} v\right)_{|_E}:=\Pi_{k}^{\nabla,E}v_{|_E}&  &\forall\, v\in H^{1}(\Omega)\\
    &\left(\Pi_{k}^{0} v\right)_{|_E}:=\Pi_{k}^{0,E}v_{|_{E}}&  &\forall\, v\in L^{2}(\Omega),
  \end{empheq}
\end{subequations}
$P_k(\mathcal{T}_{h})$ being the space of piecewise polynomials with respect to the decomposition $\mathcal{T}_{h}$ of $\Omega$. 

\

\noindent

In the following lemma we prove a polynomial approximation property of the above defined projectors. To this aim, since we shall deal with functions belonging to the space $H^1(\mathcal{T}_h) := \underset{E\in\mathcal{T}_h}{\prod} H^1(E)$, we need to introduce the following broken $H^1$ norm %
$$\| v \|_{H^1(\mathcal{T}_h)} := \left(\sum_{E\in \mathcal{T}_h} \| v \|^2_{H^1(E)}\right)^{1/2}.$$
\begin{Lemma}
Assuming \ref{A1}, for all $v\in H^{s+1}(\Omega)$ with $0\leq s\leq k$, it holds:
\begin{equation}\label{interp_property_Pi_0}  
\left\|v-\Pi_{k}^{0} v\right\|_{L^{2}(\Omega)}\apprle h^{s+1}\left\|v\right\|_{H^{s+1}(\Omega)}.
\end{equation}
Moreover, for all $v\in H^{s+1}(\Omega)$ with $1\leq s\leq k$, it holds:
\begin{equation}\label{interp_property_Pi_nabla} 
\left\|v-\Pi_{k}^{\nabla} v\right\|_{H^{1}(\mathcal{T}_h)}\apprle h^s\left\|v\right\|_{H^{s+1}(\Omega)}.
\end{equation}
\end{Lemma}
\begin{proof}
Let us denote by $R_E > 0$ the radius of the ball in $E\in\mathcal{T}_{h}$ satisfying \ref{A1}. For any $v\in L^{2}(E)$ and $q\in P_{k}(E)$ we can write
\begin{equation*}
\left\|v-\Pi_{k}^{0,E}v\right\|_{L^{2}(E)}^{2}=
\left(v-\Pi_{k}^{0,E}v,v-q\right)_{L^{2}(E)}+\left(v-\Pi_{k}^{0,E}v,q-\Pi_{k}^{0,E}v\right)_{L^{2}(E)} = \left(v-\Pi_{k}^{0,E}v,v-q\right)_{L^{2}(E)},
\end{equation*}
where we have used (\ref{proj_Pi_0}) together with $q-\Pi_{k}^{0,E}v\in P_{k}(E)$.
%
%
Then, by applying the Cauchy-Schwarz inequality, we easily get
%
%
%
\begin{equation*}
\left\|v-\Pi_{k}^{0,E}v\right\|_{L^{2}(E)}\leq\left\|v-q\right\|_{L^{2}(E)} \qquad \forall\, q\in P_{k}(E).
\end{equation*}
Now, assuming $v\in H^{s+1}(E)$ with $0\leq s\leq k$ and using the Bramble-Hilbert Lemma (see Lemma 4.3.8 in \cite{BrennerScott2008}), we have

\begin{equation*}
\left\|v-\Pi_{k}^{0,E}v\right\|_{L^{2}(E)}\leq\underset{q \in P_k(E)}{\text{inf}}\left\|v-q\right\|_{L^{2}(E)}\apprle C\hskip-.1cm\left(\frac{h_{E}}{R_E}\right)h_{E}^{s+1}\left\|v\right\|_{H^{s+1}(E)},
\end{equation*}
where the implicit constant depends only on $k$ and $C:\mathbf{R}^{+}\rightarrow\mathbf{R}^{+}$ is an increasing function.
Since, by virtue of Assumption \ref{A1}, the function $C\hskip-.08cm\left(\nicefrac{h_{E}}{R_E}\right)$ is uniformly bounded, we can easily get \eqref{interp_property_Pi_0}.
%
%
Finally, inequality (\ref{interp_property_Pi_nabla}) can be proved similarly.
\end{proof}
\subsection{The discrete variational formulation}\label{sec:DVF}

We present here a class of Galerkin type discretizations of Problem \eqref{model_problem_operator}, which includes,
but is not limited to, VEMs. In Section \ref{sec_5_discrete_scheme} we will give an example of CVEM that falls in the framework considered.
\

Let $Q_{h}^{k}\subset H_{0,\Gamma_{0}}^{1}(\Omega)$ and $X_{h}^{k}\subset H^{-\nicefrac{1}{2}}(\Gamma)$ denote two finite dimensional spaces associated to the meshes $\mathcal{T}_{h}$ and $\mathcal{T}_{h}^{\Gamma}$, respectively. 
Introducing the discrete space $V_{h}^{k}:=Q_{h}^{k}\times X_{h}^{k}$,
the Galerkin method applied to Problem (\ref{model_problem_variational}) reads: 
find $\hat{u}_{h}\in V_{h}^{k}$ such that
\begin{equation}\label{model_problem_galerkin}
\mathcal{A}_{\kappa,h}(\hat{u}_{h},\hat{v}_{h}):=\mathcal{B}_{\kappa,h}(\hat{u}_{h},\hat{v}_{h})+\mathcal{K}_{\kappa}(\hat{u}_{h},\hat{v}_{h})=\mathcal{L}_{f,h}(\hat{v}_{h}) \quad \forall\,\hat{v}_{h}\in V_{h}^{k},
\end{equation}
\noindent
where $\mathcal{A}_{\kappa,h}, \mathcal{B}_{\kappa,h}:V_{h}^{k}\times V_{h}^{k}\rightarrow\mathbf{C}$ and $\mathcal{L}_{f,h}:V_{h}^{k}\rightarrow\mathbf{C}$ 
 are suitable approximations of $\mathcal{A}_{\kappa}$, $\mathcal{B}_{\kappa}$ and $\mathcal{L}_{f}$, respectively. 

\

\noindent
In order to prove existence and uniqueness of the solution $\hat{u}_{h}\in V_{h}^{k}$ and to derive a priori error estimates, we preliminary introduce some assumptions on the discrete spaces, on the bilinear form $\mathcal{B}_{\kappa,h}$ and on the linear operator $\mathcal{L}_{f,h}$.\\

\noindent
We assume that the following properties for $Q_{h}^{k}$, $X_{h}^{k}$ and $\widetilde X_h^k := X_h^k \cap H_0^{-\nicefrac{1}{2}}(\Gamma)$ hold: for $1\leq s\leq k$
\begin{enumerate}[label=(\subscript{H1}{{\alph*}})]
    {\setlength\itemindent{50pt} \item\label{H1.a} $\underset{v_{h}\in Q_{h}^{k}}{\text{inf}} \left\|v-v_{h}\right\|_{H^{1}(\Omega)} \apprle h^{s} \left\|v\right\|_{H^{s+1}(\Omega)} \hspace{0.2cm}  \quad \quad \forall \, v\in H^{s+1}(\Omega)$};
    {\setlength\itemindent{50pt} \item\label{H1.b} $\underset{\mu_{h}\in X_{h}^{k}}{\text{inf}} \left\|\mu-\mu_{h}\right\|_{H^{-\nicefrac{1}{2}}(\Gamma)} \apprle h^{s} \left\|\mu\right\|_{H^{s-\nicefrac{1}{2}}(\Gamma)} \hspace{0.2cm}  \quad \forall \, \mu\in H^{s-\nicefrac{1}{2}}(\Gamma)$}; 
        {\setlength\itemindent{50pt} \item\label{H1.c} $\underset{\mu_{\0 h}\in \tilde X_{h}^{k}}{\text{inf}} \left\|\mu_{\0}-\mu_{\0 h}\right\|_{H^{-\nicefrac{1}{2}}(\Gamma)} \apprle h^{s} \left\|\mu_{\0} \right\|_{H^{s-\nicefrac{1}{2}}(\Gamma)} \hspace{0.2cm}  \quad \forall \, \mu_{\0} \in H^{s-\nicefrac{1}{2}}(\Gamma) \cap H_0^{-\nicefrac{1}{2}}(\Gamma)$}.
\end{enumerate}
\noindent 
According to the definition of the $\|\cdot\|_V$ norm, the above assumptions ensure the following approximation property for the product spaces $V_{h}^{k}$ and $\tilde V_{h}^{k} := Q_h^k \times \tilde X_h^k$:
\begin{itemize}
	\item given $\hat{v}=(v,\mu)\in H^{s+1}(\Omega)\times H^{s-\nicefrac{1}{2}}(\Gamma)$, there exists $\hat{v}_{h}^{I}=(v_{h}^{I},\mu_{h}^{I})\in V_{h}^{k}$ such that
	\begin{equation}\label{int_property} 
	\left\|\hat{v} - \hat{v}_{h}^{I}\right\|_V\apprle h^{s} \left(\left\|v\right\|_{H^{s+1}(\Omega)}+\left\|\mu\right\|_{H^{s-\nicefrac{1}{2}}(\Gamma)}\right); 
	\end{equation}
	\item given $\hat{v}_{\0}=(v,\mu_{\0})\in H^{s+1}(\Omega)\times (H^{s-\nicefrac{1}{2}} (\Gamma) \cap H_0^{-\nicefrac{1}{2}}(\Gamma))$, there exists $\hat{v}_{\0 h}^{I}=(v_{h}^{I},\mu_{\0 h}^{I})\in \tilde V_{h}^{k}$ such that
	\begin{equation}\label{int_property_0} 
	\left\|\hat{v}_{\0} - \hat{v}_{\0 h}^{I}\right\|_V\apprle h^{s} \left(\left\|v\right\|_{H^{s+1}(\Omega)}+\left\|\mu_{\0} \right\|_{H^{s-\nicefrac{1}{2}}(\Gamma)}\right).
	\end{equation}
\end{itemize}

\noindent
Concerning the bilinear form $\mathcal{B}_{\kappa,h}$, we assume that for all $\kappa \geq 0$: 
\begin{enumerate}[label=(\subscript{H2}{{\alph*}})]
    {\setlength\itemindent{50pt} \item\label{H2.a} $k$-consistency: $ \mathcal{B}_{\kappa,h}(\hat{q},\hat{v}_{h}) = \mathcal{B}_{\kappa}(\hat{q},\hat{v}_{h}) \quad \forall\,\hat{q}\in P_k(\mathcal{T}_h)\times X_{h}^{k}\ \text{and} \ \forall\,\hat{v}_{h}\in V_{h}^{k}$;}
    {\setlength\itemindent{50pt} \item\label{H2.b} continuity in $V$-norm:  $ \left|\B_{\kappa,h}(\hat{v}_{h},\hat{w}_{h})\right|\apprle\left\|\hat{v}_{h}\right\|_V \left\|\hat{w}_{h}\right\|_V \quad \forall \hat{v}_{h},\hat{w}_{h}\in V_{h}^{k}$}.
\end{enumerate}
\begin{remark}
It is worth noting that, in Assumption \ref{H2.a}, the evaluation of the bilinear form $\mathcal{B}_\kappa$ is well defined provided that the computation of the bilinear form $a(\cdot,\cdot)$ is split into the sum of the local contributions associated to the elements $E$ of $\mathcal{T}_h$. For simplicity of notation, here and in what follows, we assume that such splitting is considered whenever necessary. Moreover, we assume that the approximated bilinear form $\mathcal{B}_{\kappa,h}$ is well defined on the space $H^1(\mathcal{T}_h)$.
\end{remark}
\noindent
Assumptions \ref{H2.a} and \ref{H2.b} allow to prove the following consistency result.
\begin{Lemma} \label{Lemma:Bwithinter}
Let $\hat{v}:=\left(v,\mu\right)\in H^{s+1}(\Omega) \times H^{s-\nicefrac{1}{2}}(\Gamma)$, $1\leq s\leq k$, and $\hat{v}_{h}^{I}:=\left(v_{h}^{I},\mu_{h}^{I}\right)\in V_{h}^{k}$ the interpolant of $\hat{v}$ in $V_{h}^{k}$ such that relation (\ref{int_property}) holds. Then
\begin{equation*}
\left|\mathcal{B}_{\kappa}(\hat{v}_{h}^{I},\hat{w}_{h})-\mathcal{B}_{\kappa,h}(\hat{v}_{h}^{I},\hat{w}_{h})\right|\apprle h^{s}\left\|v\right\|_{H^{s+1}(\Omega)} \left\|\hat{w}_{h}\right\|_{V} \qquad \forall\,\hat{w}_{h}\in V_{h}^{k}.
\end{equation*}
\end{Lemma}
\begin{proof} By abuse of notation, let $ \Pi^{\nabla}_{k}\hat{v}:=(\Pi^{\nabla}_{k}v,\mu_{h}^{I})$. 
%
%
We start from the inequality
 \begin{equation}\label{eq:ineq_uno}
\left|\mathcal{B}_{\kappa}(\hat{v}_{h}^{I},\hat{w}_{h}) - \mathcal{B}_{\kappa,h}(\hat{v}_{h}^{I},\hat{w}_{h})\right|\leq
\left|\mathcal{B}_{\kappa}(\hat{v}_{h}^{I},\hat{w}_{h})-\mathcal{B}_{\kappa,h}(\Pi_{k}^{\nabla}\hat{v},\hat{w}_{h})\right|+\left|\mathcal{B}_{\kappa,h}(\Pi_{k}^{\nabla}\hat{v},\hat{w}_{h})-\mathcal{B}_{\kappa,h}(\hat{v}_{h}^{I},\hat{w}_{h})\right| =: I + II.
 \end{equation}
By using Assumption \ref{H2.a} and the continuity of $\mathcal{B}_{\kappa}$ in the $V$-norm, we get
%
 \begin{equation}\label{eq:ineq_due}
I=\left|\mathcal{B}_{\kappa}(\Pi_{k}^{\nabla}\hat{v}-\hat{v}_{h}^{I},\hat{w}_{h})\right|\apprle \left\|\Pi_{k}^{\nabla}\hat{v}-\hat{v}_{h}^{I}\right\|_{H^1(\mathcal{T}_h)\times H^{-\nicefrac{1}{2}}(\Gamma)}\left\|\hat{w}_{h}\right\|_{V}.
\end{equation}
Concerning the term $II$, from \ref{H2.b} we obtain
 \begin{equation}\label{eq:ineq_tre}
II=\left|\mathcal{B}_{\kappa,h}(\Pi_{k}^{\nabla}\hat{v}-\hat{v}_{h}^{I},\hat{w}_{h})\right|\apprle \left\|\Pi_{k}^{\nabla}\hat{v}-\hat{v}_{h}^{I}\right\|_{H^1(\mathcal{T}_h)\times H^{-\nicefrac{1}{2}}(\Gamma)}\left\|\hat{w}_{h}\right\|_{V}.
\end{equation}
%
%
Finally, by definition of $\Pi_{k}^{\nabla}\hat{v}$, using (\ref{interp_property_Pi_nabla}) and Assumption \ref{H1.a}
we can write
%
\begin{equation*}
\left\|\Pi_{k}^{\nabla}\hat v-\hat v_{h}^{I}\right\|_{H^1(\mathcal{T}_h)\times H^{-\nicefrac{1}{2}}(\Gamma)} = \left\|\Pi_{k}^{\nabla}v-v_{h}^{I}\right\|_{H^{1}(\mathcal{T}_h)}\leq \left\|\Pi_{k}^{\nabla}v-v\right\|_{H^{1}(\mathcal{T}_h)} + \left\|v-v_{h}^{I}\right\|_{H^{1}(\Omega)} \apprle h^{s}\left\|v\right\|_{H^{s+1}(\Omega)},
\end{equation*}
from which, combining \eqref{eq:ineq_due} and \eqref{eq:ineq_tre} with \eqref{eq:ineq_uno}, the thesis follows.
%
\end{proof}
%
Similarly, the following lemma can be proved.
\begin{Lemma} \label{Lemma:Bwithinter_0}
	Let $\hat v_{\0}:=\left(v,\mu_{\0}\right)\in H^{s+1}(\Omega) \times (H^{s-\nicefrac{1}{2}}(\Gamma) \cap H_0^{-\nicefrac{1}{2}}(\Gamma))$,  $1\leq s\leq k$, and $\hat{v}_{\0 h}^{I}:=\left(v_{h}^{I},\mu_{\0 h}^{I}\right)\in \tilde V_{h}^{k}$ the interpolant of $\hat{v}_{\0}$ in $\tilde V_{h}^{k}$ such that relation (\ref{int_property_0}) holds. Then
	\begin{equation*}
	\left|\mathcal{B}_{\0}(\hat{v}_{\0 h}^{I},\hat{w}_{\0 h})-\mathcal{B}_{\0,h}(\hat{v}_{\0 h}^{I},\hat{w}_{\0 h})\right|\apprle h^{s}\left\|v\right\|_{H^{s+1}(\Omega)} \left\|\hat{w}_{\0 h}\right\|_{V} \qquad \forall\,\hat{w}_{\0 h}\in \tilde V_{h}^{k}.
	\end{equation*}
\end{Lemma}
\noindent
In order to prove the main results of our theoretical analysis, we need to introduce further assumptions on the approximate bilinear form $\mathcal{B}_{\kappa,h}$. Denoting by $\mathcal{D}_{\kappa,h}:=\mathcal{B}_{\kappa,h}-\mathcal{B}_{0,h}$, we require:
\begin{enumerate}[label=(\subscript{H3}{{\alph*}})]
    {\setlength\itemindent{50pt} \item\label{H3.a} $\mathcal{D}_{\kappa,h}$ is continuous in the weaker $W$-norm, with $W:= L^{2}(\Omega)\times H^{-\nicefrac{1}{2}}(\Gamma)$:
    \begin{equation*}
    \left|\mathcal{D}_{\kappa,h}(\hat{v}_{h},\hat{w}_{h})\right|\apprle \left\|\hat{v}_{h}\right\|_{W} \left\|\hat{w}_{h}\right\|_{W} \qquad \forall\,\hat{v}_{h},\hat{w}_{h}\in V_{h}^{k};
    \end{equation*}}
    {\setlength\itemindent{50pt} \item\label{H3.b} $\mathcal{B}_{0,h}$ is  $\tilde V_h^k$-elliptic:
   \begin{equation*}
    \mathcal{B}_{0,h}(\hat{w}_{\0 h},\hat{w}_{\0 h}) \apprge\left\|\hat{w}_{\0 h}\right\|_{V}^{2} \qquad \forall\, \hat{w}_{\0 h} \in \tilde V_{h}^{k};
    \end{equation*}}
     {\setlength\itemindent{50pt} \item\label{H3.c}$k$-consistency in the second term of $\mathcal{B}_{0,h}$: 
        \begin{equation*} \mathcal{B}_{0,h}(\hat w_h,\hat{q}) = \mathcal{B}_0(\hat{w}_h,\hat{q}) \quad \forall\,\hat{q}\in P_k(\mathcal{T}_h)\times X_{h}^{k}\ \text{and} \ \forall\,\hat{w}_{h}\in V_{h}^{k}.
        \end{equation*}}
\end{enumerate}
\begin{remark}\label{rk:Dk}
We remark that Assumption \ref{H3.a} is the discrete counterpart of the continuity property of the bilinear form $\mathcal{D}_{\kappa}:=\mathcal{B}_{\kappa}-\mathcal{B}_{0}$. Indeed, according to the continuity of $\text{V}_{\kappa}-\text{V}_0$ and using the Cauchy-Schwarz inequality, we obtain: for $\hat{v}=(v,\mu)\in V$ and $\hat{w}=(w,\nu)\in V$
\begin{equation}\label{H3.a_B}
\left|\mathcal{D}_{\kappa}(\hat{v},\hat{w})\right|=\left|\kappa^2 m(v,w) - 2\langle \nu,(\text{V}_{\kappa}-\text{V}_0)\mu \rangle\right|\apprle\left\|\hat{v}\right\|_{W} \left\|\hat{w}\right\|_{W}.
\end{equation}
\end{remark}

\noindent
Assumptions \ref{H3.a}--\ref{H3.c} are used to prove the following Lemmas \ref{Lemma:BBwithinter}, \ref{Lemma:dualpb} and \ref{Lemma:ourdualpb}, which are then crucial to obtain the Ladyzhenskaya-Babu$\check{\text{s}}$ka-Brezzi  condition for the discrete bilinear form $\mathcal{A}_{\kappa,h}$.
\begin{Lemma} \label{Lemma:BBwithinter}
Let $\hat{v}_{h} = (v_h,\mu_h)\in\left(H^{s+1}(\Omega) \times H^{- \nicefrac{1}{2}}(\Gamma)\right)\cap V_{h}^{k}$ with $0\leq s\leq k$. Then
\begin{equation}
\left|\mathcal{D}_\kappa(\hat{v}_{h},\hat{w}_{h})-\mathcal{D}_{\kappa,h}(\hat{v}_{h},\hat{w}_{h})\right|\apprle h^{s+1} \left\|v_{h}\right\|_{H^{s+1}(\Omega)} \left\|\hat{w}_{h}\right\|_{W}, \quad\forall\, \hat{w}_{h}\in V_{h}^{k}.
\end{equation}
\end{Lemma}
\begin{proof} Let us denote, by abuse of notation, $\Pi_{k}^{0}\hat{v}_h:=\left(\Pi_{k}^{0}v_h,\mu_h\right)$. 
%
By adding and subtracting the term $\mathcal{D}_{\kappa,h}(\Pi_{k}^{0}\hat{v}_{h},\hat{w}_{h})$ and using Assumption \ref{H2.a}, we get
%
 \begin{equation*}
 \begin{aligned}
\left|\mathcal{D}_\kappa(\hat{v}_{h},\hat{w}_{h}) - \mathcal{D}_{\kappa,h}(\hat{v}_{h},\hat{w}_{h})\right|&\leq
\left|\mathcal{D}_\kappa(\hat{v}_{h}-\Pi_{k}^{0}\hat{v}_{h},\hat{w}_{h})\right|+\left|\mathcal{D}_{\kappa,h}(\hat{v}_{h}-\Pi_{k}^{0}\hat{v}_{h},\hat{w}_{h})\right|\\
&\apprle\left\|\hat{v}_{h}-\Pi_{k}^{0}\hat{v}_{h}\right\|_{W} \left\|\hat{w}_{h}\right\|_{W},
\end{aligned}
 \end{equation*}
the last inequality directly following from (\ref{H3.a_B}) and \ref{H3.a}.
%
%
%
%
Finally, by definition of $\Pi_{k}^{0}\hat{v}_h$ and using (\ref{interp_property_Pi_0}),
we can write
\begin{equation*}
\left\|\hat{v}_{h}-\Pi_{k}^{0}\hat{v}_{h}\right\|_{W}=\left\|v_{h}-\Pi_{k}^{0}v_{h}\right\|_{L^{2}(\Omega)}\apprle h^{s+1}\left\|v_{h}\right\|_{H^{s+1}(\Omega)},
\end{equation*}
from which the thesis follows.
%
\end{proof}
 \begin{Lemma}\label{Lemma:dualpb} 
Let $\hat{v}_{\0}=(v, \mu_{\0})\in \tilde V$. There exists one and only one $\hat v_{\0 h}=(v_{h},\mu_{\0 h})\in \tilde V_{h}^{k}$ such that 
\begin{equation}\label{pb:BhB}
\mathcal{B}_{0,h}(\hat w_{\0 h},\hat v_{\0 h})= \mathcal{B}_{0}(\hat w_{\0 h},\hat v_{\0})  \quad \forall\, \hat w_{\0 h}=(w_{h},\nu_{\0 h}) \in \tilde V_{h}^{k}.
\end{equation}
Moreover, it holds:
\begin{subequations}\label{eq:firstadv}
	\begin{empheq}{align}
	\label{eq:firstadv_a}
    &\left\| {\hat v}_{\0 h}\right\|_{V}\apprle\left\|{\hat v_{\0}}\right\|_{V}, \\
    \label{eq:firstadv_b}
    &\left\| \mu_{\0 h}-\mu_{\0}\right\|_{H^{-\nicefrac{3}{2}}(\Gamma)}\apprle h\left\|\hat v_{\0}\right\|_{V}, \\
    \label{eq:firstadv_c} &\left\|v_{h}-v\right\|_{L^2(\Omega)}\apprle h\left\|\hat v_{\0}\right\|_V.
	\end{empheq}
\end{subequations}
\end{Lemma} 
\begin{proof}
Existence and uniqueness of $\hat{v}_{\0 h}\in \tilde V_{h}^{k}$, solution of (\ref{pb:BhB}), follow from Assumptions \ref{H2.b} and \ref{H3.b}. Moreover, \eqref{eq:firstadv_a} holds according to \ref{H3.b} and the continuity of the bilinear form $\mathcal{B}_{0}$ (see Remark \ref{rk:isomorfismi_k0}).
%
%

\noindent
In order to prove \eqref{eq:firstadv_b}, by using a duality argument, it is sufficient to show that:
\begin{equation}\label{eq:B.7}
\left|\langle\mu_{\0 h}-\mu_{\0},\eta\rangle_{H^{-\nicefrac{3}{2}}(\Gamma)\times H^{\nicefrac{3}{2}}(\Gamma)}\right|\apprle h \left\|\hat{v}_{\0 h}\right\|_V \left\|\eta\right\|_{H^{\nicefrac{3}{2}}(\Gamma)} \quad \forall \eta\in H^{\nicefrac{3}{2}}(\Gamma).
\end{equation} 
Starting from $\eta\in H^{\nicefrac{3}{2}}(\Gamma)$, we consider $\widetilde{w}:=\left(0,0,\eta\right)\in L^{2}(\Omega)\times H^{\nicefrac{1}{2}}(\Gamma)\times H^{\nicefrac{3}{2}}(\Gamma)$ and we set $\hat{w}_{\0}:=\mathcal{B}_{0}^{-1}\widetilde{w}$. Then, we have:
\begin{equation} \label{eq:B.8}
\mathcal{B}_{0}(\hat{w}_{\0},\hat{z}_{\0})=\mathcal{B}_{0}(\mathcal{B}_{0}^{-1}\widetilde{w},\hat{z}_{\0}) = (\mathcal{B}_{0}\mathcal{B}_{0}^{-1}\widetilde{w})(\hat{z}_{\0}) = \langle\zeta_{\0},\eta\rangle_{H^{-\nicefrac{3}{2}}(\Gamma) \times H^{\nicefrac{3}{2}}(\Gamma) } \quad \forall\, \hat{z}_{\0}=(z,\zeta_{\0})\in \tilde V.
\end{equation}
From the continuity of $\mathcal{B}_{0}^{-1}:L^{2}(\Omega)\times H^{\nicefrac{1}{2}}(\Gamma)\times H^{\nicefrac{3}{2}}(\Gamma)\rightarrow H^{2}(\Omega)\times H^{\nicefrac{1}{2}}(\Gamma)$ (see  Remark \ref{rk:isomorfismi_k0}), it follows that:
\begin{equation} \label{eq:B.9}
\left\|\hat{w}_{\0}\right\|_{H^2(\Omega) \times H^{\nicefrac{1}{2}}(\Gamma)}\apprle \left\|\eta\right\|_{H^{\nicefrac{3}{2}}(\Gamma)}.
\end{equation}
Therefore, by choosing  $\hat{z}_{\0}=\hat{v}_{\0 h}-\hat{v}_{\0}$ in \eqref{eq:B.8}, we can write
\begin{equation*}
\langle\mu_{\0 h}-\mu_{\0},\eta\rangle_{H^{-\nicefrac{3}{2}}(\Gamma) \times H^{\nicefrac{3}{2}}(\Gamma)}=\mathcal{B}_{0}(\hat{w}_{\0},\hat{v}_{\0 h}-\hat{v}_{\0})=\mathcal{B}_{0}(\hat{w}_{\0}-\hat{w}_{\0 h},\hat{v}_{\0 h}-\hat{v}_{\0})+\mathcal{B}_{0}(\hat{w}_{\0 h},\hat{v}_{\0 h})-\mathcal{B}_{0}(\hat{w}_{\0 h},\hat{v}_{\0}).
\end{equation*}
Since $\hat{v}_{\0 h}\in \tilde V_{h}^{k}$ is the solution of (\ref{pb:BhB}), we rewrite the previous identity as follows:
\begin{eqnarray}\label{eq:est_32}
&\left|\langle\mu_{\0 h}-\mu_{\0},\eta\rangle_{H^{-\nicefrac{3}{2}}(\Gamma) \times H^{\nicefrac{3}{2}}(\Gamma)}\right|=\left|\mathcal{B}_{0}(\hat{w}_{\0}-\hat{w}_{\0 h},\hat{v}_{\0 h}-\hat{v}_{\0})+\mathcal{B}_{0}(\hat{w}_{\0 h},\hat{v}_{\0 h})-\mathcal{B}_{0,h}(\hat{w}_{\0 h},\hat{v}_{\0 h})\right|\nonumber\\
&\leq\left|\mathcal{B}_{0}(\hat{w}_{\0}-\hat{w}_{\0 h},\hat{v}_{\0 h}-\hat{v}_{\0})\right|+\left|\mathcal{B}_{0}(\hat{w}_{\0 h},\hat{v}_{\0 h})-\mathcal{B}_{0,h}(\hat{w}_{\0 h},\hat{v}_{\0 h})\right|=: I + II.
\end{eqnarray}
Choosing in \eqref{eq:est_32} $\hat{w}_{\0 h}=\hat{w}_{\0 h}^{I}$, the interpolant of $\hat{w}_{\0}\in \tilde V$ in $\tilde V_{h}^{k}$ such that (\ref{int_property_0}) holds, due to Lemma \ref{Lemma:Bwithinter_0} with $s=1$ and (\ref{eq:firstadv_a}), we can estimate $II$ as follows:
\begin{equation}\label{eq:II}
II\apprle h\left\|\hat{v}_{\0}\right\|_{V}\left\|\hat{w}_{\0}\right\|_{H^{2}(\Omega)\times H^{\nicefrac{1}{2}}(\Gamma)}.
\end{equation}
Combining the continuity of $\mathcal{B}_{0}$, (\ref{int_property_0}) with $s=1$ and (\ref{eq:firstadv_a}), we have:
\begin{equation}\label{eq:I}
I\apprle \left\|\hat{v}_{\0 h}-\hat{v}_{\0}\right\|_{V} \left\|\hat{w}_{\0}-\hat{w}_{\0 h}^{I}\right\|_{V} \apprle h \left\|\hat{v}_{\0 h}-\hat{v}_{\0}\right\|_{V} \left\| \hat w_{\0}\right\|_{H^{2}(\Omega) \times H^{\nicefrac{1}{2}}(\Gamma)} \apprle h\left\| \hat{v}_{\0} \right\|_{V}\left\|\hat{w}_{\0} \right\|_{H^{2}(\Omega) \times H^{\nicefrac{1}{2}}(\Gamma)}.
\end{equation}
%
Finally, from (\ref{eq:B.9}) and \eqref{eq:est_32} together with \eqref{eq:II} and \eqref{eq:I}, inequality \eqref{eq:B.7} directly follows.\\

\noindent
Inequality (\ref{eq:firstadv_c}) can be proved similarly to (\ref{eq:firstadv_b}). Indeed, if we consider $\widetilde{w}:=\left(v_{h}-v,0,0\right) \in L^2(\Omega) \times H^{\nicefrac{1}{2}}(\Gamma) \times H^{\nicefrac{3}{2}}(\Gamma)$ and $\hat{w}_{\0}=\mathcal{B}_{0}^{-1}\widetilde{w}$ in (\ref{eq:B.8}), we get
\begin{equation}\label{eq:B.8_bis}
\mathcal{B}_{0}(\hat{w}_{\0},\hat{z}_{\0})=\left(v_{h}-v,z\right)_{L^{2}(\Omega)} \qquad \forall\,\hat{z}_{\0}=(z,\zeta_{\0})\in \tilde V.
\end{equation}
\noindent
Then, choosing $\hat{z}_{\0}=\hat{v}_{\0 h}-\hat{v}_{\0}$ in (\ref{eq:B.8_bis}), we have
\begin{equation*}\label{eq:B.8_bis_z}
\mathcal{B}_{0}(\hat{w}_{\0},\hat{v}_{\0 h}-\hat{v}_{\0})=\left(v_{h}-v,v_{h}-v\right)_{L^{2}(\Omega)}=\left\|v_{h}-v\right\|_{L^{2}(\Omega)}^{2}.
\end{equation*}
Finally, by taking into account the continuity of $\mathcal{B}_{0}^{-1}$, we obtain
\begin{equation*} \label{eq:B.9_bis}
\left\|\hat{w}_{\0}\right\|_{H^2(\Omega) \times H^{\nicefrac{1}{2}}(\Gamma)}\apprle \left\|v_{h}-v\right\|_{L^{2}(\Omega)}
\end{equation*}
and, proceeding as we did to estimate \eqref{eq:est_32}, we write
\begin{equation*}
\left\|v_{h}-v\right\|_{L^2(\Omega)}^{2}\apprle h\left\|\hat{w}_{\0}\right\|_{H^2(\Omega)\times H^{\nicefrac{1}{2}}(\Gamma)}\left\|\hat{v}_{\0} \right\|_{V} \apprle h\left\|v_{h}-v\right\|_{L^2(\Omega)}\left\|\hat{v}_{\0}\right\|_{V},
\end{equation*}
from which (\ref{eq:firstadv_c}) follows.
\end{proof}

\begin{Lemma} \label{Lemma:ourdualpb} 
Let $\hat v = (v,\mu) \in V$. There exists $\hat v_h = (v_h,\mu_h) \in V_h^k$ such that
\begin{equation*}
\B_{0,h}(\hat w_h,\hat v_h) = \B_0(\hat w_h,\hat v) + \B_0((0,\bar \eta_h),\hat v_h - \hat v) \quad \forall \hat w_h \in (w_h,\eta_h) \in V_h^k
\end{equation*}
where $\bar \eta_h = \nicefrac{\langle \eta_h,1 \rangle}{\langle 1,1 \rangle}$. Moreover, it holds
\begin{subequations}\label{eq:secondadv}
	\begin{empheq}{align}
	\label{eq:secondadv_a}
    & \norma{\hat v_h}_V \apprle\left\|\hat{v}\right\|_{V}, \\
    \label{eq:secondadv_b}
    &\norma{\mu_h - \mu}_{H^{-\nicefrac{3}{2}}(\Gamma)}\apprle h \norma{\hat v}_V, \\
    \label{eq:fsecondadv_c} &\left\|v_{h}-v\right\|_{L^2(\Omega)} \apprle h \norma{\hat v}_V.
	\end{empheq}
\end{subequations}
\end{Lemma}
\begin{proof}
Let consider $\hat v_{\0} = (v,\mu_0) := (v, \mu - \bar \mu)$, with $\bar \mu = \nicefrac{\langle \mu,1 \rangle}{\langle 1,1 \rangle}$, and $\hat w_{\0 h} := (w_h, \eta_h- \bar \eta_h) \in \tilde V_h^k$. Then we have
\begin{equation*}
\B_0(\hat w_h,\hat v) = \B_0(\hat w_{\0 h},\hat v) + \B_0((0,\bar \eta_h),\hat v) = \B_0(\hat w_{\0 h},\hat v_{\0}) + \B_0((0,\bar \eta_h),\hat v) + \B_0(\hat w_{\0 h},(0,\bar \mu)).
\end{equation*}
According to Lemma \ref{Lemma:dualpb} applied to the first term at the right hand side of the above equality, there exists a unique $\hat v_{\0 h} = (v_h, {\mu_{\0 h}}) \in \tilde V_h^k$ such that
\begin{equation*}
\B_0(\hat w_h,\hat v) = \B_{0,h}(\hat w_{\0 h},\hat v_{\0 h}) + \B_0((0,\bar \eta_h),\hat v) + \B_0(\hat w_{\0 h},(0,\bar \mu)).
\end{equation*}
Using \ref{H3.c} with $\hat{q} = (0,\bar \mu)$, we can write
\begin{equation*}
\B_0(\hat w_h,\hat v) = \B_{0,h}(\hat w_{\0 h},\hat v_{\0 h}) + \B_0((0,\bar \eta_h),\hat v) + \B_{0,h}(\hat w_{\0 h},(0,\bar \mu)) = \B_{0,h}(\hat w_{\0 h},\hat v_h) + \B_0((0,\bar \eta_h),\hat v),
\end{equation*}
where we have set $\hat v_h = (v_h,\mu_h) := (v_h, {\mu_{\0 h}}+\bar \mu) \in V_h^k$.
Moreover, by adding and subtracting in this latter the term $\B_0((0,\bar \eta_h),\hat v_h)$ and using \ref{H2.a}, we get
\begin{equation*}
\B_0(\hat w_h,\hat v) = \B_{0,h}(\hat w_h,\hat v_h) + \B_0((0,\bar \eta_h),\hat v - \hat v_h).
\end{equation*}
By applying the Cauchy-Schwarz inequality to estimate the term $$\norma{\bar\mu}_{H^{-\nicefrac{1}{2}}(\Gamma)}\apprle \norma{\mu}_{H^{-\nicefrac{1}{2}}(\Gamma)},$$ and  \eqref{eq:firstadv_a}, we prove \eqref{eq:secondadv_a} as follows:
$$\norma{\hat v_h}_V  = \norma{\hat v_{\0 h} + (0,\bar\mu)}_V \apprle \norma{\hat v_{\0}}_V + \norma{\bar\mu}_{H^{-\nicefrac{1}{2}}(\Gamma)} \apprle
\left\|\hat v\right\|_{V} + \norma{\mu}_{H^{-\nicefrac{1}{2}}(\Gamma)} \apprle\left\|\hat{v}\right\|_{V}.$$
Finally, by using \eqref{eq:firstadv_b}-\eqref{eq:firstadv_c} we easily prove \eqref{eq:secondadv_b} and \eqref{eq:fsecondadv_c}:
	\begin{align*} 
    &\norma{\mu_h - \mu}_{H^{-\nicefrac{3}{2}}(\Gamma)} = \left\|\mu_{\0 h}-\mu_0\right\|_{H^{-\nicefrac{3}{2}}(\Gamma)}\apprle h\left\|\hat v_{\0}\right\|_{V} \apprle h \norma{\hat v}_V, \\ \nonumber &\left\|v_{h}-v\right\|_{L^2(\Omega)}\apprle h\left\|\hat v_{\0}\right\|_V \apprle h \norma{\hat v}_V.
\end{align*}   \qedhere	
\end{proof} 
%
%
%
%
%
\subsection{Error estimate in the energy norm}
In the present section we show the validity of the inf-sup condition for the discrete bilinear form $\A_{\kappa,h}$, with $\kappa>0$, and we prove that the discrete scheme has the optimal approximation order, providing for the optimal error estimate in the $V$-norm.
\begin{Theorem}\label{th:infsup}
Assume that $\kappa^2$ is not an eigenvalue of the Laplacian in $\Omega$ endowed with a Dirichlet boundary condition on $\Gamma$. Then, for $h$ small enough,
\begin{equation*}
\sup_{\substack{\hat v_{h} \in V_h	\\	
	 \hat v_{h}\ne 0}} \frac{\A_{\kappa,h}(\hat w_h , \hat v_h)}{\norma{\hat v_h}_V} \apprge \norma{\hat w_h}_{V} \quad \forall \, \hat w_h \in V_h.
\end{equation*}
\end{Theorem}
\begin{proof}
Given $\hat w_h \in V_h^k$, let $\hat v := \A_{\kappa}^{*-1} J \hat w_h \in V$ where $J: V \to V'$ denotes the canonical continuous map $(J \hat w)(\hat z) := (\hat w,\hat z)_{V}$.
Therefore we can write: 
\begin{align} \nonumber
\A_{\kappa}(\hat z, \hat v) & = \A_{\kappa}(\hat z, \A_{\kappa}^{*-1} J \hat w_h) = (\A_{\kappa}\hat z)(\A_{\kappa}^{*-1} J \hat w_h) = (\A_{\kappa}^*\A_{\kappa}^{*-1} J \hat w_h)(\hat z) \\ & = (J \hat w_h)(\hat z) = (\hat w_h,\hat z)_{V}, \qquad \forall \, \hat z \in V. \label{eq:infsup1}
\end{align}
Moreover, according to the continuity of $\mathcal{A}_{\kappa}^{*-1}$ (see Remark \ref{rk:Astar_cont}) and of $J$, we obtain 
\begin{equation} \label{eq:infsup2}
\norma{\hat v}_V \apprle \norma{\hat w_h}_V.
\end{equation}
Now, by virtue of Lemma \ref{Lemma:ourdualpb}, writing $\hat v = (v,\mu) \in V$, there exists $\hat v_h = (v_h,\mu_h) \in V_h^k$  such that 
\begin{equation}\label{eq:Boh_utile2}
\B_{0,h}(\hat w_h,\hat v_h) = \B_0(\hat w_h,\hat v) + \B_0((0,\bar \eta_h),\hat v_h - \hat v) \quad \forall \, \hat w_h = (w_h,\eta_h) \in V_h^k
\end{equation}
where $\bar \eta_h = \nicefrac{\langle \eta_h,1 \rangle}{\langle 1,1 \rangle}$ and such that \eqref{eq:secondadv_a}-\eqref{eq:fsecondadv_c} hold.
Proceeding as in Theorem 5.2 of \cite{GaticaMeddahi2020}, and recalling the
definitions of $\A_{0,h}$ and $\A_{\kappa,h}$ and of $\D_{\kappa,h}$ and $\D_{\kappa}$ (see assumption \ref{H3.a} and Remark \ref{rk:Dk}), we rewrite $\A_{\kappa,h}$ as follows:
\begin{align}
\A_{\kappa,h} &= \A_{0,h} + (\A_k - \A_0) + (\A_0 - \A_{0,h}) + (\A_{\kappa,h} - \A_{\kappa}) \nonumber\\
& =\A_{0,h} + (\A_k - \A_0) +  (\B_0 - \B_{0,h}) +  (\B_{\kappa,h} -\B_{\kappa})\nonumber\\
& =  \A_{0,h} + (\A_k - \A_0) + (\D_{\kappa,h}- \D_{\kappa}).  \label{eq:infsup4}
\end{align}
Using Lemma \ref{Lemma:BBwithinter} with $s=0$, we have
\begin{equation}\label{eq:Dk}
\abs*{(\D_{\kappa,h} - \D_\kappa)(\hat w_h, \hat v_h)}
 \apprle h \norma{\hat w_h}_V \norma{\hat v_h}_W \apprle h \norma{\hat w_h}_V \norma{\hat v_h}_V. 
\end{equation}
%
Recalling \eqref{definition_A_kappa}--\eqref{definition_K_kappa} and \eqref{model_problem_galerkin}, and using \eqref{eq:Boh_utile2}, we get:
\begin{align} \nonumber
\A_{0,h}(\hat w_{h}, \hat v_{h}) & = \B_{0,h}(\hat w_{h}, \hat v_{h}) + \K_0(\hat w_{h}, \hat v_{h}) = \B_0(\hat w_{h}, \hat v)+ \B_0((0,\bar \eta_h),\hat v_h - \hat v)
+ \K_0(\hat w_{h},\hat v_{h})\nonumber\\ 
& = \B_0(\hat w_{h}, \hat v)+ \K_0(\hat w_{h},\hat v) + \B_0((0,\bar \eta_h),\hat v_h - \hat v)
+ \K_0(\hat w_{h},\hat v_{h}) - \K_0(\hat w_{h},\hat v) \nonumber\\
& =  \A_0(\hat w_{h}, \hat v) + \B_0((0,\bar \eta_h),\hat v_h - \hat v) - 2 \langle \mu_{h} - \mu,\text{K}_0 w_h \rangle_{\Gamma}.  \label{eq:JN4.14_bis}  
\end{align}
By applying the Hölder inequality and \eqref{eq:secondadv_b}, we can estimate the last term in \eqref{eq:JN4.14_bis} as follows
\begin{align} \nonumber
\abs*{\langle \mu_{h} - \mu,\text{K}_0  w_h \rangle_{H^{-\nicefrac{3}{2}}(\Gamma) \times H^{\nicefrac{3}{2}}(\Gamma)}} & \apprle  \norma{\mu_{h} - \mu}_{H^{-\nicefrac{3}{2}}(\Gamma)} \norma{\text{K}_0  w_h}_{H^{\nicefrac{3}{2}}(\Gamma)} \\ & \apprle h \norma{\hat v}_V \norma{\text{K}_0  w_h}_{H^{\nicefrac{3}{2}}(\Gamma)}. \label{eq:JN4.15}
\end{align}
Then, using the continuity of $\text{K}_0 : H^{\nicefrac{1}{2}}(\Gamma) \to H^{\nicefrac{3}{2}}(\Gamma)$ (see \cite{JohnsonNedelec1980}, formula (2.11)) and the trace theorem, we obtain
\begin{equation*} %
\norm{\text{K}_0 w_h}_{H^{\nicefrac{3}{2}}(\Gamma)} \apprle \norm{ w_h}_{H^{\nicefrac{1}{2}}(\Gamma)} \apprle \norma{w_h}_{H^1(\Omega)} \le \norma{\hat w_{h}}_V,
\end{equation*}
and, hence, combining this latter with \eqref{eq:JN4.15}, it follows that
\begin{equation} \label{eq:JN4.16_bis}
\abs*{\langle \mu_{h} - \mu,\text{K}_0 w_h \rangle_{H^{-\nicefrac{3}{2}}(\Gamma) \times H^{\nicefrac{3}{2}}(\Gamma)}} \apprle h \norma{\hat v}_V \norma{\hat w_{h}}_V.
\end{equation}
Then, from \eqref{eq:JN4.14_bis}  and \eqref{eq:JN4.16_bis}, we obtain
\begin{align}\label{eq:stima_A0h}
\A_{0,h}(\hat w_h, \hat v_h) & \apprge  \A_0(\hat w_h, \hat v) +  \B_0((0,\bar \eta_h),\hat v_h - \hat v) - h \norma{\hat v}_V \norma{\hat w_h}_V.
\end{align}
By explicitly writing
\begin{align*}
\B_0((0,\bar \eta_h),\hat v_h - \hat v) &= -\langle \bar \eta_h, v_h-v \rangle_\Gamma + 2 \langle \mu_h - \mu, \text{V}_0 \bar \eta_h \rangle_\Gamma \\
& = -\bar \eta_h \langle 1,  v_h-v \rangle_\Gamma + 2\bar \eta_h \langle \mu_h - \mu, \text{V}_0 1 \rangle_\Gamma=: I + II,
\end{align*}
and using the Cauchy-Schwarz inequality to bound $\abs*{\bar \eta_h} \apprle \norma{\eta_h}_{H^{-\nicefrac{1}{2}}}$, 
we can estimate $II$ by using Hölder inequality, the continuity of $\text{V}_0 : H^{\nicefrac{1}{2}}(\Gamma) \to H^{\nicefrac{3}{2}}(\Gamma)$ (see Remark \ref{rk:reg_V0}) and \eqref{eq:secondadv_b}: 
\begin{align*}
\abs*{II} &\apprle \norma{\eta_h}_{H^{-\nicefrac{1}{2}}(\Gamma)}\norma{\mu_h - \mu}_{H^{-\nicefrac{3}{2}}(\Gamma)} \norma{V_0 1 }_{H^{\nicefrac{3}{2}}(\Gamma)} 
 \apprle h \norma{\hat v}_V \norma{\hat w_h}_V.
\end{align*}
To estimate the term $I$, we use Hölder inequality and the trace theorem (see e.g. Eq. (2.1) of \cite{Costabel1988}) and we obtain, for $0 < \varepsilon < \nicefrac{1}{2}$:
\begin{align*}
\abs*{I} &
\apprle \norma{\eta_h}_{H^{-\nicefrac{1}{2}}(\Gamma)} \abs*{\langle 1,  v_h-v \rangle_\Gamma} 
\apprle \norma{\hat w_h}_{V} \abs*{\langle 1, v_h-v \rangle_\Gamma} \apprle \norma{\hat w_h}_{V} \norma{v_h-v}_{H^\varepsilon(\Gamma)} \\
&\apprle \norma{\hat w_h}_{V}  \norma{v_h-v}_{H^{\nicefrac{1}{2}+\varepsilon}(\Omega)}.
\end{align*}
Then, using the characterization of the fractional Sobolev space $H^{\nicefrac{1}{2}+\varepsilon}(\Omega)$ as the real interpolation between 
$L^2(\Omega)$ and $H^1(\Omega)$, by a standard result concerning the norm of real interpolation spaces (see Prop. 2.3 of \cite{LionsMagenes1968}),	 
it holds that $\norma{v_h-v}_{H^{\nicefrac{1}{2}+\varepsilon}(\Omega)}\leq \norma{v_h-v}^{\nicefrac{1}{2}-\varepsilon}_{L^{2}(\Omega)}\norma{v_h-v}^{\nicefrac{1}{2}+\varepsilon}_{H^1(\Omega)}$. Hence, by applying \eqref{eq:secondadv_a} and \eqref{eq:fsecondadv_c}, we finally get:
\begin{align*}
\abs*{I} &\apprle h^{\nicefrac{1}{2}-\varepsilon}\norma{\hat v}_V^{\nicefrac{1}{2}-\varepsilon}\norma{v_h-v}^{\nicefrac{1}{2}+\varepsilon}_{H^1(\Omega)} \norma{\hat w_h}_{V}  \apprle h^{\nicefrac{1}{2}-\varepsilon}\norma{\hat v}_V^{\nicefrac{1}{2}-\varepsilon}(\norma{v}_{H^1(\Omega)} + \norma{v_h}_{H^1(\Omega)})^{\nicefrac{1}{2}+\varepsilon} \norma{\hat w_h}_{V} \\
& \apprle h^{\nicefrac{1}{2}-\varepsilon}\norma{\hat v}_V^{\nicefrac{1}{2}-\varepsilon}(\norma{\hat v}_{V} + \norma{\hat v_h}_{V})^{\nicefrac{1}{2}+\varepsilon} \norma{\hat w_h}_{V} \apprle h^{\nicefrac{1}{2}-\varepsilon}\norma{\hat v}_V^{\nicefrac{1}{2}-\varepsilon}\norma{\hat v}_{V}^{\nicefrac{1}{2}+\varepsilon} \norma{\hat w_h}_{V} = h^{\nicefrac{1}{2}-\varepsilon}\norma{\hat v}_V \norma{\hat w_h}_{V}.
\end{align*}
Combining \eqref{eq:stima_A0h} with the bounds for $I$ and $II$, we can write
\begin{align}\label{eq:stima_A0h_bis}
\A_{0,h}(\hat w_h, \hat v_h) & \apprge  \A_0(\hat w_h, \hat v) - h^{\nicefrac{1}{2}-\varepsilon}\norma{\hat v}_V \norma{\hat w_h}_{V}
- h \norma{\hat v}_V \norma{\hat w_h}_V \nonumber\\
&\apprge  \A_0(\hat w_h, \hat v) - h^{\nicefrac{1}{2}-\varepsilon}\norma{\hat v}_V \norma{\hat w_h}_{V}.
\end{align}
Starting from \eqref{eq:infsup4}, using \eqref{eq:Dk} and \eqref{eq:stima_A0h_bis}, it follows  
\begin{align}\label{eq:stima_Akh}
\A_{\kappa,h}(\hat w_h,\hat v_h) & \apprge \A_{0,h}(\hat w_h,\hat v_h) + (\A_k - \A_0)(\hat w_h,\hat v_h) - h \norma{\hat w_h}_V \norma{\hat v_h}_V \nonumber\\ &
\apprge \A_0(\hat w_h,\hat v) - h^{\nicefrac{1}{2}-\varepsilon} \norma{\hat w_h}_V \norma{\hat v}_V + (\A_k - \A_0)(\hat w_h,\hat v_h) \nonumber\\
& = \A_{\kappa}(\hat w_h,\hat v) - h^{\nicefrac{1}{2}-\varepsilon} \norma{\hat v}_V \norma{\hat w_h}_V +  (\A_{\kappa} - \A_0)(\hat w_h,\hat v_h) + (\A_0 - \A_{\kappa})(\hat w_h,\hat v)
\nonumber\\ 
& = \norma{\hat w_h}_V^2 - h^{\nicefrac{1}{2}-\varepsilon} \norma{\hat v}_V \norma{\hat w_h}_V + (\A_{\kappa} - \A_0)(\hat w_h, \hat v_h - \hat v) 
\end{align}
having used \eqref{eq:infsup1} in the last equality.

Concerning the last term in \eqref{eq:stima_Akh}, we explicitly write:
\begin{equation*}
(\A_{\kappa} - \A_0)(\hat w_h, \hat v_h - \hat v)= 
-\kappa^2 m(w_h,v_h - v) + 2 \langle \mu_h - \mu, (\text{V}_{\kappa} - \text{V}_0)\eta_h - (\text{K}_{\kappa} - \text{K}_0) w_h \rangle
\end{equation*}
and, by using  
the continuity of $m$, the Hölder inequality and the continuity of
$\text{V}_{\kappa} - \text{V}_0 : H^{-\nicefrac{1}{2}}(\Gamma) \to H^{\nicefrac{3}{2}}(\Gamma)$ and of $\text{K}_{\kappa} - \text{K}_0 : H^{\nicefrac{1}{2}}(\Gamma) \to H^{\nicefrac{3}{2}}(\Gamma)$ (see Lemma \ref{Lemma:Vcompa}),
\begin{align}\label{eq:stima_Akh_bis}
\abs*{(\A_{\kappa} - \A_0)(\hat w_h, \hat v_h - \hat v)} & \apprle 
\norma{\hat w_h}_V \norma{v_h - v}_{L^2(\Omega)}+  \norma{\hat w_h}_V\norma{\mu_h - \mu}_{H^{-\nicefrac{3}{2}}(\Gamma)} 
\apprle h \norma{\hat w_h}_V \norma{\hat v_h}_V,
\end{align}
having used, in the last bound, \eqref{eq:secondadv_b} and \eqref{eq:fsecondadv_c}.
Finally, combining \eqref{eq:stima_Akh} with \eqref{eq:stima_Akh_bis} and \eqref{eq:infsup2} we get
\begin{equation*}
\A_{\kappa,h}(\hat w_h,\hat v_h) \apprge 
\norma{\hat w_h}_V^2 - h^{\nicefrac{1}{2}-\varepsilon} \norma{\hat w_h}_V \norma{\hat v}_V \apprge (1-h^{\nicefrac{1}{2}-\varepsilon})\norma{\hat w_h}_V^2 
\end{equation*}
from which, for $h$ small enough, the claim follows.
\end{proof}
We conclude this section by proving the convergence error estimate for Problem \eqref{model_problem_galerkin}.
\begin{Theorem}
\label{Theorem:JN1}
Assume that $\kappa^2$ is not an eigenvalue of the Laplacian in $\Omega$ endowed with a Dirichlet boundary condition on $\Gamma$.
Furthermore, assume that there exist $k \in \mathbf{N}$ such that for all $1 \le s \le k$ and $\kappa > 0$, Assumptions \ref{H1.a}-\ref{H1.c}, \ref{H2.a}, \ref{H2.b}, \ref{H3.a}-\ref{H3.c} hold, and  $\sigma : L^2(\Omega) \to \mathbf{R}^+$ such that  
\begin{enumerate}[label=(\subscript{H4}{{\alph*}})]
    {\setlength\itemindent{50pt} \item\label{H4.a} $\left|\mathcal{L}_{f}(\hat{v}_{h}) - \mathcal{L}_{f,h}(\hat{v}_{h})\right| \apprle h^{s}\left\|\hat{v}_{h}\right\|_{V} \, \sigma(f) \qquad \forall \, \hat{v}_{h} \in V_{h}^{k}$}.
\end{enumerate}
Then, for $h$ small enough, Problem \eqref{model_problem_galerkin} admits a unique solution $\hat u_h \in V_h^k$ and if $\hat u$, solution of Problem \eqref{model_problem_operator}, satisfies $\hat u \in H^{s+1}(\Omega) \times H^{s-\nicefrac{1}{2}}(\Gamma)$, it holds
\begin{equation*}
\norma{\hat u - \hat u_h}_V \apprle h^s \left( \norma{u}_{H^{s+1}(\Omega)} + \sigma(f)\right).
\end{equation*}
\end{Theorem}
\begin{proof}
Existence and uniqueness of $\hat u_h$ follows from the discrete inf-sup condition of Theorem \ref{th:infsup}.
Let $\hat u_h^I \in V_h^k$ be the interpolant of $\hat u$. By virtue of Theorem \ref{th:infsup} there exists $\hat v_h^* = (v_h^*,\mu_h^*) \in V_h^k$ such that
\begin{equation*}
\norma{\hat u_h - \hat u_h^I}_V \apprle \frac{\A_{\kappa,h}(\hat u_h - \hat u_h^I , \hat v_h^*)}{\norma{\hat v_h^*}_V}.
\end{equation*}
Since $\hat u$ and $\hat u_h$ are solution of \eqref{model_problem_operator} and \eqref{model_problem_galerkin} respectively, we have
\begin{align*}
\norma{\hat u_h - \hat u_h^I}_V \norma{\hat v_h^*}_V & \apprle \A_{\kappa,h}(\hat u_h - \hat u_h^I , \hat v^*_h) = \A_{\kappa,h}(\hat u_h, \hat v^*_h) - \A_{\kappa,h}(\hat u_h^I,\hat v^*_h) 
 \\ & = \mathcal{L}_{f,h}(\hat v^*_h) - \A_{\kappa,h}(\hat u_h^I , \hat v^*_h) +  
 [\A_{\kappa}(\hat u,\hat v^*_h) - \mathcal{L}_{f}(\hat v^*_h)] 
 \\ & = [\mathcal{L}_{f,h}(\hat v^*_h)- \mathcal{L}_{f}(\hat v^*_h)] + \A_{\kappa}(\hat u - \hat u_h^I,\hat v^*_h) + [\A_{\kappa}(\hat u_h^I,\hat v^*_h) - \A_{\kappa,h}(\hat u_h^I , \hat v^*_h)]\\
 &= [\mathcal{L}_{f,h}(\hat v^*_h)- \mathcal{L}_{f}(\hat v^*_h)] + \A_{\kappa}(\hat u - \hat u_h^I,\hat v^*_h) + [(\B_{\kappa} - \B_{\kappa,h})(\hat u_h^I, \hat v^*_h)].
\end{align*}
Then, by using Assumption \ref{H4.a}, the continuity of $\A_{\kappa}$ and Lemma \ref{Lemma:Bwithinter}, we obtain
\begin{align*}
\norma{\hat u_h - \hat u_h^I}_V \norma{\hat v_h^*}_V & \apprle h^{s}\left\|\hat v^*_{h}\right\|_{V} \, \sigma(f) + \norma{\hat u - \hat u_h^I}_V \norma{\hat v_h^*}_V +  h^s \norma{u}_{H^{s+1}(\Omega)} \norma{\hat v_h^*}_V,
\end{align*}
from which it easily follows
\begin{equation} \label{eq:JN1.2}
\norma{\hat u - \hat u_h}_V \le \norma{\hat u -\hat u_h^I}_V + \norma{\hat u_h - \hat u_h^I}_V \apprle \norma{\hat u - \hat u_h^I}_V + h^s \norma{u}_{H^{s+1}(\Omega)} + h^s \sigma(f).
\end{equation}
Finally, combining \eqref{int_property} and \eqref{eq:JN1.2} we obtain the thesis.
\end{proof}


\section{The discrete scheme}\label{sec_5_discrete_scheme}
In this section we introduce the discrete CVEM-BEM scheme for the coupling procedure (\ref{model_problem_variational}). We start by briefly describing the main tools of the VEM; we refer the interested reader to \cite{AhmadAlsaediBrezziMariniRusso2013, BeiraoBrezziCangianiManziniMariniRusso2013, BeiraoRussoVacca2019} for a deeper presentation. In what follows, we denote by $\mathbf{V}_{1},\ldots,\mathbf{V}_{n_{\text{\tiny{E}}}}$ the $n_{\text{\tiny{E}}}$ vertices of 
an element $E\in\mathcal{T}_{h}$ and by $e_{1},\ldots,e_{n_{\text{\tiny{E}}}}$ the edges of its boundary $\partial E$. 
For simplicity of presentation, we assume that at most one edge is curved while the remaining ones are straight. We identify the curved edge by $e_{1}\subset\partial\Omega$, to which we associate a regular invertible parametrization $\gamma_{E}:I_{E}\rightarrow e_{1}$, where $I_{E}\subset\mathbf{R}$ is a closed interval. Furthermore, we denote by $\mathbf{V}_{E}$, $h_{E}$ and $|E|$ the mass center, the diameter and the Lebesgue measure of $E$, respectively. Additionally, we call $N_{\text{\tiny{V}}}$ and $N_{\text{{e}}}$ the numbers of total vertices and edges of $\mathcal{T}_{h}$, respectively.

\

In what follows we will show that all the assumptions, used to obtain the theoretical results in Section \ref{sec:DVF}, are satisfied.

\subsection{The discrete spaces $Q_{h}^{k}$, $X_{h}^{k}$ and $\widetilde{X}_{h}^{k}$: validity of Assumptions \ref{H1.a}--\ref{H1.c}}
\noindent
In order to describe the discrete space $Q_{h}^{k}$, introduced in Section \ref{sec:DVF}
in a generic setting, 
we preliminarily consider for each $E\in\mathcal{T}_{h}$ the following local finite dimensional \emph{augmented} virtual space $\widetilde{Q}^{k}_{h}(E)$ and the local \emph{enhanced} virtual space $Q^{k}_{h}(E)$ defined respectively, 
\begin{equation*}\label{local_VEM_space}
\widetilde{Q}^{k}_{h}(E):=\left\{v_{h}\in H^{1}(E)\cap C^{0}(E) \ : \ \Delta v_{h}\in P_{k}(E),  \ v_{h}\,\raisebox{-.5em}{$\vert_{e_{1}}$}\in\widetilde{P}_{k}(e_{1}), \ v_{h}\,\raisebox{-.5em}{$\vert_{e_{i}}$}\in P_{k}(e_{i}) \ \forall\, i=2,\ldots,n_{\text{\tiny{E}}}\right\}
\end{equation*}
and 
\begin{equation*}\label{local_enhanced_VEM_space}
Q^{k}_{h}(E):=\left\{v_{h}\in\widetilde{Q}^{k}_{h}(E) \ : \ m^{\text{\tiny{E}}}\left(\Pi_{k}^{\nabla,E}v_{h},q\right)=m^{\text{\tiny{E}}}\left(v_{h},q\right) \ \forall\, q\in P_{k}(E)/P_{k-2}(E)\right\},
\end{equation*}
where $\widetilde{P}_{k}(e_{1}):=\left\{\widetilde{q}:=q\circ\gamma_{E}^{-1} \ : \ q\in P_{k}(I_E)\right\}$ and  $P_{k}(E)/P_{k-2}(E)$ stands for the space of all polynomials of degree $k$ on $E$ that are $L^{2}$-orthogonal to all polynomials of degree $k-2$ on $E$.

For details on such spaces, we refer the reader 
to \cite{BeiraoRussoVacca2019} (see Remark 2.6)   and to \cite{AhmadAlsaediBrezziMariniRusso2013}
(see Section 3).\\ 

\noindent
It is possible to prove (see Proposition 2 in  \cite{AhmadAlsaediBrezziMariniRusso2013} and Proposition 2.2 in \cite{BeiraoRussoVacca2019})
that the dimension of $Q^{k}_{h}(E)$ is
\begin{equation*}\label{dim_local_VEM_space}
n:=\dim(Q^{k}_{h}(E))=kn_{\text{\tiny{E}}}+\frac{k(k-1)}{2}
\end{equation*}
and that a generic element $v_{h}$ of $Q^{k}_{h}(E)$ is uniquely determined by the following $n$ conditions (see \cite{BeiraoRussoVacca2019}, Proposition 2.2):
\begin{itemize}
\item its values at the $n_{\text{\tiny{E}}}$ vertices of $E$;
\item its values at the $k-1$ internal points of the $(k+1)-$point Gauss-Lobatto quadrature rule on every straight edge $e_{2},\ldots,e_{n_{E}} \in \partial E$; 
\item its values at the $k-1$ internal points of $e_{1}$ that are the images, through $\gamma_E$, of the $(k+1)-$point Gauss-Lobatto quadrature rule on $I_E$; 
\item the internal $k(k-1)/2$ moments of $v_h$ against a polynomial basis $\mathcal{M}_{k-2}(E)$ of $P_{k-2}(E)$ defined for $k\geq 2$, as:
\begin{equation}\label{moments}
\frac{1}{|E|}\int\limits_{E}v_{h}(\mathbf{x})p(\mathbf{x})\,\dd\mathbf{x} \qquad \forall\, p\in \mathcal{M}_{k-2}(E)\ \text{with} \ \|p\|_{L^{\infty}(E)} \apprle 1.
\end{equation}
\end{itemize}
According to the definition of $\widetilde{Q}^{k}_{h}(E)$,  it is easy to check that $P_{0}(E)\subset Q^{k}_{h}(E)$ while, in general, $P_{k}(E)\not\subset Q^{k}_{h}(E)$, for $k>0$. Now, choosing an arbitrary but fixed ordering of the degrees of freedom such that these are indexed by $i=1,\cdots,n$, we introduce as in \cite{BeiraoBrezziCangianiManziniMariniRusso2013} the operator $\text{dof}_{i}:Q^{k}_{h}(E) \longrightarrow \mathbf{C}$, defined as 
$$\text{dof}_{i}(v_{h}):= \text{the value of the $i$-th local degree of freedom of}\, v_{h}.$$
The basis functions $\left\{\Phi_{j}\right\}_{j=1}^n$ chosen for the space $Q^{k}_{h}(E)$ are the standard Lagrangian ones, such that
\begin{equation*}\label{basis_VEM}
\text{dof}_{i}(\Phi_{j})=\delta_{ij}, \qquad i,j=1,\ldots,n,
\end{equation*}
\noindent
$\delta_{ij}$ being the Kronecker delta.\\

\noindent
On the basis of the definition of the local enhanced virtual space $Q^{k}_{h}(E)$, we are allowed to construct the global enhanced virtual space 
\begin{equation*}\label{global_VEM_space}
Q^{k}_{h}:=\left\{v_{h}\in H^{1}_{0,\Gamma_0}(\Omega) \ : \ v_{h}\,\raisebox{-.5em}{$\vert_{E}$}\in Q^{k}_{h}(E)\quad \forall E\in\mathcal{T}_h\right\}.
\end{equation*}
\noindent
We remark that the global degrees of freedom for a generic element $v_{h}\in Q^{k}_{h}$ are:
\begin{itemize}
\item its values at each of the $\widetilde{N}_{\text{\tiny{V}}}$ vertices of $\mathcal{T}_{h}$ that do not belong to $\Gamma_0$;
\item its values at the $k-1$ internal points of the $(k+1)-$point Gauss-Lobatto quadrature rule on each of the $\bar{N}_{\text{\tiny{e}}}$ straight edges of $\mathcal{T}_{h}$; 
\item its values at the $k-1$ internal points of the $\widetilde{N}_{\text{\tiny{e}}}$ curved edge of $\mathcal{T}_{h}$, that  do not belong to $\Gamma_0$ and that are the images through the parametrization $\gamma_E$ of the the $(k+1)-$point Gauss-Lobatto quadrature rule on parametric interval $I_E$; 
\item its moments up to order $k-2$ in each of the $N_{\text{\tiny{E}}}$ elements of $\mathcal{T}_{h}$, for $k\geq 2$:
\begin{equation*}\label{moments_global}
\frac{1}{|E|}\int\limits_{E}v_{h}(\mathbf{x})p(\mathbf{x})\,\dd\mathbf{x} \qquad \forall\, p\in \mathcal{M}_{k-2}(E)(E) \ \text{with} \ \|p\|_{L^{\infty}(E)} \apprle 1.
\end{equation*}
\end{itemize}
\noindent
Consequently, $Q^{k}_{h}$ has dimension
\begin{equation}\label{global_dimension}
N:=\dim(Q^{k}_{h})=\widetilde{N}_{\text{\tiny{V}}}+(k-1)(\widetilde{N}_{\text{\tiny{e}}}+\bar{N}_{\text{\tiny{e}}})+\frac{k(k-1)}{2}N_{\text{\tiny{E}}}.
\end{equation}
\noindent
\begin{remark}
We remark that the global enhanced virtual space $Q^{k}_{h}$ defined above is slightly different from that introduced in the pioneering paper on CVEM \cite{BeiraoRussoVacca2019}, the latter being defined for the solution of the Laplace problem. However, as highlighted in Remark 2.6 in \cite{BeiraoRussoVacca2019}, the theoretical analysis therein contained can be extended to our context by following the ideas of  \cite{AhmadAlsaediBrezziMariniRusso2013}.
\end{remark}
In the following lemma we prove that Assumption \ref{H1.a} holds for the space $Q_{h}^{k}$.
\begin{Lemma}\label{Proposition:NewBRRV3.1}
Let $v\in H^{s+1}(\Omega)$ with $\nicefrac{1}{2} <s\leq k$. Then
\begin{equation*}
\inf_{v_h\in Q_{h}^{k}}\|v-v_{h}\|_{H^{1}(\Omega)} \apprle h^{s} \|v\|_{H^{s+1}(\Omega)}.
\end{equation*}
\end{Lemma}
\begin{proof}
Let $E$ be an element of $\mathcal{T}_{h}$. By virtue of Theorem 3.7 in \cite{BeiraoRussoVacca2019} there exists $v_h^I$, interpolant of $v$ in $Q_h^k$, such that 
\begin{equation*}
|v - v_{h}^I|_{H^{1}(E)}\apprle h_{E}^{s}\|v\|_{H^{s+1}(E)}.
\end{equation*}
Moreover, by using the Poincaré-Friedrichs inequality (see (2.11) in \cite{BrennerGuanSung2017}), we can write 
\begin{equation*}
\|v -v_{h}^I\|_{L^{2}(E)}\apprle h_{E}|v - v_{h}^I|_{H^{1}(E)}+\left|\int_{\partial E}\left[v(\mathbf{x})-v_{h}^I(\mathbf{x})\right]ds\right|\apprle h_{E}^{s+1}\|v\|_{H^{s+1}(E)}+\int_{\partial E}\left|v(\mathbf{x})-v_{h}^I(\mathbf{x})\right|ds.
\end{equation*}
Then, by applying the Hölder inequality, Lemma 3.2 and (3.20) in  \cite{BeiraoRussoVacca2019}, we can estimate the second term at the right hand side of the above inequality as follows
\begin{align*}
\int_{\partial E}\left|v(\mathbf{x})-v_{h}^I(\mathbf{x})\right|ds&=\sum_{e\subset\partial E}\int_{e}\left|v(\mathbf{x})-v_{h}^I(\mathbf{x})\right|ds\leq\sum_{e\subset\partial E}|e|^{\nicefrac{1}{2}}\|v-v_{h}^I\|_{L^{2}(e)}\leq h_{E}^{\nicefrac{1}{2}}\sum_{e\subset\partial E}\|v-v_{h}^I\|_{L^{2}(e)}\\
&\apprle h_{E}^{\nicefrac{1}{2}}\sum_{e\subset\partial E}h_{E}^{s+\nicefrac{1}{2}}\|v\|_{H^{s+\nicefrac{1}{2}}(e)} \apprle h_{E}^{s+1}\|v\|_{H^{s+1}(E)}.
\end{align*}
%
%
%
%
\noindent
Combining the local bounds for the $L^{2}$-norm and for the $H^{1}$-seminorm of $v-v_{h}^I$ on $E$, we obtain 
\begin{equation*}
\|v -v_{h}^I\|_{H^{1}(\Omega)}\apprle h^{s}\|v\|_{H^{s+1}(\Omega)},
\end{equation*}
from which the thesis easily follows.
\end{proof}

%
Finally, we introduce the boundary element space $X_{h}^{k}$ associated to the artificial boundary $\Gamma$
\begin{equation*}
X_{h}^{k}:=\left\{\lambda \in L^{2}(\Gamma) \ : \lambda_{\mkern 1mu \vrule height 2ex\mkern2mu e}\in\widetilde{P}_{k}(e), \ \forall e\in\Gamma\right\} \qquad \text{with} \qquad |e|<h,
\end{equation*}
where $|e|$ denotes the length of the edge $e$. By virtue of Theorem 4.3.20 in \cite{SauterSchwab2011}, we have that the space $X_{h}^{k}$ satisfies the interpolation property \ref{H1.b}. For what concerns the space $\widetilde{X}_{h}^{k}=X_{h}^{k}\cap H_{0}^{-\nicefrac{1}{2}}(\Gamma)$ and the corresponding hypothesis \ref{H1.c}, we refer to (3.2b) in  \cite{JohnsonNedelec1980}.
Moreover, a natural basis for the space $X_{h}^{k}$ consists in the choice of the functions $\Phi_{j_{{|_\Gamma}}}$, which are the restriction of $\Phi_{j}$ on $\Gamma$.

\subsection{The discrete bilinear forms $\mathcal{A}_{\kappa,h}$ and $\mathcal{B}_{\kappa,h}$: validity of Assumptions \ref{H2.a}, \ref{H2.b} and \ref{H3.a}--\ref{H3.c}}

\noindent
\noindent
In order to define computable discrete local bilinear forms $a^{\text{\tiny{E}}}_{h}:Q^{k}_{h}(E)\times Q^{k}_{h}(E)\rightarrow\mathbf{C}$ and $m^{\text{\tiny{E}}}_{h}:Q^{k}_{h}(E)\times Q^{k}_{h}(E)\rightarrow\mathbf{C}$, following \cite{BeiraoBrezziCangianiManziniMariniRusso2013} and by using the definition of $\Pi_{k}^{\nabla,E}$ and $\Pi_{k}^{0,E}$, we first split $a^{\text{\tiny{E}}}$ and $m^{\text{\tiny{E}}}$ in a part that can be computed exactly (up to the machine precision) and in a part that will be suitably approximated:
\begin{align}
\label{discrete_local_bilinear_form_a_0} a^{\text{\tiny{E}}}(u_{h},v_{h})&=a^{\text{\tiny{E}}}\left(\Pi_{k}^{\nabla,E}u_{h},\Pi_{k}^{\nabla,E}v_{h}\right)+a^{\text{\tiny{E}}}\left(\left(I-\Pi_{k}^{\nabla,E}\right)u_{h},\left(I-\Pi_{k}^{\nabla,E}\right)v_{h}\right)\\ 
\label{discrete_local_bilinear_form_m_0} 
m^{\text{\tiny{E}}}(u_{h},v_{h})&=m^{\text{\tiny{E}}}\left(\Pi_{k}^{0,E}u_{h},\Pi_{k}^{0,E}v_{h}\right)+m^{\text{\tiny{E}}}\left(\left(I-\Pi_{k}^{0,E}\right)u_{h},\left(I-\Pi_{k}^{0,E}\right)v_{h}\right),
\end{align}
\noindent
$I$ being the identity operator. The implementation steps for the computation of $a^{\text{\tiny{E}}}\left(\Pi_{k}^{\nabla,E}u_{h},\Pi_{k}^{\nabla,E}v_{h}\right)$ and $m^{\text{\tiny{E}}}\left(\Pi_{k}^{0,E}u_{h},\Pi_{k}^{0,E}v_{h}\right)$ require the choice of a suitable basis of the space $P_{k}(E)$, that allows to define in practice the projectors $\Pi_{k}^{\nabla,E}$ and $\Pi_{k}^{0,E}$. In accordance with the standard literature on VEM (see \cite{BeiraoBrezziMariniRusso2014}, Section 3.1), we have considered the basis of the scaled monomials, i.e.
\begin{displaymath}\label{eq:monomi}
\mathcal{M}_{k}(E):=\left\{p_{\bm{\alpha}}(\mathbf{x})=\left(\frac{\mathbf{x}-\mathbf{V}_{E}}{h_{E}}\right)^{\bm{\alpha}} , \ \bm{\alpha}=(\alpha_{1},\alpha_{2}) \ : \ |\bm{\alpha}|=\alpha_{1}+\alpha_{2}\leq k\right\},
\end{displaymath}
where, we recall, ${\bf V}_{E}$ and $h_E$ denote the mass centre and the diameter of $E$, respectively.

%

%

\noindent
Following \cite{BeiraoBrezziMariniRusso2014}, the second term in (\ref{discrete_local_bilinear_form_a_0}) is approximated by the following bilinear form which represents a stabilization term:
\begin{align}
\label{stab_bilinear_form_a}
s^{\text{\tiny{E}}}\left(\left(I-\Pi_{k}^{\nabla,E}\right)u_{h},\left(I-\Pi_{k}^{\nabla,E}\right)v_{h}\right)&:=\sum\limits_{j=1}^{n}\text{dof}_{j}\left(\left(I-\Pi_{k}^{\nabla,E}\right)u_{h}\right)\text{dof}_{j}\left(\left(I-\Pi_{k}^{\nabla,E}\right)v_{h}\right).
\end{align}
On the contrary,  since the analysis of the method 
 requires the ellipticity property only for the bilinear form $\mathcal{B}_{0,h}$ (see Assumption \ref{H3.b}), an analogous stabilizing term is not needed in (\ref{discrete_local_bilinear_form_m_0}). Therefore we define the approximations $a^{\text{\tiny{E}}}_{h}$ and $m^{\text{\tiny{E}}}_{h}$ of $a^{\text{\tiny{E}}}$ and $m^{\text{\tiny{E}}}$, respectively, as follows:
\begin{align}
\label{discrete_local_bilinear_form_a} 
a^{\text{\tiny{E}}}_{h}(u_{h},v_{h})&:=a^{\text{\tiny{E}}}\left(\Pi_{k}^{\nabla,E}u_{h},\Pi_{k}^{\nabla,E}v_{h}\right)+s^{\text{\tiny{E}}}\left(\left(I-\Pi_{k}^{\nabla,E}\right)u_{h},\left(I-\Pi_{k}^{\nabla,E}\right)v_{h}\right)\\ 
\label{discrete_local_bilinear_form_m} 
m^{\text{\tiny{E}}}_{h}(u_{h},v_{h})&:=m^{\text{\tiny{E}}}\left(\Pi_{k}^{0,E}u_{h},\Pi_{k}^{0,E}v_{h}\right).
\end{align}
As shown in \cite{BeiraoRussoVacca2018} and \cite{BeiraoRussoVacca2019}, the approximate bilinear forms satisfy the following properties: 
\begin{itemize}
\item $k$-consistency: for all $v_{h}\in Q^{k}_{h}(E)$ and for all $q\in P_{k}(E)$:
\begin{equation}\label{consistency_am} a^{\text{\tiny{E}}}_{h}(v_{h},q)=a^{\text{\tiny{E}}}(v_{h},q)\qquad \text{and} \qquad
 m^{\text{\tiny{E}}}_{h}(v_{h},q)=m^{\text{\tiny{E}}}(v_{h},q);
\end{equation}
\item stability: 
for all $v_{h}\in Q^{k}_{h}(E)$:
\begin{equation}\label{stability_am}
a^{\text{\tiny{E}}}(v_{h},v_{h})\apprle a^{\text{\tiny{E}}}_{h}(v_{h},v_{h}) \apprle a^{\text{\tiny{E}}}(v_{h},v_{h}) \qquad \text{and} \qquad 
 m^{\text{\tiny{E}}}_{h}(v_{h},v_{h})  \apprle m^{\text{\tiny{E}}}(v_{h},v_{h}).
\end{equation}
%
\end{itemize}
\noindent
The global approximate bilinear forms $a_{h},m_{h}:Q^{k}_{h}\times Q^{k}_{h}\rightarrow\mathbf{C}$ are then defined by  summing up the local contributions as follows:
\begin{equation*}\label{global_discrete_bilinear_form}
a_{h}(u_{h},v_{h}):=\sum\limits_{E\in\mathcal{T}_{h}}a^{\text{\tiny{E}}}_{h}(u_{h},v_{h}) \qquad \text{and} \qquad m_{h}(u_{h},v_{h}):=\sum\limits_{E\in\mathcal{T}_{h}}m^{\text{\tiny{E}}}_{h}(u_{h},v_{h}).
\end{equation*}
From the right hand side of (\ref{stability_am}), it immediately follows that
\begin{equation}\label{bound_m}
m_{h}(v_{h},v_{h})\apprle \left\|v_{h}\right\|_{L^{2}(\Omega)}^{2} \qquad \forall \, v_{h}\in Q^{k}_{h}
\end{equation}
while, combining (\ref{stability_am}) with the Poincaré-Friedrichs inequality (see (5.3.3) in \cite{BrennerScott2008}), we have:
\begin{equation}\label{bound_a}
\left\|v_{h}\right\|_{H^{1}(\Omega)}^{2}\apprle a_{h}(v_{h},v_{h})\apprle \left\|v_{h}\right\|_{H^{1}(\Omega)}^{2} \qquad \forall v_{h}\in Q^{k}_{h}.
\end{equation}
%

%
%
%
%
%
%
\

Moreover, the characterization of the virtual element space $Q_{h}^{k}$ and the boundary element space $X_{h}^{k}$ allows us to formally define the bilinear form $\mathcal{B}_{\kappa,h}:V^k_{h}\times V^k_{h}\rightarrow\mathbf{C}$,
\begin{equation*} \label{eq:modifiedB}
\mathcal{B}_{\kappa,h}(\hat{u}_{h},\hat{v}_{h}):= a_{h}(u_{h},v_{h})-\kappa^{2}m_{h}(u_{h},v_{h}) - \langle\lambda_{h},v_{h}\rangle_{\Gamma} + \langle \mu_{h},u_{h}\rangle_{\Gamma} + 2\langle\mu_{h},\text{V}_{\kappa} \lambda_{h}\rangle_{\Gamma}
\end{equation*}
for $\hat u_{h} = (u_{h}, \lambda_{h}), \hat v_{h} = (v_{h},\mu_{h})\in V^k_{h}$. 

From the $k$-consistency of the discrete bilinear forms $a^{\text{\tiny{E}}}_{h}$ and $m^{\text{\tiny{E}}}_{h}$ (see \eqref{consistency_am}), it immediately follows that $\mathcal{B}_{\kappa,h}$ satisfies Assumption \ref{H2.a}. 
Furthermore, the continuity of the bilinear forms $a_{h}$ and $m_{h}$, as well as the continuity of the boundary operator $\text{V}_\kappa$, ensure the $V$-norm continuity of $\mathcal{B}_{\kappa,h}$, i.e. Assumption \ref{H2.b}. Analogously, the $W$-norm continuity of $\mathcal{D}_{\kappa,h}=\mathcal{B}_{\kappa,h}-\mathcal{B}_{0,h}$, i.e. Assumption \ref{H3.a}, is a consequence of the continuity of $m_{h}$ and of Lemma \ref{Lemma:Vcompa}.\\

\noindent

Now, to prove the $\widetilde{V}_{h}^{k}$-ellipticity, i.e. assumption \ref{H3.b}, we focus on the term
\begin{equation*} \label{eq:modifiedB0}
\mathcal{B}_{0,h}(\hat{v}_{\0 h},\hat{v}_{\0 h}) = a_{h}(v_{h},v_{h}) + 2\langle\mu_{\0 h},\text{V}_{0} \mu_{\0 h}\rangle_{\Gamma}
\end{equation*}
\noindent for $\hat v_{\0 h} = (v_h,\mu_{\0 h}) \in \widetilde V^k_h$. In order to bound the first and the second term in the above sum,
we use (\ref{bound_a}) and Theorem 6.22 in \cite{Steinbach2008}, respectively, and we get
\begin{displaymath}
\mathcal{B}_{0,h}(\hat{v}_{\0 h},\hat{v}_{\0 h})
\apprge\left\|v_{h}\right\|_{H^{1}(\Omega)}^{2}+\left\|\mu_{\0 h}\right\|_{H^{-\nicefrac{1}{2}}(\Gamma)}^{2}=\left\|\hat{v}_{h}\right\|_{V}^{2}.
\end{displaymath}
\noindent
Thus, Assumption \ref{H3.c} is a direct consequence of the $k$-consistency  \eqref{consistency_am}.

\subsection{The discrete linear operator $\mathcal{L}_{f,h}$: validity of Assumption \ref{H4.a}}
In the present section, we define the discrete linear operator $\mathcal{L}_{f,h}:V_{h}^{k}\rightarrow\mathbf{C}$ such that
\begin{equation*}
\mathcal{L}_{f,h}(\hat v_{h}) := \begin{cases}
 \sum\limits_{E \in \mathcal{T}_h} m^{E}(f,\Pi^{0,E}_{1}v_{h}) & k=1,2\, ,\\
 \sum\limits_{E \in \mathcal{T}_h} m^{E}(f,\Pi^{0,E}_{k-2}v_{h}) &  k \ge 3.
\end{cases}
\end{equation*}
Assuming $f \in H^{k-1}(\Omega)$, in \cite{BrennerGuanSung2017} (see Lemma 3.4) it has been proved that
\begin{equation*}
\left|\mathcal{L}_{f}(\hat v_{h}) - \mathcal{L}_{f,h}(\hat v_{h})\right|\apprle h^{k}\left|f\right|_{H^{k-1}(\Omega)} \left\|v_{h}\right\|_{H^{1}(\Omega)}.
\end{equation*}
Hence, Assumption \ref{H4.a} is fulfilled with $\sigma(f) = \left|f\right|_{H^{k-1}(\Omega)}$.\\
\subsection{Algebraic formulation of the discrete problem}

\newcommand{\Stot}{\mathcal{S}}
\newcommand{\Sint}{\mathcal{S}^I}
\newcommand{\SG}{\mathcal{S}^{\Gamma_0}}
\newcommand{\SB}{\mathcal{S}^{\Gamma}}

For what follows, in order to detail the algebraic form of the coupled CVEM-BEM method, it is convenient to re-order and split the complete index set $\Stot:=\{j=1,\cdots,N\}$  of the basis functions $\left\{\Phi_{j}\right\}_{j\in\Stot}$ of $Q^{k}_{h}$ as
\begin{equation}\label{def_Stot}
\Stot=\Sint\cup\SB,
\end{equation}
where $\Sint$ and $\SB$ denote the sets of indices related to the internal degrees of freedom and to the degrees of freedom lying on $\Gamma$, respectively.
With this choice, we have
$$Q_{h}^{k} = \text{span}\left\{\Phi_{j}\right\}_{j\in\Sint\cup\SB}, \qquad X^{k}_{h}= \text{span}\left\{\Phi_{j_{{|_\Gamma}}}\right\}_{j\in\SB}.$$
%

\noindent
In order to derive the linear system associated to the discrete problem \eqref{model_problem_galerkin}, we expand the unknown function $\hat u_{h} = (u_{h}, \lambda_{h}) \in Q_{h}^{k} \times X^{k}_{h}$  as
\begin{equation}\label{interpolants}
\begin{aligned}
u_{h}(\mathbf{x})=:\sum\limits_{j\in\Stot}u_{h}^{j}\Phi_{j}(\mathbf{x}) \qquad \text{with} \quad u_{h}^{j}=\text{dof}_{j}(u_{h})\\
\lambda_{h}(\mathbf{x})=:\sum\limits_{j\in\SB}\lambda_{h}^{j}\Phi_{j_{{|_\Gamma}}}(\mathbf{x}) \qquad \text{with} \quad \lambda_{h}^{j}=\text{dof}_{j}(\lambda_{h}).
\end{aligned}
\end{equation}

%
\noindent
Hence, using the basis functions of $Q^k_h$ to test the discrete counterpart of equation \eqref{model_problem_variational_1}, we get for $i\in\Sint\cup\SB$
%
\begin{align}\label{VEM_discrete_equation}
\sum\limits_{j\in\Sint\cup\SB}u_{h}^{j}&\sum\limits_{E\in\mathcal{T}_{h}}\left[a^{\text{\tiny{E}}}_{h}(\Phi_{j},\Phi_{i})-\kappa^{2}m^{\text{\tiny{E}}}_{h}(\Phi_{j},\Phi_{i})\right]-\sum\limits_{j\in\SB}\lambda_{h}^{j}\langle\Phi_{j},\Phi_{i}\rangle_{\Gamma} = \mathcal{L}_{f,h}((\Phi_i,0)).
\end{align}

To write the matrix form of the above linear system, we introduce the stiffness and mass matrices $\mathbb{A}$, $\mathbb{M}$ and the matrix $\mathbb{Q}$ whose entries are respectively defined by
\begin{displaymath}
\mathbb{A}_{ij} :=  \sum\limits_{E\in\mathcal{T}_{h}}a^{\text{\tiny{E}}}_{h}(\Phi_{j},\Phi_{i}), \qquad 
\mathbb{M}_{ij} :=  \sum\limits_{E\in\mathcal{T}_{h}}m^{\text{\tiny{E}}}_{h}(\Phi_{j},\Phi_{i}), \qquad \mathbb{Q}_{ij} :=  \langle\Phi_{j},\Phi_{i}\rangle_{\Gamma}
\end{displaymath}
and the column vectors $\mathbf{u} = \left[u_{h}^{j}\right]_{j\in\Sint\cup\SB}$, $\boldsymbol{\lambda} = \left[\lambda_{h}^{j}\right]_{j\in\SB}$ and $\mathbf{f} = \left[\mathcal{L}_{f,h}((\Phi_i,0))\right]_{i\in\Sint\cup\SB}$.
In accordance with the splitting of the set of the degrees of freedom \eqref{def_Stot}, we consider the block partitioned representation of the above matrices and vectors (with obvious meaning of the notation), and we rewrite equation \eqref{VEM_discrete_equation} as follows:
\begin{eqnarray}\label{VEM_system}
\left[
\begin{array}{ll}
\mathbb{A}^{I I}-\kappa^{2}\mathbb{M}^{I I}& \mathbb{A}^{I \Gamma} - \kappa^{2}\mathbb{M}^{I \Gamma}\\
&\\
\mathbb{A}^{\Gamma I}-\kappa^{2}\mathbb{M}^{\Gamma I} & \mathbb{A}^{\Gamma \Gamma} -\kappa^{2}\mathbb{M}^{\Gamma \Gamma}\\
\end{array}\right]\left[\begin{array}{l}
\mathbf{u}^I\\
\\
\mathbf{u}^\Gamma
\end{array}
\right]
-	\left[
\begin{array}{l}
\mathbb{Q}^{I \Gamma} \boldsymbol{\lambda}\\
\\
\mathbb{Q}^{\Gamma \Gamma} \boldsymbol{\lambda}
\end{array}
\right]
= 	\left[
\begin{array}{l}
\mathbf{f}^{I}\\
\\
\mathbf{f}^\Gamma
\end{array}\right]
\end{eqnarray}
noting that $\mathbb{Q}^{I \Gamma}$ is a null matrix since $\langle\Phi_{j},\Phi_{i}\rangle_{\Gamma} = 0$ for $i\in\Sint$, $j\in\SB$.
\noindent

\

For what concerns the discretization of the BI-NRBC, by inserting \eqref{interpolants} in \eqref{model_problem_variational_2} and testing with the functions $\Phi_i$, $i \in\SB$, we obtain
\begin{equation}\label{boundary_integral_equation_discretized}
\begin{aligned}
\sum\limits_{j\in\SB}&\left\{u_{h}^{j}\left[\frac{1}{2}\int\limits_{\Gamma}\Phi_{j}(\mathbf{x})\Phi_{i}(\mathbf{x})\dd\Gamma_\mathbf{x}-\int\limits_{\Gamma}\left(\int\limits_{\Gamma}\frac{\partial G}{\partial\mathbf{n}_{\mathbf{y}}}(\mathbf{x},\mathbf{y})\Phi_{j}(\mathbf{y})\,\dd\Gamma_{\mathbf{y}}\right)\Phi_{i}(\mathbf{x})\dd\Gamma_{\mathbf{x}}\right]\right.\\
&\,\,\left.+\lambda_{h}^{j}\int\limits_{\Gamma}\left(\int\limits_{\Gamma}G(\mathbf{x},\mathbf{y})\Phi_{j}(\mathbf{y})\,\dd\Gamma_{\mathbf{y}}\right)\Phi_{i}(\mathbf{x})\dd\Gamma_{\mathbf{x}}\right\}=0.
\end{aligned}
\end{equation}
To detail the computation of the integrals in \eqref{boundary_integral_equation_discretized}, 
we start by splitting the integral on the whole $\Gamma$ into the sum of the contributions associated to each boundary edge $\Gamma_{_\ell}$, $\ell = 1,\cdots,N^\Gamma$
\begin{equation}\label{boundary_integral_equation_discretized_split}
\begin{aligned}
\sum\limits_{j\in\SB}&\left\{u_{h}^{j}\left[\frac{1}{2}\sum_{\ell=1}^{N^\Gamma}\int\limits_{\Gamma_{_\ell}}\Phi_{j}(\mathbf{x})\Phi_{i}(\mathbf{x})\dd\Gamma_\mathbf{x}-\sum_{\ell=1}^{N^\Gamma}\int\limits_{\Gamma_\ell}\left(\sum_{r=1}^{N^\Gamma}\int\limits_{\Gamma_r}\frac{\partial G}{\partial\mathbf{n}_{\mathbf{y}}}(\mathbf{x},\mathbf{y})\Phi_{j}(\mathbf{y})\,\dd\Gamma_{\mathbf{y}}\right)\Phi_{i}(\mathbf{x})\dd\Gamma_{\mathbf{x}}\right]\right.\\
&\,\,\left.+\lambda_{h}^{j}\sum_{\ell=1}^{N^\Gamma}\int\limits_{\Gamma_\ell}\left(\sum_{r=1}^{N^\Gamma}\int\limits_{\Gamma_r}G(\mathbf{x},\mathbf{y})\Phi_{j}(\mathbf{y})\,\dd\Gamma_{\mathbf{y}}\right)\Phi_{i}(\mathbf{x})\dd\Gamma_{\mathbf{x}}\right\}=0.
\end{aligned}
\end{equation}
Then, denoting by $E_\ell$, $\ell = 1,\cdots, N^\Gamma$, the mesh element of $\mathcal{T}_h$ that has one of its curved edges on $\Gamma$ and by $\gamma_{E_\ell}:I_{E_\ell}\rightarrow \Gamma_\ell$ the associated pameterization, we rewrite \eqref{boundary_integral_equation_discretized_split} by introducing $\gamma_{E_\ell}$ and hence by reducing the integration over $\Gamma_\ell$ to that over the parametric interval $I_{E_\ell}$: 
\begin{equation*}
\begin{aligned}
\sum\limits_{j\in\SB}&\left\{u_{h}^{j}\left[\frac{1}{2}\sum_{\ell=1}^{N^\Gamma}\int\limits_{I_{E_\ell}}\Phi_{j}\Big(\gamma_{E_\ell}(\vartheta)\Big)\Phi_{i}\Big(\gamma_{E_\ell}(\vartheta)\Big)\Big|\gamma'_{E_\ell}(\vartheta)\Big|\dd\vartheta\right.\right.\\
&\,\,\left.\left. -\sum_{\ell=1}^{N^\Gamma}\int\limits_{I_{E_\ell}}\left(\sum_{r=1}^{N^\Gamma}\int\limits_{I_{E_r}}\frac{\partial G}{\partial\mathbf{n}_{\mathbf{y}}}\Big(\gamma_{E_\ell}(\vartheta),\gamma_{E_r}(\sigma)\Big)\Phi_{j}\Big(\gamma_{E_r}(\sigma)\Big)\,\Big|\gamma'_{E_r}(\sigma)\Big|\dd\sigma\right)\Phi_{i}\Big(\gamma_{E_\ell}(\vartheta)\Big)\Big|\gamma'_{E_\ell}(\vartheta)\Big|\dd\vartheta\right]\right.\\
&\,\,\left.+\lambda_{h}^{j}\sum_{\ell=1}^{N^\Gamma}\int\limits_{I_{E_\ell}}\left(\sum_{r=1}^{N^\Gamma}\int\limits_{I_{E_r}}G\Big(\gamma_{E_\ell}(\vartheta),\gamma_{E_r}(\sigma)\Big)\Phi_{j}\Big(\gamma_{E_r}(\sigma)\Big)\,\Big|\gamma'_{E_r}(\sigma)\Big|\dd\sigma\right)\Phi_{i}\Big(\gamma_{E_\ell}(\vartheta)\Big)\Big|\gamma'_{E_\ell}(\vartheta)\Big|\dd\vartheta\right\}=0.
\end{aligned}
\end{equation*}
Finally, introducing the BEM matrices $\mathbb{V}$ and $\mathbb{K}$ whose entries, for $i,j\in \SB$, are respectively
\begin{align*}
\mathbb{V}_{ij}&:=\sum_{\ell=1}^{N^\Gamma}\int\limits_{I_{E_\ell}}\left(\sum_{r=1}^{N^\Gamma}\int\limits_{I_{E_r}}G\Big(\gamma_{E_\ell}(\vartheta),\gamma_{E_r}(\sigma)\Big)\Phi_{j}\Big(\gamma_{E_r}(\sigma)\Big)\,\Big|\gamma'_{E_r}(\sigma)\Big|\dd\sigma\right)\Phi_{i}\Big(\gamma_{E_\ell}(\vartheta)\Big)\Big|\gamma'_{E_\ell}(\vartheta)\Big|\dd\vartheta\\
\mathbb{K}_{ij}&:=\sum_{\ell=1}^{N^\Gamma}\int\limits_{I_{E_\ell}}\left(\sum_{r=1}^{N^\Gamma}\int\limits_{I_{E_r}}\frac{\partial G}{\partial\mathbf{n}_{\mathbf{y}}}\Big(\gamma_{E_\ell}(\vartheta),\gamma_{E_r}(\sigma)\Big)\Phi_{j}\Big(\gamma_{E_r}(\sigma)\Big)\,\Big|\gamma'_{E_r}(\sigma)\Big|\dd\sigma\right)\Phi_{i}\Big(\gamma_{E_\ell}(\vartheta)\Big)\Big|\gamma'_{E_\ell}(\vartheta)\Big|\dd\vartheta,
\end{align*}
the matrix form of the BI-NRBC reads
\begin{equation}\label{BEM_system}
\left(\frac{1}{2}\mathbb{Q}^{\Gamma \Gamma}-\mathbb{K}\right)\mathbf{u}^{\Gamma}+\mathbb{V}\boldsymbol{\lambda}=\mathbf{0}.
\end{equation}
\noindent
%
%
By combining \eqref{VEM_system} with \eqref{BEM_system} we obtain the final linear system
%
\begin{eqnarray}\label{final_system}
\left[
\begin{array}{ccc}
\mathbb{A}^{I I}-\kappa^{2}\mathbb{M}^{I I}& \mathbb{A}^{I \Gamma} - \kappa^{2}\mathbb{M}^{I \Gamma} & \mathbb{O}\\

&\\
\mathbb{A}^{\Gamma I}-\kappa^{2}\mathbb{M}^{\Gamma I} & \mathbb{A}^{\Gamma \Gamma} -\kappa^{2}\mathbb{M}^{\Gamma \Gamma} & -\mathbb{Q}^{\Gamma \Gamma}\\
&\\
\mathbb{O}   & \frac{1}{2}\mathbb{Q}^{\Gamma \Gamma}-\mathbb{K}& \mathbb{V}
\end{array}\right]\left[\begin{array}{l}
\mathbf{u}^I\\
\\
\mathbf{u}^\Gamma\\
\\
\boldsymbol{\lambda}
\end{array}
\right]
= 	\left[
\begin{array}{c}
\mathbf{f}^{I}\\
\\
\mathbf{f}^\Gamma\\
\\
\mathbf{0}
\end{array}\right]
\end{eqnarray}
which represents the matrix form of the coupling of equations \eqref{VEM_discrete_equation} and \eqref{boundary_integral_equation_discretized_split}.


\section{Numerical results}\label{sec_6_num_results}
In this section, we present some numerical test to validate the theoretical results and to show the effectiveness of the proposed method. To this aim, some preliminary features are addressed.\\
\indent
For the generation of the partitioning $\mathcal{T}_{h}$ of the computational domain $\Omega$, we have used the GMSH software (see \cite{GeuzaineRemacle2009}). In particular, we have built unstructured conforming meshes consisting of quadrilaterals, by employing the Mesh.ElementOrder option within the GMSH code. If an element $E$ has a (straight) edge bordering with a curvilinear part of $\partial\Omega$, we transform it into a curved boundary edge by means of a suitable parametrization. 
In order to develop a convergence analysis, we start by considering the coarse mesh associated to the level of refinement zero (lev. 0) and all the successive refinements are obtained by halving each side of its elements.\\
\indent
To test our numerical approach, the order $k$ of the approximation spaces is chosen equal to 1 (linear) and 2 (quadratic) for both spaces $Q_{h}^{k}$ and $X_{h}^{k}$. Moreover, recalling that the approximate solution $u_{h}$ is not known inside the polygons, as suggested in \cite{BeiraoRussoVacca2019} we compute the $H^1$-seminorm and $L^2$-norm errors, and the corresponding Estimated Order of Convergence (EOC), by means of the following formulas:
\begin{itemize}
\item $H^{1}$-seminorm error $\varepsilon^{\nabla,k}_{\text{lev}}:=\sqrt{\frac{\sum\limits_{E\in\mathcal{T}_{h}}\left|u-\Pi_{k}^{\nabla,E}u_{h}\right|^{2}_{H^{1}(E)}}{\sum\limits_{E\in\mathcal{T}_{h}}\left|u\right|^{2}_{H^{1}(E)}}}$ and $\text{EOC}:=\log_{2}\left(\frac{\varepsilon^{\nabla,k}_{\text{lev}+1}}{\varepsilon^{\nabla,k}_{\text{lev}}}\right)$;
\item $L^{2}$-norm error $\varepsilon^{0,k}_{\text{lev}}:=\sqrt{\frac{\sum\limits_{E\in\mathcal{T}_{h}}\left\|u-\Pi_{k}^{0,E}u_{h}\right\|^{2}_{L^{2}(E)}}{\sum\limits_{E\in\mathcal{T}_{h}}\left\|u\right\|^{2}_{L^{2}(E)}}}$ and $\text{EOC}:=\log_{2}\left(\frac{\varepsilon^{0,k}_{\text{lev}+1}}{\varepsilon^{0,k}_{\text{lev}}}\right)$.
\end{itemize}
\noindent
In the above formulas the superscript $k=1,2$ refers to the linear or quadratic order approximation of $u$, the subscript $\text{lev}$ refers to the refinement level and, we recall, $\Pi_{k}^{\nabla,E}$ and $\Pi_{k}^{0,E}$ are the local $H^{1}$ and $L^{2}$-projector defined in (\ref{proj_Pi_nabla}) and (\ref{proj_Pi_0}), respectively.\\

\noindent
All the numerical computations have been performed on a cluster with two Intel\textsuperscript{\textregistered} XEON\textsuperscript{\textregistered} E5-2683v4 CPUs (2.1 GHz clock frequency and 16 cores) by means of parallel Matlab\textsuperscript{\textregistered} codes.

\subsection{On the computation of the integrals involved in the proposed approach}

We start by describing the quadrature techniques adopted for the computation of the integrals appearing in the local bilinear forms $a_h^E$ and $m_h^E$ in \eqref{discrete_local_bilinear_form_a} and \eqref{discrete_local_bilinear_form_m}.
To this aim we point out that, if the element $E$ is a polygon with straight edges, the choice of $\mathcal{M}_{k}(E)$ defined by \eqref{eq:monomi} allows for an exact (up to the machine precision) and easy computation of the first term in the right hand side of \eqref{discrete_local_bilinear_form_a} and of \eqref{discrete_local_bilinear_form_m} (for the details we refer to formulas (27)--(30) in \cite{DesiderioFallettaScuderi2021}).		
On the contrary, if the element $E$ is a polygon with a curved edge, the corresponding integrals can be computed by applying the $n$-point Gauss-Lobatto quadrature rule. This latter, contrarily to the former, is affected by a quadrature error if the involved parametrization is not of polynomial type. However, in our test, the choice $n=8$ has guaranteed the optimal convergence order of the global scheme.

\

For what concerns the evaluation of the $H^1$-seminorm and $L^2$-norm errors, to
compute the associated integrals over polygons we have used the $n$-point quadrature formulas proposed in \cite{SommarivaVianello2007} and \cite{SommarivaVianello2009}, which are exact for polynomials of degree at most $2n$. For curved polygons, we have applied the generalization of these formulas suggested in \cite{SommarivaVianello2007} (see Remark 1). In this case too, we have chosen $n=8$.

\

Finally, for what concerns the computation of the integrals defining the BEM matrix elements, since in the theoretical analysis we have assumed that the boundary integral operators are not approximated, it is crucial to compute them with a high accuracy. We recall that the numerical integration difficulties spring from the asymptotic behaviour of the Hankel function $H_0^{(1)}(r)$ near the origin (see \eqref{asymptotic_behaviour_H_0_1}), the latter being the kernel of the single layer operator $\text{V}_\kappa$. To compute the corresponding integrals with high accuracy by a small number of nodes, we have used the very simple and efficient polynomial smoothing technique
proposed in \cite{MonegatoScuderi1999}, referred as the \emph{q-smoothing} technique. It is worth noting that such technique is applied only when the distance $r$ approaches to zero. This case  corresponds to the matrix entries belonging to the main diagonal and to those co-diagonals, for which the supports of the basis functions overlap or are contiguous.
After having introduced the \emph{q-smoothing} transformation, with $q=3$, we have applied the $n$-point Gauss-Legendre quadrature rule with  $n=9$ for the outer integrals, and $n=8$ for the inner ones (see \cite{FallettaMonegatoScuderi2014} and Remark 3 in  \cite{DesiderioFallettaScuderi2021} for further details). For the computation of all the other integrals, we have applied a $9\times 8$-point Gauss-Legendre product quadrature rule. Incidentally, we point out that the integrals involving the Bessel function $H^{(1)}_1(r)$, appearing in the kernel of the double layer operator $\text{K}_\kappa$ (see \eqref{asymptotic_behaviour_H_1_1}), do not require a particular quadrature strategy, since its singularity $1/r$ is factored out by the same behaviour of the Jacobian near the origin. Hence, for the computation of the entries of the matrix $\mathbb{K}$, we have directly applied a $9\times 8$-point Gauss-Legendre product quadrature rule.

\

The above described quadrature strategy has guaranteed the computation of all the mentioned integrals with a full precision accuracy (16-digit double precision arithmetic) for both $k=1$ and $k=2$.

%
%
%
\subsection{Definition of the test problem}

\noindent
We consider Problem (\ref{dirichlet_problem}) with source $f = 0$ and Dirichlet condition
\begin{equation}\label{example_dir_condition}
g(\mathbf{x})=\frac{\imath}{4}H_{0}^{(1)}\left(\kappa|\mathbf{x}-\mathbf{x}_{0}|\right) \quad \text{with} \quad \mathbf{x}_{0}=(0,0), \ \mathbf{x}\in\Gamma_{0}
\end{equation}
prescribed on the boundary $\Gamma_{0}$, $H_{0}^{(1)}$ being the 0-th order Hankel function of the first kind. In this case, the exact solution $u(\mathbf{x})$ is known and it is the field produced by the point source $\mathbf{x}_{0}$. Its expression is given by \eqref{example_dir_condition} for every $\mathbf{x}\in\mathbf{R}^{2}$. In the sequel, we report the numerical results corresponding to two choices of the wave number $\kappa=1$ and $\kappa=10$. 

It is worth noting that the system \eqref{final_system} is associated to the discretization of the model problem considered for the theoretical analysis in which, we recall, the Dirichlet boundary condition is null. In the forthcoming examples, dealing with null source $f$ and non vanishing Dirichlet datum $g$, the right hand side term of \eqref{final_system} involves the function $g$ instead of $f$ (see, for example,  \cite{DesiderioFallettaScuderi2021} for details on the corresponding algebraic linear system).

\

In Example 1 we test the CVEM-BEM approach for the choice of a computational domain having both interior and artificial curved boundaries, for which the theoretical analysis has been given. As we will show, the use of the CVEM reveals to be crucial to retrieve the optimal convergence order of the global scheme. Indeed, the standard VEM approach, defined on the polygon that approximates the curved computational domain, allows retrieving only a sub-optimal convergence order (see Figure \ref{fig:circ_circ_error_VEM_kappa_1}).

However, even if we have not provided theoretical results for the CVEM-BEM coupling in the case of a non sufficiently smooth artificial boundary $\Gamma$, in Example 2 we show that the proposed method allows to obtain the optimal convergence order also when a polygonal $\Gamma$ is considered.

 \subsection{Example 1. Computational domain with curved boundaries} 
 \noindent
Let us consider the unbounded region $\Omega_{\text{e}}$, external to the unitary disk $\Omega_{0}:=\{\mathbf{x}=(x_{1},x_{2})\in\mathbf{R}^{2} \ : \ x_{1}^{2}+x_{2}^{2}\leq1\}$. The artificial boundary $\Gamma$ is the circumference of radius 2, $\Gamma = \{\mathbf{x}=(x_{1},x_{2})\in\mathbf{R}^{2} \ : \ x_{1}^{2}+x_{2}^{2} = 4\}$, so that the finite computational domain $\Omega$ is the annulus bounded internally by $\Gamma_{0} = \partial \Omega_0$ and externally by $\Gamma$.  
%

In Table \ref{tab:dofs_circ_circ}, we report the total number of the degrees of freedom associated to the CVEM space, corresponding to each decomposition level of the computational domain, and the approximation orders $k=1,2$. To give an idea of the curvilinear mesh used, in Figure \ref{fig:circ_circ_meshes} we plot the meshes corresponding to level 0 (left plot) and level 2 (right plot).  We remark that
the maximum level of refinement we have considered is lev. 7 for $k=1$, whose number of degrees of freedom coincides with that of lev. 6 for $k=2$. Therefore, in the following tables the symbol $\times$ means that the corresponding simulation has not been performed. 
\begin{figure}[H]
	\centering
	\includegraphics[width=0.70\textwidth]{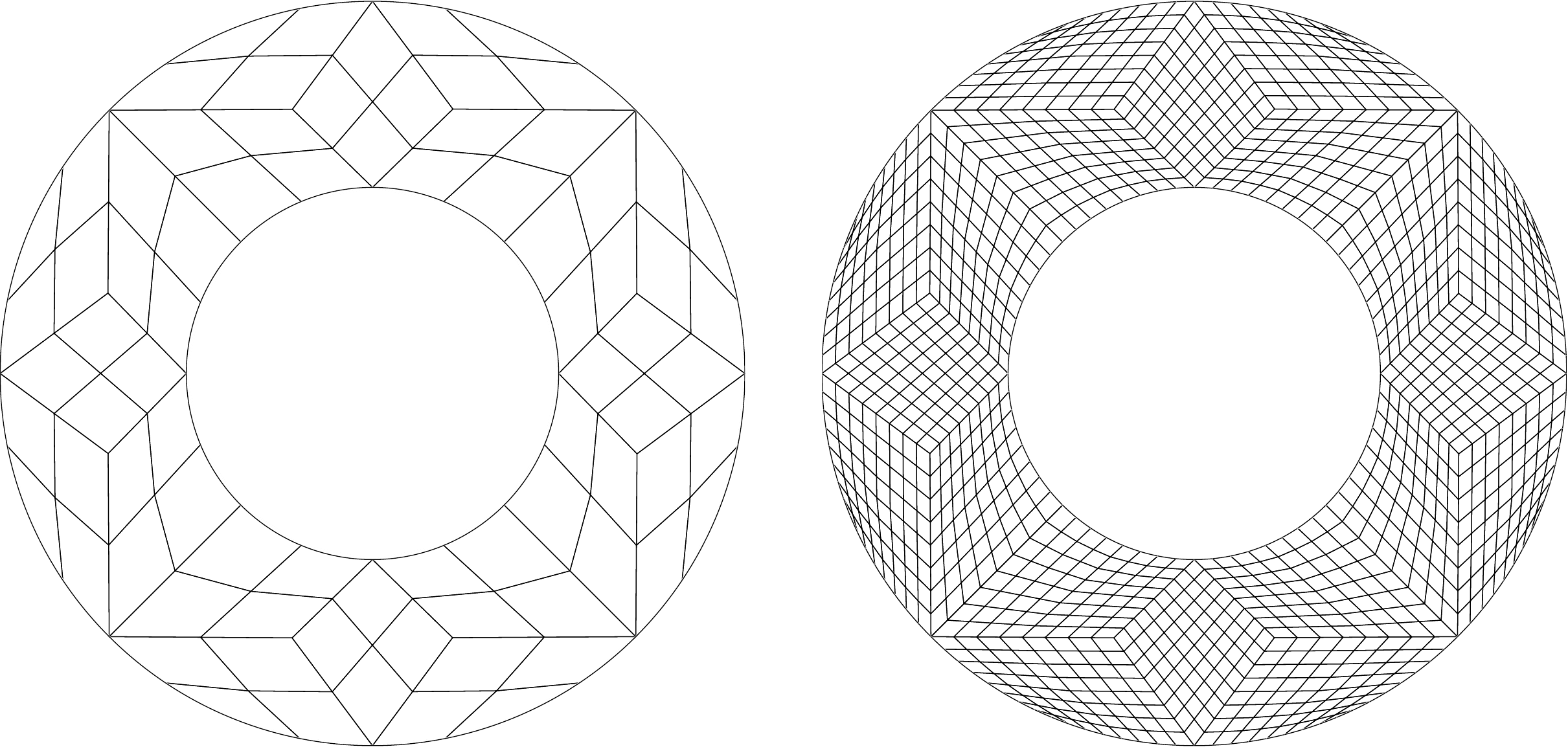}
	\caption{Meshes of $\Omega$ for lev. 0 (left plot) and lev. 2 (right plot).}
	\label{fig:circ_circ_meshes}
\end{figure}
\begin{table}[H]
 \centering
\begin{tabular}{lcccccccc}
\toprule%
	& lev. 0 & lev. 1 & lev. 2 & lev. 3	& lev. 4 & lev. 5 & lev. 6 & lev. 7\\ 
\toprule%
$k=1$	&   $104$	&	   $368$ 	&	$1,376$	&	  $5,312$ &	$20,864$	&	  $82,688$ &   	   $329,216$  &   $1,313,792$\\
$k=2$	&   $368$	&	$1,376$ 	&	$5,312$	& 	$20,864$ & 	$82,688$	& 	$329,216$ & 	$1,313,792$ & \multicolumn{1}{c}{$\times$}\\
\bottomrule
\end{tabular}
\caption{Example 1. Total number of degrees of freedom for $k=1,2$ and for different levels of refinement.}
\label{tab:dofs_circ_circ}
\end{table}
\noindent
In Figures \ref{fig:circ_circ_kappa_1} and \ref{fig:circ_circ_kappa_10}, we show the real and imaginary parts of the numerical solution for the wave numbers $\kappa=1$ and $\kappa=10$ respectively, obtained by the quadratic approximation associated to the minimum refinement level for which the graphical behaviour is accurate and not wavy; this latter is lev. 3 for Figure \ref{fig:circ_circ_kappa_1} and lev. 5 for Figure \ref{fig:circ_circ_kappa_10}. 
%
\begin{figure}[H]
	\centering
	\includegraphics[width=1.00\textwidth]{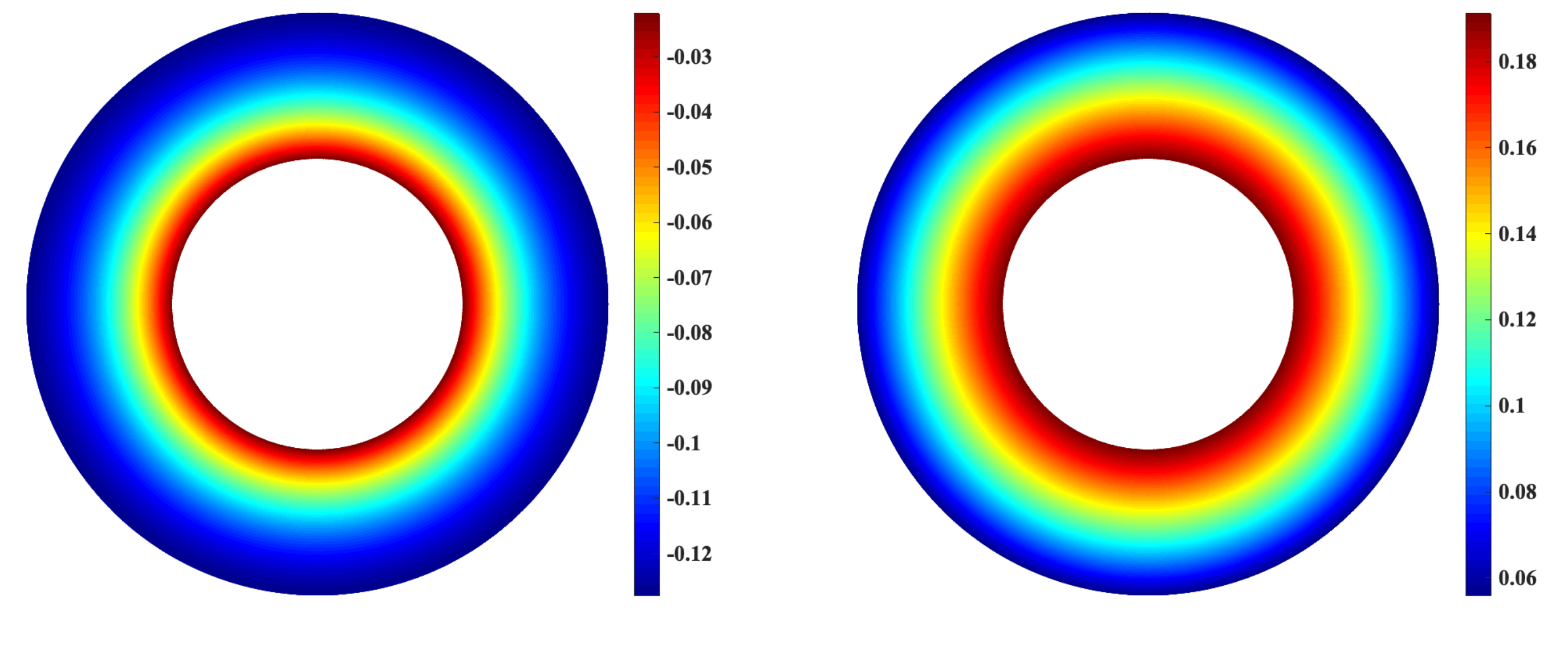}
	\caption{Example 1. Real (left plot) and imaginary (right plot) part of the numerical solution for $\kappa=1$, lev. 3 and $k=2$.}
	\label{fig:circ_circ_kappa_1}
\end{figure}

\begin{figure}[H]
	\centering
	\includegraphics[width=1.00\textwidth]{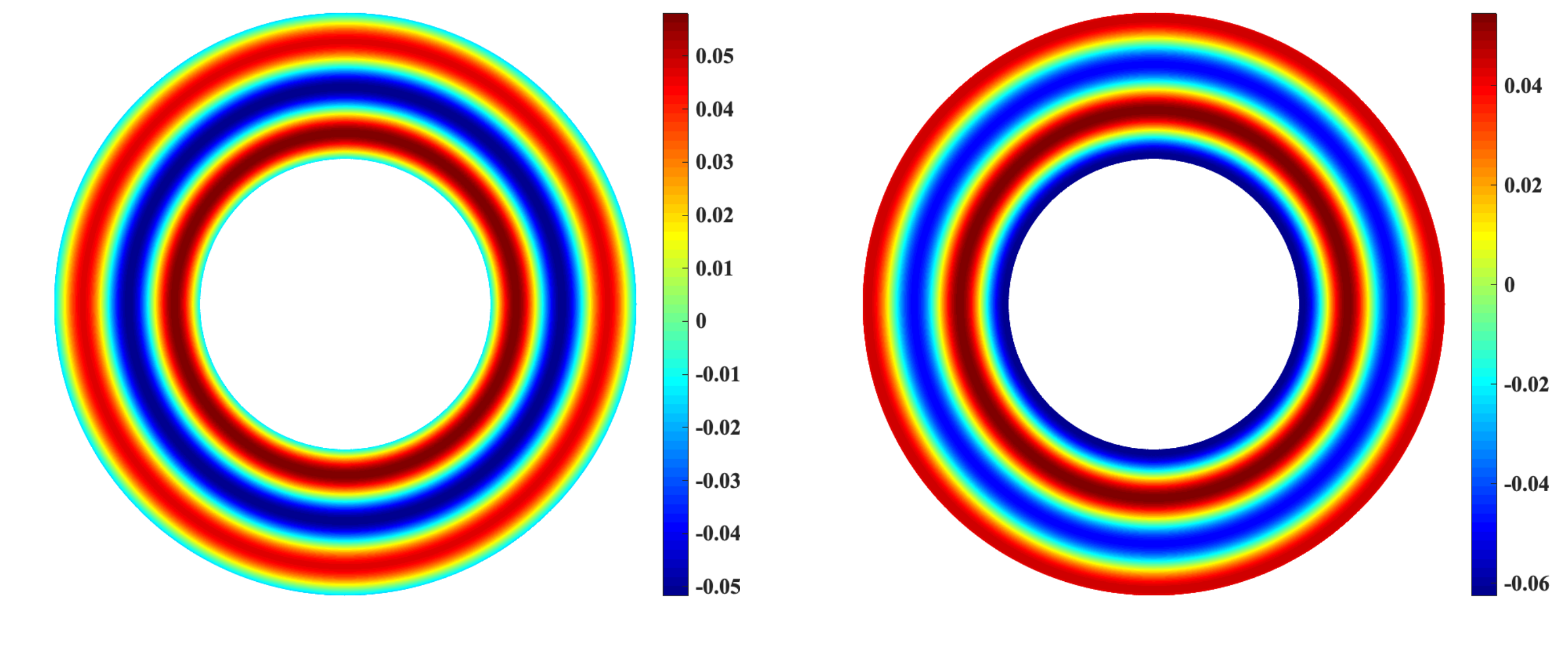}
	\caption{Example 1. Real (left plot) and imaginary (right plot) part of the numerical solution for $\kappa=10$, lev. 5 and $k=2$.}
	\label{fig:circ_circ_kappa_10}
\end{figure}
\noindent
In Tables \ref{tab:convergence_circ_circ_CVEM_BEM_kappa_1} and \ref{tab:convergence_circ_circ_CVEM_BEM_kappa_10}, we report the errors $\varepsilon^{\nabla,k}_{\text{lev}}$ and  $\varepsilon^{0,k}_{\text{lev}}$ and the corresponding EOC. As we can see, for both $\kappa=1$ and $\kappa=10$, the $H^{1}$-seminorm error confirms the convergence order $k$ of the method. Although we did not provide the $L^{2}$-norm error estimate, the reported numerical results show the expected  convergence order $k+1$. 
%
%
\begin{table}[H]
 \centering
  \def\sym#1{\ifmmode^{#1}\else\(^{#1}\)\fi}%
    \scalebox{0.94}{
\begin{tabular}{lc||cccc||cccc}
\toprule%
\multicolumn{2}{c}{} & \multicolumn{4}{c}{$L^{2}$-norm} & \multicolumn{4}{c}{$H^{1}$-seminorm}\\
\toprule%
lev.	& $h$ & $\varepsilon^{0,1}_{\text{lev}}$ & $\text{EOC}$ & $\varepsilon^{0,2}_{\text{lev}}$	& $\text{EOC}$ & $\varepsilon^{\nabla,1}_{\text{lev}}$ & $\text{EOC}$ & $\varepsilon^{\nabla,2}_{\text{lev}}$	& $\text{EOC}$\\ 	
\toprule%
$0$ & $8.02e-01$ & $1.64e-02$ & $$ & $5.83e-04$ & $$ & $5.22e-02$ & $$ & $6.07e-03$ & $$\\
$$  &  $$  &  $$  &  $1.9$  &  $$  &  $3.0$  &  $$  &  $1.0$  &  $$  &  $2.0$\\
$1$ & $4.28e-01$ & $4.52e-03$ & $$ & $7.23e-05$ & $$ & $2.59e-02$ & $$ & $1.54e-03$ & $$\\
$$  &  $$  &  $$  &  $1.9$  &  $$  &  $3.0$  &  $$  &  $1.0$  &  $$  &  $2.0$\\
$2$ & $2.22e-01$ & $1.18e-03$ & $$ & $9.00e-06$ & $$ & $1.29e-02$ & $$ & $3.88e-04$ & $$\\
$$  &  $$  &  $$  &  $2.0$  &  $$  &  $3.0$  &  $$  &  $1.0$  &  $$  &  $2.0$\\
$3$ & $1.13e-01$ & $3.00e-04$ & $$ & $1.12e-06$ & $$ & $6.44e-03$ & $$ & $9.72e-05$ & $$\\
$$  &  $$  &  $$  &  $2.0$  &  $$  &  $3.0$  &  $$  &  $1.0$  &  $$  &  $2.0$\\
$4$ & $5.68e-02$ & $7.56e-05$ & $$ & $1.40e-07$ & $$ & $3.22e-03$ & $$ & $2.42e-05$ & $$\\
$$  &  $$  &  $$  &  $2.0$  &  $$  &  $3.0$  &  $$  &  $1.0$  &  $$  &  $2.0$\\
$5$ & $2.85e-02$ & $1.90e-05$ & $$ & $1.75e-08$ & $$ & $1.61e-03$ & $$ & $6.07e-06$ & $$\\
$$  &  $$  &  $$  &  $2.0$  &  $$  &  $3.0$  &  $$  &  $1.0$  &  $$  &  $2.0$\\
$6$ & $1.43e-02$ & $4.75e-06$ & $$ & $2.20e-09$ & $$ & $8.04e-04$ & $$ & $1.52e-06$ & $$\\
$$  &  $$  &  $$  &  $2.0$  &  $$  &  $\times$  &  $$  &  $1.0$  &  $$  &  $\times$\\
$7$ & $7.14e-03$ & $1.19e-06$ & $$ & $\times$ & $$ & $4.02e-04$ & $$ & $\times$ & $$\\
\bottomrule
\end{tabular}
}
\caption{Example 1. $L^{2}$-norm and $H^{1}$-seminorm relative errors and corresponding EOC, for $\kappa=1$ and $k=1,2$.}
\label{tab:convergence_circ_circ_CVEM_BEM_kappa_1}
\end{table}
\begin{table}[H]
 \centering
  \def\sym#1{\ifmmode^{#1}\else\(^{#1}\)\fi}%
    \scalebox{0.94}{
\begin{tabular}{lc||cccc||cccc}
\toprule%
\multicolumn{2}{c}{} & \multicolumn{4}{c}{$L^{2}$-norm} & \multicolumn{4}{c}{$H^{1}$-seminorm}\\
\toprule%
lev.	& $h$ & $\varepsilon^{0,1}_{\text{lev}}$ & $\text{EOC}$ & $\varepsilon^{0,2}_{\text{lev}}$	& $\text{EOC}$ & $\varepsilon^{\nabla,1}_{\text{lev}}$ & $\text{EOC}$ & $\varepsilon^{\nabla,2}_{\text{lev}}$	& $\text{EOC}$\\ 	
\toprule%
$0$ & $8.02e-01$ & $6.03e-01$ & $$ & $2.57e-01$ & $$ & $5.77e-01$ & $$ & $3.07e-01$ & $$\\
$$  &  $$  &  $$  &  $0.8$  &  $$  &  $2.7$  &  $$  &  $0.6$  &  $$  &  $1.8$\\
$1$ & $4.28e-01$ & $3.52e-01$ & $$ & $4.00e-02$ & $$ & $3.92e-01$ & $$ & $8.59e-02$ & $$\\
$$  &  $$  &  $$  &  $1.5$  &  $$  &  $3.2$  &  $$  &  $1.1$  &  $$  &  $2.0$\\
$2$ & $2.22e-01$ & $1.33e-01$ & $$ & $4.37e-03$ & $$ & $1.84e-01$ & $$ & $2.18e-02$ & $$\\
$$  &  $$  &  $$  &  $1.8$  &  $$  &  $3.2$  &  $$  &  $1.2$  &  $$  &  $2.0$\\
$3$ & $1.13e-01$ & $3.76e-02$ & $$ & $4.71e-04$ & $$ & $7.88e-02$ & $$ & $5.49e-03$ & $$\\
$$  &  $$  &  $$  &  $1.9$  &  $$  &  $3.1$  &  $$  &  $1.1$  &  $$  &  $2.0$\\
$4$ & $5.68e-02$ & $9.74e-03$ & $$ & $5.51e-05$ & $$ & $3.65e-02$ & $$ & $1.38e-03$ & $$\\
$$  &  $$  &  $$  &  $2.0$  &  $$  &  $3.0$  &  $$  &  $1.0$  &  $$  &  $2.0$\\
$5$ & $2.85e-02$ & $2.46e-03$ & $$ & $6.75e-06$ & $$ & $1.78e-02$ & $$ & $3.44e-04$ & $$\\
$$  &  $$  &  $$  &  $2.0$  &  $$  &  $3.0$  &  $$  &  $1.0$  &  $$  &  $2.0$\\
$6$ & $1.43e-02$ & $6.16e-04$ & $$ & $8.39e-07$ & $$ & $8.86e-03$ & $$ & $8.61e-05$ & $$\\
$$  &  $$  &  $$  &  $2.0$  &  $$  &  $\times$  &  $$  &  $1.0$  &  $$  &  $\times$\\
$7$ & $7.14e-03$ & $1.54e-04$ & $$ & $\times$ & $$ & $4.42e-03$ & $$ & $\times$ & $$\\
\bottomrule
\end{tabular}
}
\caption{Example 1. $L^{2}$-norm and $H^{1}$-seminorm relative errors and corresponding EOC, for $\kappa=10$ and $k=1,2$.}
\label{tab:convergence_circ_circ_CVEM_BEM_kappa_10}
\end{table}
\noindent
For completeness, to highlight the importance of the use of CVEM for curved domains, in Figure \ref{fig:circ_circ_error_VEM_kappa_1} we report 
the behaviour of the $H^{1}$-seminorm and $L^{2}$-norm errors, obtained by applying the standard VEM defined on polygonal approximations of the computational domain.
As expected, the optimal rate of convergence is confirmed for the $H^1$-seminorm, while the approximation of the geometry affects the $L^2$-norm error only in the case $k=2$. 
\begin{figure}[H]
	\centering
	\includegraphics[width=1.00\textwidth]{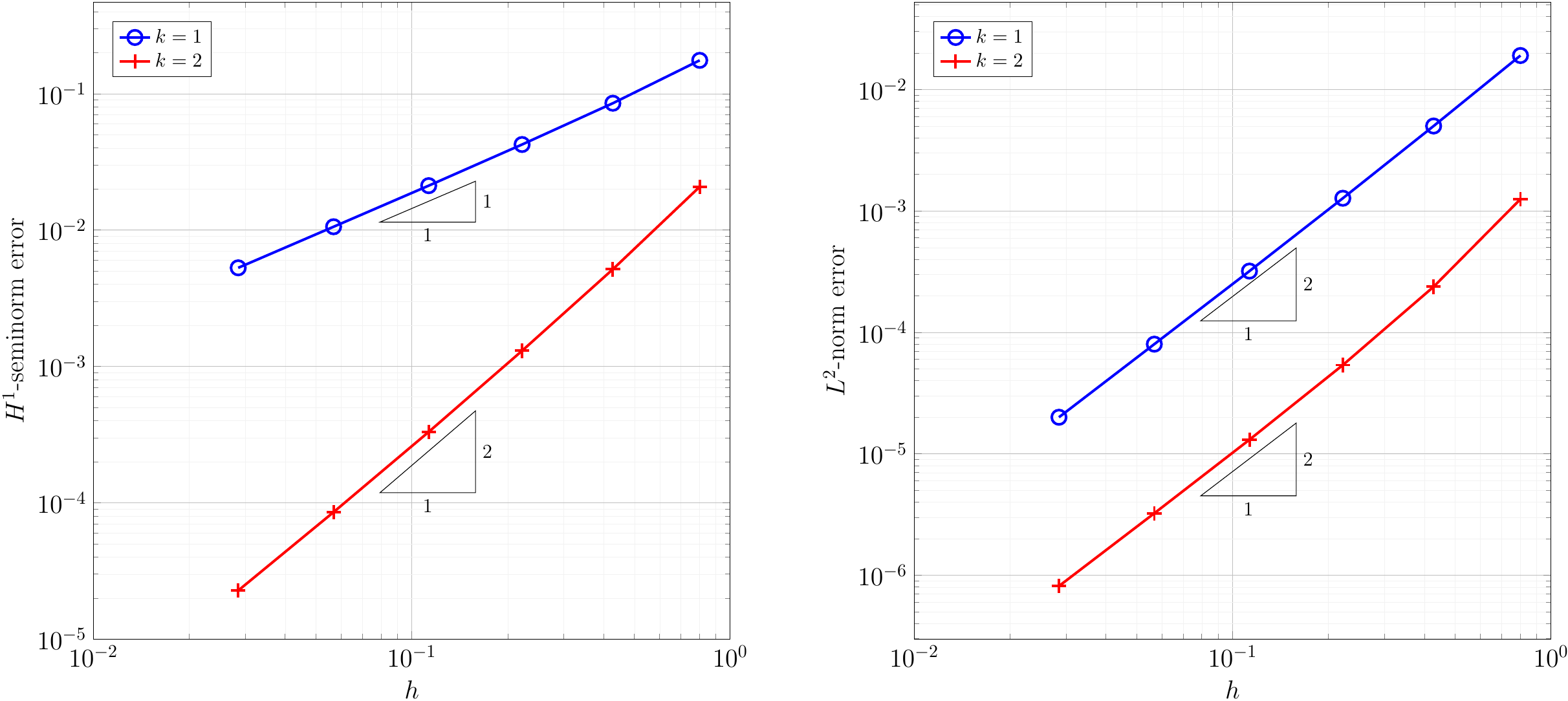}
	\caption{Example 1. $H^{1}$-seminorm (left plot) and $L^{2}$-norm (right plot) error for a sequence of ``straight'' meshes, for wave number $\kappa=1$}
\label{fig:circ_circ_error_VEM_kappa_1}	
\end{figure}

\subsection{Example 2. Computational domain with piece-wise linear boundaries}

\noindent
Let now $\Omega_{0}:=[-1,1]^{2}$ and $\Gamma$ the contour of the square $[-2,2]^{2}$ (see Figure \ref{fig:box_quad_meshes}). Since the domain of interest is a polygon, we can apply the standard (non curvilinear) VEM on polygonal meshes without introducing an approximation of the geometry. 
In Table \ref{tab:dofs_box_quad}, we report the total number of degrees of freedom of the VEM space, associated to each decomposition level of the computational domain, for $k=1,2$. In Figure \ref{fig:box_quad_meshes}, we plot the meshes corresponding to lev. 0 (left plot) and lev. 3 (right plot). We remark that the maximum level of refinement we have considered is lev. 7 for k = 1, whose number of degrees of freedom coincides with that of lev. 6 for k = 2.\\
\begin{figure}[H]
	\centering
	\includegraphics[width=0.70\textwidth]{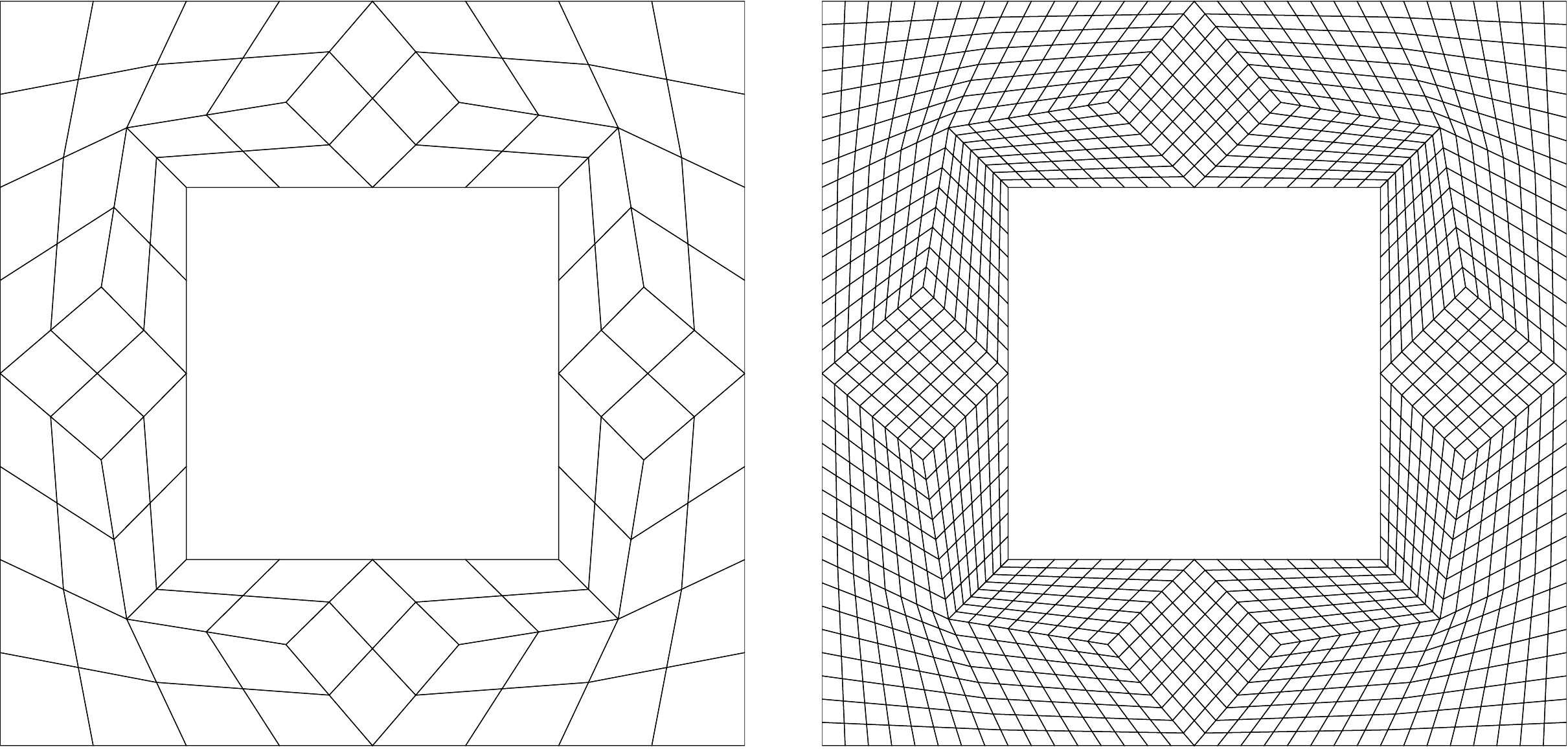}
	\caption{Example 2. Meshes of $\Omega$ for lev. 0 (left plot) and lev. 3 (right plot).}
	\label{fig:box_quad_meshes}
\end{figure}
\begin{table}[H]
 \centering
\begin{tabular}{lrrrrrrrr}
\toprule%
	& lev. 0 & lev. 1 & lev. 2 & lev. 3	& lev. 4 & lev. 5 & lev. 6 & lev. 7\\ 
\toprule%
$k=1$	& $120$ &    $432$ & $1,632$ &   $6,336$ & $24,960$ &   $99,072$  &   $394,752$ & $1,575,940$\\
$k=2$	& $432$ &  $1,632$ & $6,336$ & $24,960$ & $99,072$ & $394,752$ &  $1,575,940$ & $-$\\
\bottomrule
\end{tabular}
\caption{Example 2. Total number of degrees of freedom for $k=1,2$ and for different levels of refinement.}
\label{tab:dofs_box_quad}
\end{table}
\noindent
In Figures \ref{fig:quad_quad_kappa_1} and \ref{fig:quad_quad_kappa_10}, we show the real and imaginary parts of the numerical solution for the wave numbers $\kappa=1$ and $\kappa=10$ respectively, obtained by the quadratic approximation associated to the minimum refinement level for which the graphical behaviour is accurate and not wavy; this latter is lev. 3 for Figure \ref{fig:quad_quad_kappa_1} and lev. 5 for Figure \ref{fig:quad_quad_kappa_10}. 
\begin{figure}[H]
	\centering
	\includegraphics[width=1.00\textwidth]{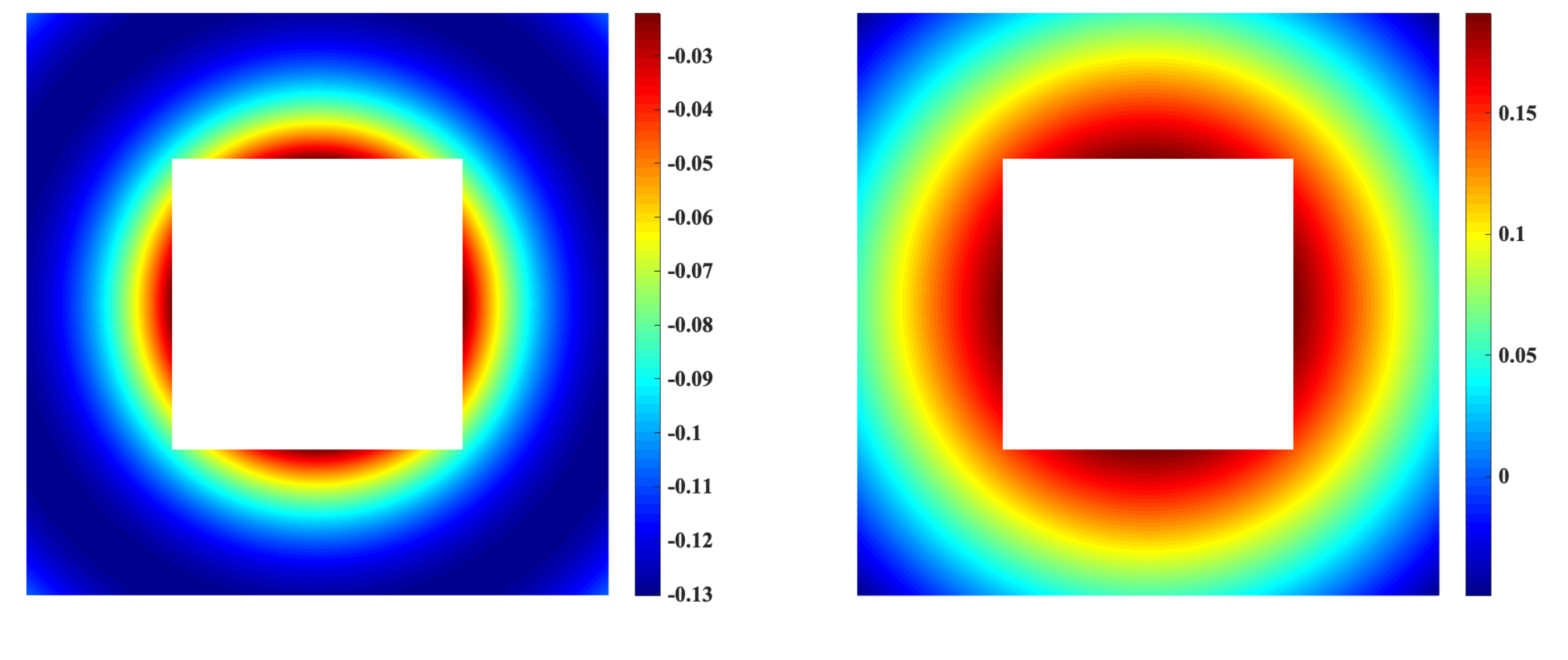}
	\caption{Example 1. Real (left plot) and imaginary (right plot) part of the numerical solution for $\kappa=1$, lev. 3 and $k=2$.}
	\label{fig:quad_quad_kappa_1}
\end{figure}

\begin{figure}[H]
	\centering
	\includegraphics[width=1.00\textwidth]{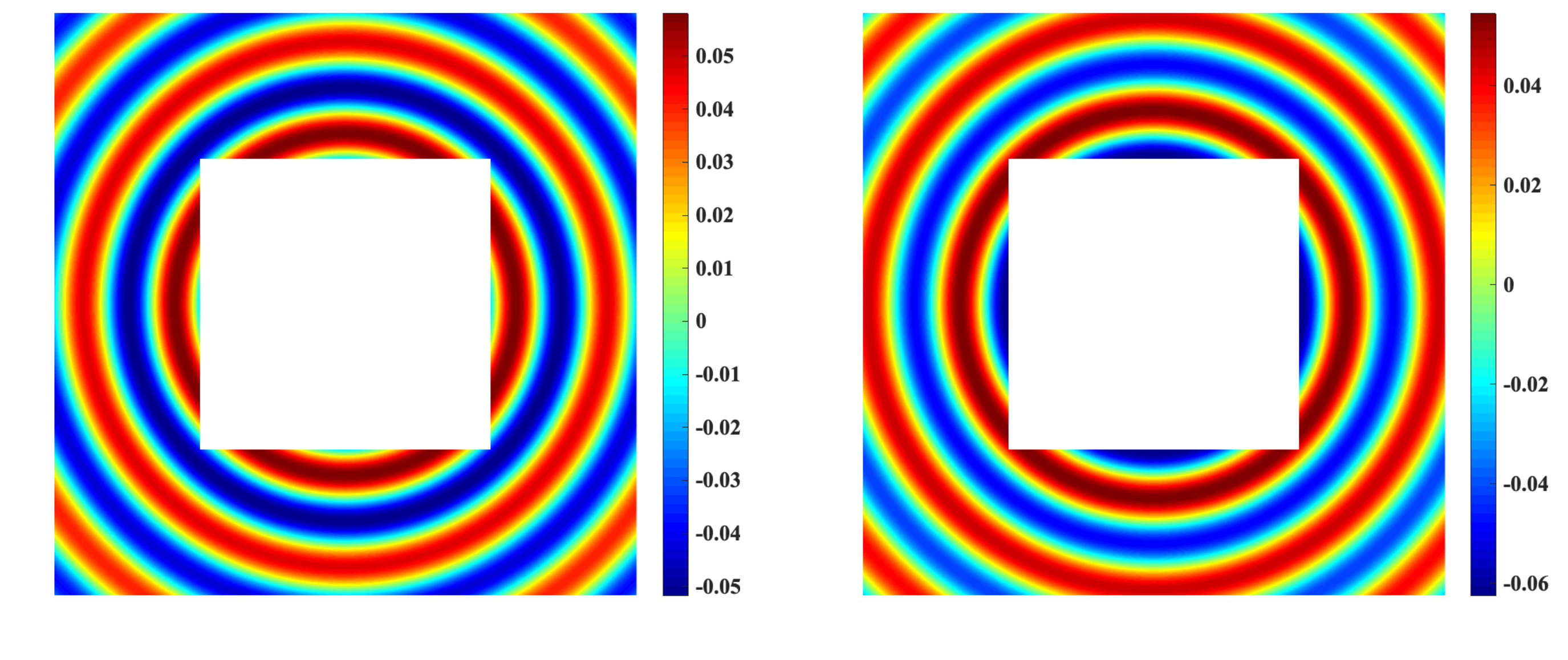}
	\caption{Example 1. Real (left plot) and imaginary (right plot) part of the numerical solution for $\kappa=10$, lev. 5 and $k=2$.}
	\label{fig:quad_quad_kappa_10}
\end{figure}
\noindent
As we can see from Tables \ref{tab:convergence_VEM_BEM_kappa_1} and \ref{tab:convergence_VEM_BEM_kappa_10}, the expected convergence order of the VEM-BEM approach for both $H^1$-seminorm and $L^2$-norm errors are confirmed, even if the assumption on the regularity of the artificial boundary $\Gamma$, required by the theory, is not satisfied. 
\begin{table}[H]
 \centering
  \def\sym#1{\ifmmode^{#1}\else\(^{#1}\)\fi}%
    \scalebox{0.94}{
\begin{tabular}{lc||cccc||cccc}
\toprule%
\multicolumn{2}{c}{} & \multicolumn{4}{c}{$L^{2}$-norm} & \multicolumn{4}{c}{$H^{1}$-seminorm}\\
\toprule%
lev.	& $h$ & $\varepsilon^{0,1}_{\text{lev}}$ & $\text{EOC}$ & $\varepsilon^{0,2}_{\text{lev}}$	& $\text{EOC}$ & $\varepsilon^{\nabla,1}_{\text{lev}}$ & $\text{EOC}$ & $\varepsilon^{\nabla,2}_{\text{lev}}$	& $\text{EOC}$\\ 	
\toprule%
$0$ & $7.60e-01$ & $1.71e-02$ & $$ & $8.34e-04$ & $$ & $1.57e-01$ & $$ & $1.66e-02$ & $$\\
$$  &  $$  &  $$  &  $2.0$  &  $$  &  $3.0$  &  $$  &  $1.1$  &  $$  &  $2.0$\\
$1$ & $3.85e-01$ & $4.37e-03$ & $$ & $1.01e-04$ & $$ & $7.57e-02$ & $$ & $4.07e-03$ & $$\\
$$  &  $$  &  $$  &  $2.0$  &  $$  &  $3.0$  &  $$  &  $1.0$  &  $$  &  $2.0$\\
$2$ & $1.94e-01$ & $1.10e-03$ & $$ & $1.26e-05$ & $$ & $3.78e-02$ & $$ & $1.02e-03$ & $$\\
$$  &  $$  &  $$  &  $2.0$  &  $$  &  $3.0$  &  $$  &  $1.0$  &  $$  &  $2.0$\\
$3$ & $9.73e-02$ & $2.74e-04$ & $$ & $1.57e-06$ & $$ & $1.89e-02$ & $$ & $2.56e-04$ & $$\\
$$  &  $$  &  $$  &  $2.0$  &  $$  &  $3.0$  &  $$  &  $1.0$  &  $$  &  $2.0$\\
$4$ & $4.87e-02$ & $6.86e-05$ & $$ & $1.96e-07$ & $$ & $9.46e-03$ & $$ & $6.40e-05$ & $$\\
$$  &  $$  &  $$  &  $2.0$  &  $$  &  $3.0$  &  $$  &  $1.0$  &  $$  &  $2.0$\\
$5$ & $2.44e-02$ & $1.71e-05$ & $$ & $2.46e-08$ & $$ & $4.73e-03$ & $$ & $1.60e-05$ & $$\\
$$  &  $$  &  $$  &  $2.0$  &  $$  &  $2.9$  &  $$  &  $1.0$  &  $$  &  $2.0$\\
$6$ & $1.22e-02$ & $4.29e-06$ & $$ & $3.35e-09$ & $$ & $2.36e-03$ & $$ & $4.03e-06$ & $$\\
$$  &  $$  &  $$  &  $2.0$  &  $$  &  $\times$  &  $$  &  $1.0$  &  $$  &  $\times$\\
$7$ & $6.10e-03$ & $1.07e-06$ & $$ & $\times$ & $$ & $1.18e-03$ & $$ & $\times$ & $$\\
\bottomrule
\end{tabular}
}
\caption{Example 2. $L^{2}$-norm and $H^{1}$-seminorm relative errors and corresponding EOC, for $\kappa=1$ and $k=1,2$.}
\label{tab:convergence_VEM_BEM_kappa_1}
\end{table}
\begin{table}[H]
 \centering
  \def\sym#1{\ifmmode^{#1}\else\(^{#1}\)\fi}%
    \scalebox{0.94}{
\begin{tabular}{lc||cccc||cccc}
\toprule%
\multicolumn{2}{c}{} & \multicolumn{4}{c}{$L^{2}$-norm} & \multicolumn{4}{c}{$H^{1}$-seminorm}\\
\toprule%
lev.	& $h$ & $\varepsilon^{0,1}_{\text{lev}}$ & $\text{EOC}$ & $\varepsilon^{0,2}_{\text{lev}}$	& $\text{EOC}$ & $\varepsilon^{\nabla,1}_{\text{lev}}$ & $\text{EOC}$ & $\varepsilon^{\nabla,2}_{\text{lev}}$	& $\text{EOC}$\\ 	
\toprule%
$0$ & $7.60e-01$ & $1.02e-00$ & $$ & $4.21e-01$ & $$ & $1.05e-00$ & $$ & $5.54e-01$ & $$\\
$$  &  $$  &  $$  &  $1.0$  &  $$  &  $3.7$  &  $$  &  $0.7$  &  $$  &  $1.7$\\
$1$ & $3.85e-01$ & $5.22e-01$ & $$ & $3.25e-02$ & $$ & $6.43e-01$ & $$ & $1.25e-01$ & $$\\
$$  &  $$  &  $$  &  $1.7$  &  $$  &  $3.1$  &  $$  &  $1.2$  &  $$  &  $2.0$\\
$2$ & $1.94e-01$ & $1.60e-01$ & $$ & $3.78e-03$ & $$ & $2.77e-01$ & $$ & $3.24e-02$ & $$\\
$$  &  $$  &  $$  &  $1.9$  &  $$  &  $3.1$  &  $$  &  $1.2$  &  $$  &  $2.0$\\
$3$ & $9.73e-02$ & $4.22e-02$ & $$ & $4.55e-04$ & $$ & $1.23e-01$ & $$ & $8.16e-03$ & $$\\
$$  &  $$  &  $$  &  $2.0$  &  $$  &  $3.0$  &  $$  &  $1.1$  &  $$  &  $2.0$\\
$4$ & $4.87e-02$ & $1.07e-02$ & $$ & $5.62e-05$ & $$ & $5.92e-02$ & $$ & $2.04e-03$ & $$\\
$$  &  $$  &  $$  &  $2.0$  &  $$  &  $3.0$  &  $$  &  $1.0$  &  $$  &  $2.0$\\
$5$ & $2.44e-02$ & $2.67e-03$ & $$ & $7.01e-06$ & $$ & $2.93e-02$ & $$ & $5.11e-04$ & $$\\
$$  &  $$  &  $$  &  $2.0$  &  $$  &  $3.0$  &  $$  &  $1.0$  &  $$  &  $2.0$\\
$6$ & $1.22e-02$ & $6.68e-04$ & $$ & $8.85e-07$ & $$ & $1.46e-02$ & $$ & $1.28e-04$ & $$\\
$$  &  $$  &  $$  &  $2.0$  &  $$  &  $\times$  &  $$  &  $1.0$  &  $$  &  $\times$\\
$7$ & $6.10e-03$ & $1.67e-04$ & $$ & $\times$ & $$ & $7.29e-03$ & $$ & $\times$ & $$\\
\bottomrule
\end{tabular}
}
\caption{Example 2. $L^{2}$-norm and $H^{1}$-seminorm relative errors and corresponding EOC, for $\kappa=10$ and $k=1,2$.}
\label{tab:convergence_VEM_BEM_kappa_10}
\end{table}

%


\section{Conclusions and perspectives}\label{sec_7_conclusions}
In this paper we have proposed a novel numerical approach for the solution of 2D Helmholtz problems defined in unbounded regions, external to bounded obstacles. It consists in reducing the unbounded domain to a finite computational one and in the coupling of the CVEM with the one equation BEM, by means of the Galerkin approach.
%
We have carried out the theoretical analysis of the method in a quite abstract framework, providing an optimal error estimate in the energy norm, under assumptions that can, in principle, include a variety of discretization spaces wider than those we have considered, for both the BIE and the interior PDE.
While the VEM/CVEM has been extensively and successfully applied to interior problems, its application to exterior problems is still at an early stage and, to the best of the authors' knowledge, the CVEM approach has been applied in this paper for the first time to solve exterior frequency-domain wave propagation problems in the Galerkin context. 
%
%

\

\indent
We remark that the above mentioned coupling has been proposed in a conforming approach context, so that the order of the CVEM and BEM approximation spaces have been chosen with the same polynomial degree of accuracy, and the grid used for the BEM discretization is the one inherited by the interior CVEM decomposition.
  

It is worth noting that it is possible in principle to decouple the CVEM and the BEM discretization, both in terms of degree of accuracy and of non-matching grids. 
This would lead to a non-conforming coupling approach by using, for example, a mortar type technique (see for instance \cite{bertoluzza_2019}). Such an approach would offer the further advantage of coupling different types of approximation spaces and of using fast techniques for the discretization of the BEM (see for example the very recent papers \cite{chaillat_2017, bertoluzza_2020, desiderio_2018, desiderio_2020}). This will be the subject of a future investigation.



\section*{Acknowledgments}
\noindent
This research benefits from the HPC (High Performance Computing) facility of the University of Parma, Italy.
 
 \section*{Funding}
\noindent
This work was performed as part of the GNCS-IDAM 2020 research program \emph{``Metodologie innovative per problemi di propagazione di onde in domini illimitati: aspetti teorici e computazionali''}. The second and the third author were partially supported by 
 MIUR grant \emph{``Dipartimenti di Eccellenza 2018-2022''}, CUP E11G18000350001.

\bibliographystyle{plain}
\bibliography{bibtex_num}{}

\begin{thebibliography}{10}

\bibitem{AhmadAlsaediBrezziMariniRusso2013}
B.~Ahmad, A.~Alsaedi, F.~Brezzi, L.~D. Marini, and A.~Russo.
\newblock Equivalent projectors for virtual element methods.
\newblock {\em Comput. Math. Appl.}, 66(3):376--391, 2013.

\bibitem{AIMI2021741}
A.~Aimi, L.~Desiderio, P.~Fedeli, and A.~Frangi.
\newblock A fast boundary-finite element approach for estimating anchor losses
  in micro-electro-mechanical system resonators.
\newblock {\em Applied Mathematical Modelling}, 97:741--753, 2021.

\bibitem{antonietti_2016}
P.F. Antonietti, L.~Beir\~{a}o~da Veiga, S.~Sacchi, and M.~Verani.
\newblock A {C}$^1$ virtual element method for the {C}ahn-{H}illiard equation
  with polygonal meshes.
\newblock {\em SIAM J. Numer. Anal.}, 54(1):34--56, 2016.

\bibitem{BeiraoBrezziCangianiManziniMariniRusso2013}
L.~Beir\~{a}o~da Veiga, F.~Brezzi, A.~Cangiani, G.~Manzini, L.~D. Marini, and
  A.~Russo.
\newblock Basic principles of virtual element methods.
\newblock {\em Math. Models Methods Appl. Sci.}, 23(1):199--214, 2013.

\bibitem{BeiraoBrezziMariniRusso2014}
L.~Beir\~{a}o~da Veiga, F.~Brezzi, L.~D. Marini, and A.~Russo.
\newblock The hitchhiker's guide to the virtual element method.
\newblock {\em Math. Models Methods Appl. Sci.}, 24(8):1541--1573, 2014.

\bibitem{beirao_2011}
L.~Beir\~{a}o~da Veiga, K.~Lipnikov, and G.~Manzini.
\newblock Arbitrary order nodal mimetic discretizations of elliptic problems on
  polygonal meshes.
\newblock {\em SIAM J. Numer. Anal.}, 49(5):1737--1760, 2011.

\bibitem{BeiraoRussoVacca2018}
L.~Beir\~{a}o~da Veiga, A.~Russo, and G.~Vacca.
\newblock The virtual element method with curved edges.
\newblock {\em arXiv:1711.04306 [math.NA]}, 2018.

\bibitem{BeiraoRussoVacca2019}
L.~Beir\~{a}o~da Veiga, A.~Russo, and G.~Vacca.
\newblock The virtual element method with curved edges.
\newblock {\em ESAIM Math. Model. Numer. Anal.}, 53(2):375--404, 2019.

\bibitem{benvenuti_2019}
E.~Benvenuti, A.~Chiozzi, G.~Manzini, and N.~Sukumar.
\newblock Extended virtual element method for the {L}aplace problem with
  singularities and discontinuities.
\newblock {\em Comput. Methods Appl. Mech. Engrg.}, 356:571--597, 2019.

\bibitem{berrone_2018}
S.~Berrone, A.~Borio, and G.~Manzini.
\newblock {SUPG} stabilization for the nonconforming virtual element method for
  advection-diffusion-reaction equations.
\newblock {\em Comput. Methods Appl. Mech. Engrg.}, 340:500--529, 2018.

\bibitem{bertoluzza_2019}
S.~Bertoluzza and S.~Falletta.
\newblock {FEM} solution of exterior elliptic problems with weakly enforced
  integral non reflecting boundary conditions.
\newblock {\em J. Sci. Comput.}, 81(2):1019--1049, 2019.

\bibitem{bertoluzza_2020}
S.~Bertoluzza, S.~Falletta, and L.~Scuderi.
\newblock Wavelets and convolution quadrature for the efficient solution of a
  2{D} space-time {BIE} for the wave equation.
\newblock {\em Appl. Math. Comput.}, 366:124726, 2020.

\bibitem{BrennerGuanSung2017}
S.C. Brenner, Q.~Guan, and L.Y. Sung.
\newblock Some estimates for virtual element methods.
\newblock {\em Comput. Methods Appl. Math.}, 17(4):553--574, 2017.

\bibitem{BrennerScott2008}
S.C. Brenner and L.R. Scott.
\newblock {\em The {M}athematical {T}heory of {F}inite {E}lement {M}ethods},
  volume~15 of {\em Texts in Applied Mathematics}.
\newblock Springer, New York, third edition, 2008.

\bibitem{brezzi_2009}
F.~Brezzi, A.~Buffa, and K.~Lipnikov.
\newblock Mimetic finite differences for elliptic problems.
\newblock {\em M2AN Math. Model. Numer. Anal.}, 43:277--295, 2009.

\bibitem{certik_2020}
O.~Certik, F.~Gardini, G.~Manzini, L.~Mascotto, and G.~Vacca.
\newblock The p- and hp-versions of the virtual element method for elliptic
  eigenvalue problems.
\newblock {\em Comput. Math. Appl.}, 79(7):2035--2056, 2020.

\bibitem{chaillat_2017}
S.~Chaillat, L.~Desiderio, and P.~Ciarlet.
\newblock Theory and implementation of $\mathcal{H}$-matrix based iterative and
  direct solvers for {H}elmholtz and elastodynamic oscillatory kernels.
\newblock {\em J. Comput. Phys.}, 351:165--186, 2017.

\bibitem{ColtonKress2019}
D.~Colton and R.~Kress.
\newblock {\em Inverse {A}coustic and {E}lectromagnetic {S}cattering {T}heory},
  volume~93 of {\em Applied Mathematical Sciences}.
\newblock Springer, Cham, 2019.

\bibitem{Costabel1987}
M.~Costabel.
\newblock Symmetric methods for the coupling of finite elements and boundary
  elements (invited contribution).
\newblock In {\em Boundary elements {IX}, {V}ol. 1 ({S}tuttgart, 1987)}, pages
  411--420. Comput. Mech., Southampton, 1987.

\bibitem{Costabel1988}
M.~Costabel.
\newblock Boundary integral operators on {L}ipschitz domains: elementary
  results.
\newblock {\em SIAM J. Math. Anal.}, 19(3):613--626, 1988.

\bibitem{desiderio_2018}
L.~Desiderio.
\newblock An $\mathcal{H}$-matrix based direct solver for the {B}oundary
  {E}lement {M}ethod in 3{D} elastodynamics.
\newblock {\em AIP Conf. Proc.}, 1978:120005$\_$1--120005$\_$4, 2018.

\bibitem{desiderio_2020}
L.~Desiderio and S.~Falletta.
\newblock Efficient solution of two-dimensional wave propagation problems by
  {CQ}-wavelet {BEM}: algorithm and applications.
\newblock {\em SIAM J. Sci. Comput.}, 42(4):B894--B920, 2020.

\bibitem{DesiderioFallettaScuderi2021}
L.~Desiderio, S.~Falletta, and L.~Scuderi.
\newblock A virtual element method coupled with a boundary integral non
  reflecting condition for 2{D} exterior {H}elmholtz problems.
\newblock {\em Comput. Math. Appl.}, 84:296--313, 2021.

\bibitem{FallettaMonegatoScuderi2014}
S.~Falletta, G.~Monegato, and L.~Scuderi.
\newblock A space-time {BIE} method for wave equation problems: the
  (two-dimensional) {N}eumann case.
\newblock {\em IMA J. Numer. Anal.}, 34(1):390--434, 2014.

\bibitem{GaticaMeddahi2020}
G.~N. Gatica and S.~Meddahi.
\newblock Coupling of virtual element and boundary element methods for the
  solution of acoustic scattering problems.
\newblock {\em J. Numer. Math.}, 28(4):223--245, 2020.

\bibitem{GeuzaineRemacle2009}
C.~Geuzaine and J.F. Remacle.
\newblock Gmsh: a three-dimensional finite element mesh generator with built-in
  pre- and post processing facilities.
\newblock {\em Internat. J. Numer. Methods Engrg.}, (79):1309--1331, 2009.

\bibitem{Han1990}
H.~D. Han.
\newblock A new class of variational formulations for the coupling of finite
  and boundary element methods.
\newblock {\em J. Comput. Math.}, 8(3):223--232, 1990.

\bibitem{HsiaoWendland2008}
G.~C. Hsiao and W.~L. Wendland.
\newblock {\em Boundary {I}ntegral {E}quations}, volume 164 of {\em Applied
  Mathematical Sciences}.
\newblock Springer-Verlag, Berlin, 2008.

\bibitem{JohnsonNedelec1980}
C.~Johnson and J.-C. N\'{e}d\'{e}lec.
\newblock On the coupling of boundary integral and finite element methods.
\newblock {\em Math. Comp.}, 35(152):1063--1079, 1980.

\bibitem{LionsMagenes1968}
J.L. Lions and E.~Magenes.
\newblock {\em Probl\`emes aux {L}imites non {H}omog\`enes et {A}pplications.
  {II}}.
\newblock Travaux et Recherches Math\'{e}matiques, No. 18. Dunod, Paris, 1968.

\bibitem{manzini_2017}
G.~Manzini, K.~Lipnikov, J.D. Moulton, and M.~Shashkov.
\newblock Convergence analysis of the mimetic finite difference method for
  elliptic problems with staggered discretizations of diffusion coefficients.
\newblock {\em SIAM J. Numer. Anal.}, 55(6):2956--2981, 2017.

\bibitem{MarquezMeddahiSelgas2003}
S.~M\'{a}rquez, A.~Meddahi and V.~Selgas.
\newblock Computing acoustic waves in an inhomogeneous medium of the plane by a
  coupling of spectral and finite elements.
\newblock {\em SIAM J. Numer. Anal.}, 41(5):1729--1750, 2003.

\bibitem{MonegatoScuderi1999}
G.~Monegato and L.~Scuderi.
\newblock Numerical integration of functions with boundary singularities.
\newblock {\em J. Comput. Appl. Math.}, 112(1-2):201--214, 1999.

\bibitem{QuarteroniValli1994}
A.~Quarteroni and A.~Valli.
\newblock {\em Numerical {A}pproximation of {P}artial {D}ifferential
  {E}quations}, volume~23 of {\em Springer Series in Computational
  Mathematics}.
\newblock Springer-Verlag, Berlin, 1994.

\bibitem{SaranenVainikko2002}
J.~Saranen and G.~Vainikko.
\newblock {\em Periodic {I}ntegral and {P}seudodifferential {E}quations with
  {N}umerical {A}pproximation}.
\newblock Springer Monographs in Mathematics. Springer-Verlag, Berlin, 2002.

\bibitem{SauterSchwab2011}
S.~A. Sauter and C.~Schwab.
\newblock {\em Boundary {E}lement {M}ethods}, volume~39 of {\em Springer Series
  in Computational Mathematics}.
\newblock Springer-Verlag, Berlin, 2011.

\bibitem{Sayas2009}
F.J. Sayas.
\newblock The validity of {J}ohnson-{N}\'{e}d\'{e}lec's {BEM}-{FEM} coupling on
  polygonal interfaces.
\newblock {\em SIAM J. Numer. Anal.}, 47(5):3451--3463, 2009.

\bibitem{SommarivaVianello2007}
A.~Sommariva and M.~Vianello.
\newblock Product {G}auss cubature over polygons based on {G}reen's integration
  formula.
\newblock {\em BIT}, 47(2):441--453, 2007.

\bibitem{SommarivaVianello2009}
A.~Sommariva and M.~Vianello.
\newblock Gauss-{G}reen cubature and moment computation over arbitrary
  geometries.
\newblock {\em J. Comput. Appl. Math.}, 231(2):886--896, 2009.

\bibitem{Steinbach2008}
O.~Steinbach.
\newblock {\em Numerical {A}pproximation {M}ethods for {E}lliptic {B}oundary
  {V}alue {P}roblems}.
\newblock Springer, New York, 2008.

\end{thebibliography}

\end{document}